\documentclass[fleqn,a4paper]{cas-sc}

\usepackage[T1]{fontenc}

\usepackage{amsmath,amssymb,amsfonts,stmaryrd}
\usepackage{graphicx,float} 
\usepackage{mathrsfs} 
\usepackage{color}
\usepackage{setspace} 
\usepackage{lipsum}
\usepackage{graphicx,float}
\usepackage{color}
\usepackage{epstopdf}

\usepackage{multirow}

\usepackage{url}
\usepackage{diagbox}
\def\ds{\displaystyle}

\def\O{\Omega}

\def\g{\gamma}

\def\t{\theta}

\newcommand{\set}[1]{\lbrace #1 \rbrace}

\newcommand{\jump}[1]{\llbracket #1 \rrbracket}

\newcommand{\jumpp}[1]{\left\llbracket #1 \right\rrbracket}

\newcommand{\norm}[1]{\lVert#1\rVert}

\newcommand{\n}{\boldsymbol{n}}
\usepackage{booktabs}
\usepackage{float}

\newcommand{\bu}{\boldsymbol{u}}
\newcommand\bv{\boldsymbol{v}}
\newcommand\bw{\boldsymbol{w}}

\newcommand\bn{\boldsymbol{n}}

\newcommand\curl{\mathop{\mathbf{curl}}\nolimits}

\newcommand\bF{\boldsymbol{f}}

\newcommand\bW{\boldsymbol{W}}
\newcommand\bxi{\boldsymbol{\xi}}
\newcommand\bchi{\boldsymbol{\chi}}

\newcommand\cT{\mathcal{T}}
\def\CT{{\mathcal T}}

\newcommand\bsig{\boldsymbol{\sigma}}
\newcommand\btau{\boldsymbol{\tau}}

\newcommand\R{\mathbb{R}}

\renewcommand\H{\mathrm{H}}

\renewcommand\L{\mathrm{L}}

\renewcommand\O{\Omega}
\newcommand\DO{\partial\O}

\newcommand\bdiv{\mathop{\mathbf{div}}\nolimits}

\renewcommand\div{\mathop{\mathrm{div}}\nolimits}
\newcommand\rot{\mathop{\mathrm{rot}}\nolimits}

\newcommand\tr{\mathop{\mathrm{tr}}\nolimits}

\renewcommand\t{\mathtt{t}}

\newcommand\LO{\L^2(\O)}

\newcommand\HsO{\H^s(\O)}

\renewcommand\t{\mathtt{t}}

\newcommand\qon{\qquad\hbox{on }}
\newcommand{\vertiii}[1]{{\left\vert\kern-0.25ex\left\vert\kern-0.25ex\left\vert #1 
    \right\vert\kern-0.25ex\right\vert\kern-0.25ex\right\vert}}

\numberwithin{equation}{section}
\numberwithin{figure}{section}
\newtheorem{remark}{Remark}[section]
\newtheorem{theorem}{Theorem}[section]
\newtheorem{lemma}[theorem]{Lemma}

\newtheorem{corollary}[theorem]{Corollary}
\newenvironment{proof}{\noindent{\it Proof.}}{\hfill$\square$}
\allowdisplaybreaks

\usepackage{amsopn}

\ifpdf
\hypersetup{
  pdftitle={Augmented mixed FEM for nonlinear poroelasticity},
  pdfauthor={F. Lepe, G. Rivera, R. Ruiz-Baier and S. Villa-Fuentes}
}
\fi

\newcommand{\bepsilon}{\mbox{\boldmath{$\varepsilon$}}}
\def\CT{{\mathcal T}}

\def\rD{\mathrm{D}}
\def\rN{\mathrm{N}}
\newcommand{\bzero}{\boldsymbol{0}}
\def\bdiv{\mathbf{div}}
\def\qan{{\quad\hbox{and}\quad}}

\def\qon{{\quad\hbox{on}\quad}}
\def\wt{\widetilde}
\def\rH{\mathrm{H}}

\newcommand{\bgamma}{{\boldsymbol{\gamma}}}
\newcommand{\bbeta}{{\boldsymbol\eta}}
\newcommand{\bH}{\mathbf{H}}

\newcommand{\wh}{\widehat}
\renewcommand{\wt}{\widetilde}
\newcommand{\cJ}{\mathcal{J}}
\newcommand{\bbL}{\boldsymbol{\mathcal{L}}}
\newcommand{\bbH}{\boldsymbol{\mathcal{H}}}

\newcommand{\bL}{\mathbf{L}}
\newcommand{\bX}{\mathbf{X}}
\newcommand{\rL}{\mathrm{L}}

\begin{document}
\shortauthors{F. Lepe,  G. Rivera, R. Ruiz-Baier and S. Villa-Fuentes}
\shorttitle{Augmented mixed FEM for nonlinear poroelasticity}
\title[mode=title]{An augmented  mixed finite element method for the Biot problem with nonlinear permeability}

\author[1]{Felipe Lepe}[orcid=0000-0002-7929-9572]
\ead{FLepe@ubiobio.cl}

\author[2]{Gonzalo Rivera}[orcid=0000-0002-6449-6506]
\ead{Gonzalo.Rivera@ulagos.cl}

\author[3]{Ricardo Ruiz-Baier}[orcid=0000-0003-3144-5822]
\ead{Ricardo.RuizBaier@monash.edu}

\author[2]{Segundo Villa-Fuentes}[orcid=0000-0002-0377-6555]
\ead{Segundo.VillaFuentes@ulagos.cl}

\affiliation[1]{organization={GIMNAP-Departamento de Matem\'atica, Universidad del B\'io-B\'io},    
addressline={Casilla 5-C},  
city={Concepci\'on}, country={Chile}}

\affiliation[2]{organization={Departamento de Ciencias Exactas,
	Universidad de Los Lagos},    
addressline={Casilla 933},  
city={Osorno}, country={Chile}}

\affiliation[3]{organization={School of Mathematics, Monash University},     addressline={9 Rainforest Walk}, postcode={3800},  city={Melbourne}, state={Victoria}, country={Australia}}


\begin{abstract}
In two and three dimensions, we analyze a mixed formulation of the nonlinear Biot  equations in terms of the poroelastic stress tensor, the displacement, the pressure, and the rotation tensor. To establish the well-posedness of the problem, we introduce suitable stabilization terms, leading to an augmented mixed formulation in which the associated parameters are chosen appropriately. Owing to the nonlinearity of the problem, we employ a fixed-point strategy to prove the existence and uniqueness of solutions. We propose a  finite element scheme for which we  establish existence and uniqueness of a discrete solution, together with a priori  error estimates. Additionally, we design and analyze an a posteriori error estimator of the residual type which results to be reliable and efficient. Finally, we present a series of numerical experiments in two and three dimensions in order  to assess the performance of the proposed method.
\end{abstract}

\begin{keywords}
Mixed finite element formulation \sep  Nonlinear poroelasticity  \sep  A priori and a posteriori analysis
\MSC[2020]  35M30 \sep 65N12 \sep  65N15  \sep 65N30 \sep 74S99
\end{keywords}
 
\maketitle


\section{Introduction}\label{sec:intro}
 \subsection{Scope}
Nonlinear poroelasticity models describe the coupled interaction between the flow of an interstitial fluid and the mechanical deformation of a fully saturated porous solid skeleton. This coupling is essential to study several models related in geomechanics and in biomechanics, such as the   filtration of aqueous humor through cartilage-like structures in the eye (with direct application to the study of glaucoma formation), the mechanical response of cartilage and trabecular meshwork, or subsurface deformation in reservoir and geotechnical engineering, among others \cite{coussy04,borregales2021iterative}. An important aspect of  these applications is that  the permeability of the medium is not necessarily fixed, implying that the permeability coefficient may change depending on the material involved, typically through the solid displacement (or the associated dilation) and/or the pore fluid pressure. This dependence of the permeability on the physical characteristics of the media, introduces {naturally} a   nonlinear coupling between the momentum and mass balance equations, {leading to the use of} monotone-operator arguments {for the analysis} that are otherwise standard for the linear Biot consolidation problem. The mathematical analysis of this class of nonlinear poroelasticity models (existence, uniqueness, and regularity of solutions) has been addressed using a variety of tools, including   Galerkin approximations combined with Brouwer fixed-point arguments and compactness, the theory of monotone operators, semigroup theory, and doubly nonlinear evolution equations; see, e.g., \cite{bociu2016analysis,bociu2022weak,bociu2021nonlinear,cao2013analysis,gaspar2016numerical,kraus2024fixed,showalter2001partially,tavakoli2013existence,van2023mathematical}. In our case, we focus the analysis via an augmentation technique on the formulation.

The augmentation strategy consists in adding to the variational formulation  suitable residual-type terms (built from the constitutive and equilibrium equations and weighted by stabilisation parameters that are fixed a priori in terms of the physical data) in order to obtain a coercive bilinear form  on the whole space associated with the stress, displacement, and pressure unknowns, without requiring this triple of finite element subspaces to satisfy an inf-sup condition among themselves. This idea, introduced for linear
elasticity and Stokes-type problems in \cite{gatica06elast} {and applied to other type of problems such as flow transport and poroelasticity  (see for instance \cite{agr2015,MR4599323} and the references therein)}, has the practical advantage of allowing simple {and standard} conforming finite element families, such as Brezzi--Douglas--Marini or Raviart--Thomas elements for the stress combined with continuous Lagrange elements for the displacement and the pressure, while also permitting a strong (rather than only weak) imposition of essential displacement boundary conditions, which is
beneficial in benchmark and contact-type problems where an accurately resolved displacement trace is required. The weak imposition of stress symmetry, through the rotation tensor, is not absorbed by the augmentation and still yields a saddle-point structure; its well-posedness follows from the classical inf-sup condition for this pairing, which is already known to hold for the PEERS and Arnold--Falk--Winther families  \cite{arnold-2007}. Augmented mixed formulations of this type have since been extended to a variety of multiphysics and coupled problems; see, e.g., \cite{agr2015,barrios22,barrios25,barrios20,gatica2013priori,MR4599323} for representative applications including stress-assisted diffusion and viscous flow-transport couplings, and this is the approach we pursue herein. 
		
In the present paper we focus on the classical form of the Biot poroelasticity system written in terms of the solid displacement $\bu$ and the pore fluid pressure $p$, where the intrinsic permeability $\kappa(\bu,p)$ is allowed to depend nonlinearly on both the displacement and the pressure. Beyond the primal displacement-pressure formulation, several works have proposed mixed formulations in which the poroelastic (total) stress $\bsig$ is introduced as an additional
unknown, following a Hellinger--Reissner-type principle; see, e.g.,
\cite{ambartsumyan2020coupled,MR3667080,elyes18,lee16,MR3200272} for the linear case, and \cite{ambartsumyan2019nonlinear,caucao2022multipoint,li20,MR4599323} for extensions to nonlinear permeability and to poroelasticity/free-fluid couplings. This route was exploited by the authors in two companion contributions. In \cite{lamichhane24} we proposed four- and five-field Hu--Washizu-type mixed formulations, in which the permeability depends on a linear combination of the fluid pressure and the dilation and the infinitesimal strain tensor is retained as an explicit unknown (with either strong or weak imposition of stress symmetry); there the resulting continuous and discrete problems have the structure of a \emph{twofold} saddle-point
problem, and their unique solvability is obtained \emph{directly} from the twofold saddle-point variant of the Babu\v{s}ka--Brezzi theory, combined with a fixed-point argument that handles the whole coupled system at once. In \cite{KLBRV2026} we instead rewrote the permeability constitutive law in terms of the total poroelastic stress and the fluid pressure, which allowed us to recover a Hellinger--Reissner-type formulation (without solving explicitly for the strain) and to derive an a posteriori error analysis for the resulting scheme; there the coupled system is structured as a saddle-point formulation perturbed by a linearised saddle-point block together with two off-diagonal perturbations, and its analysis proceeds by \emph{decoupling} the linearised problem into two separate saddle-point sub-problems, one for weakly symmetric elasticity and one for mixed reaction-diffusion, whose individual well-posedness follows from classical results, and which are then reconnected, together with the outer nonlinearity, through a fixed-point argument.

In this work we consider instead a permeability law $\kappa(\bu,p)$ that depends explicitly on the solid displacement $\bu$ (rather than on the stress or the strain), as is common in constitutive models for anisotropic and deformation-dependent hydraulic conductivity \cite{AteshianWeiss2010}. A direct consequence of keeping $\bu$, instead of $\bsig$ or $\bepsilon(\bu)$, as the argument of the permeability is that the displacement must be approximated in a space for which the permeability evaluation, and in particular its Lipschitz continuity with respect to $\bu$, is meaningful and easy to control; the natural choice is then the classical energy space $\bH^1(\Omega)$, rather than a discontinuous displacement space as is typically used in dual-mixed elasticity formulations. Enforcing $\bH^1$-conformity for $\bu$ together with the Dirichlet boundary condition, while retaining $\bsig$ as an $H(\mathrm{div})$-conforming unknown, is  the setting in which \emph{augmented} mixed formulations are useful.

Other advantages of the augmentation strategy adopted here is that it lets us approximate the displacement in its natural space $\bH^1(\Omega)$ and evaluate the nonlinear permeability directly at $\bu$ (rather than at the total stress, as in \cite{KLBRV2026}, or at an explicitly reconstructed strain, as in \cite{lamichhane24}), which is the modeling choice motivated in Section~\ref{sec:model}. Secondly, the augmented formulation avoids   the twofold saddle-point structure of \cite{lamichhane24} and the decoupling into elasticity and diffusion sub-problems required in \cite{KLBRV2026}. Here, after linearising the argument of the permeability, the problem reduces to a single saddle-point problem for the stress-displacement-pressure-rotation quadruple, whose well-posedness follows directly from the classical Babu\v{s}ka--Brezzi theory. The nonlinearity inherited from $\kappa(\bu,p)$ is still handled, as in our earlier analyses, by a fixed-point argument: freezing the arguments of the permeability defines a fixed-point operator whose well-definedness, self-mapping property, and Lipschitz continuity are established using this (single) saddle-point theory, and the Banach
fixed-point theorem then yields existence and uniqueness of the solution under a smallness assumption on the data.

\subsection{Outline} The organization of the paper is as follows. In Section~\ref{sec:model}
{we introduce and describe the model problem  together with  the} associated augmented variational formulation for the continuous problem. {Section~\ref{sec:existence_and_uniq} is focused on the analysis of  the} continuous problem, establishing its well-posedness by combining a fixed-point argument with the Babu\v{s}ka--Brezzi theory. {In Section~\ref{sec:fem} we introduce the numerical scheme  where we introduce the finite
element subspaces, the discrete variational problem and its well-posedness. Additionally} we derive a C\'ea estimate, and  the corresponding orders of convergence from a priori error analysis. Section~\ref{sec:a-posteriori} focuses on the design {and analysis} of a residual-based a posteriori error estimator, together with the proofs of its reliability and efficiency. {In Section~\ref{sec:numerics} we report a series of} numerical examples that validate and illustrate the theoretical {results from previous sections} {and finally}  in Section~\ref{sec:concl} {we present a summary of our results and a discussion of possible extensions.}

\subsection{Notations and preliminaries}

{Let us denote the spatial dimension by $d$}. Given
any Hilbert space $X$, let $\bX$ and $\boldsymbol{\mathcal{X}}$ denote, respectively,
the space of vectors and tensors  with
entries in $X$. In particular, $\mathbf{I}$ is the identity matrix of
$\R^{d\times d}$, and $\mathbf{0}$ denotes a generic null vector or tensor. 
Given $\btau:=(\tau_{ij})$ and $\bsig:=(\sigma_{ij})\in\R^{d\times d}$, 
we define, as usual, the transpose tensor $\btau^{\t}:=(\tau_{ji})$, 
the trace $\ds \tr(\btau):=\sum_{i=1}^d\tau_{ii}$ and the tensor inner product $\ds \btau:\bsig:=\sum_{i,j=1}^d\tau_{ij}\sigma_{ij}$. 

Let $\O$ be a polygonal Lipschitz bounded domain of $\R^d$ with
boundary $\DO$. For $s\geq 0$, $\norm{\cdot}_{s,\O}$ stands indistinctly
for the norm of the Hilbertian Sobolev spaces $\HsO$, $\boldsymbol{\H}^s(\O)$ or
$\boldsymbol{\mathcal{H}}^s(\O)$ for scalar, vectorial and tensorial fields, respectively, with the convention $\H^0(\O):=\LO$, $\boldsymbol{\H}^0(\O)=\boldsymbol{\L}^2(\O)$ and $\boldsymbol{\mathcal{H}}^0(\O):=\boldsymbol{\mathcal{L}}^2(\O)$. We also define the Hilbert space 
$\boldsymbol{\mathcal{H}}(\bdiv,\O):=\set{\btau\in\boldsymbol{\mathcal{L}}^2(\O):\ \bdiv\btau\in\boldsymbol{\L}^2(\O)}$, whose norm
is given by $\norm{\btau}^2_{\bdiv,\O}
:=\norm{\btau}_{0,\O}^2+\norm{\bdiv\btau}^2_{0,\O}$.

\section{The model problem}
\label{sec:model}
Let $\O\subset\mathbb{R}^d$, with $d\in\{2,3\}$, be an open and  bounded domain with Lipschitz boundary $\partial\Omega=\Gamma$ which is partitioned into disjoint sub-boundaries $\Gamma:= \overline{\Gamma_\rD} \cup \overline{\Gamma_\rN}$, and it is assumed for the sake of simplicity that both sub-boundaries are non-empty $|\Gamma_\rD|\cdot|\Gamma_\rN|>0$. The poroelasticity equations are governed by the following system of partial differential equations
\begin{subequations}\begin{align}
-\bdiv\bsig&=\bF
\quad\text{in }\O,
\label{eq:momentum}\\
\bsig-\mathcal{C}\varepsilon(\bu)+\alpha p\mathbf{I}&=\boldsymbol{0}
\quad\text{in }\O,
\label{eq:constitutive}\\
c_0 p+\alpha\tr(\varepsilon(\bu))
-\div(\kappa(\bu,p)\nabla p)
&=g
\quad\text{in }\O,\label{eq:mass}
\end{align}\end{subequations}
where $\bsig$ is the symmetric poroelastic  tensor related to the balance of angular momentum, $\bu$ is the displacement, $\varepsilon(\bu):=\frac{1}{2}(\nabla\bu+(\nabla\bu)^{\texttt{t}})$ is the infinitesimal strain,  $\mathcal{C}$ represents the invertible  Hooke's operator, symmetric and positive definite, which is defined by 
\begin{equation*}
\mathcal{C}\btau:=\lambda\tr(\btau)\mathbf{I}+2\mu\btau,
\end{equation*}
where $\btau\in\mathbb{R}^{d\times d}$ is an arbitrary  tensor, whereas $\mu$ and $\lambda$ correspond to the Lam\'e constants given by
\begin{equation*}
\lambda:=\frac{E\nu}{(1+\nu)(1-2\nu)}\quad\text{and}\quad\mu:=\frac{E}{2(1+\nu)},
\end{equation*}
where $E$ is the Young's modulus and $\nu$ is the Poisson ratio. It is clear that when the Poisson ratio tends to $1/2$, the Lam\'e constant $\lambda$ blows up. On the model, $\kappa(\cdot,\cdot)$ represents  the intrinsic permeability (divided by the fluid viscosity) of the laminar flow in the medium, whereas
$p$ represents the pressure, $\mathbf{I}\in\mathbb{R}^{d\times d}$ is the identity tensor, $\alpha$ is the so-called Biot-Willis parameter which is  chosen in the range $\alpha\in [0,1]$, $\bF$ and $g$ are external forces  representing  a prescribed body force per unit of volume and a net volumetric fluid production rate, respectively.

Let us now obtain a variational formulation. First, since $\mathcal{C}^{-1}$ is well defined, {the following identity holds}
\begin{equation}\label{eq:Hooke-inv}
\mathcal C^{-1}\btau
=
\dfrac1{2\mu}\,\btau
-
\dfrac{\lambda}{2\mu(d\lambda+2\mu)}\,\tr(\btau)\mathbf I=
\frac1{2\mu}\,\btau^{{\texttt{d}}}
+
\frac1{d\lambda+2\mu}\,\tr(\btau)\mathbf I,
\end{equation}
{where $\btau^{\texttt{d}}$ represents the deviatoric tensor associated to $\btau$, which is defined by $\btau^{\texttt{d}}:=\btau-d^{-1}\tr(\btau){\mathbf{I}}$.} Now,  introducing the rotation tensor $\boldsymbol{\gamma}:=\displaystyle\frac{1}{2}(\nabla\bu-(\nabla\bu)^{\texttt{t}})$, we  rewrite system \eqref{eq:momentum}-\eqref{eq:mass} as follows
\begin{subequations}\begin{align}
-\bdiv\bsig&=\bF
\quad\text{in }\O,
\label{eq:momentum-1}\\
\mathcal{C}^{-1}(\bsig+\alpha p\mathbf{I})-\nabla\bu+\boldsymbol{\gamma}&=\boldsymbol{0}
\quad\text{in }\O,
\label{eq:constitutive-1}\\
\Big(c_0+\frac{d\alpha^2}{d\lambda+2\mu}\Big) p+\frac{\alpha}{d\lambda+2\mu}\tr(\bsig)
-\div(\kappa(\bu,p)\nabla p)
&=g
\quad\text{in }\O,\label{eq:mass-1}
\end{align}\end{subequations}

Finally, we consider for the problem above the following boundary conditions, for prescribed $z_\Gamma \in \rH^{1/2}(\Gamma)$: 
\begin{equation}\label{eq:bc}
\bu = \bzero \qon \Gamma_\rD, \qquad 
\bsig \bn = \bzero \qon \Gamma_\rN, \qan \kappa(\bu,p)\nabla p \cdot \bn= z_\Gamma  \qon \Gamma.
\end{equation}
Now, taking advantage of the definition of $\mathcal{C}^{-1}$ given in \eqref{eq:Hooke-inv}, multiplying \eqref{eq:constitutive-1} against $\btau\in\bbH_{\Gamma_\rN}(\bdiv,\O):=\{\btau\in\bbH(\bdiv,\O)\,:\, \btau\boldsymbol{n}=\boldsymbol{0}\,\,\text{on}\,\Gamma_N\}$, and \eqref{eq:mass-1} against $q\in\rH^1(\Omega)$, integrating by parts, using the boundary condition \eqref{eq:bc}, and impose the symmetry of $\bsig$ weakly, we have
\begin{align}\label{eq:weak-1}
\int_{\O}\mathcal{C}^{-1}\bsig:\btau
+ \frac{\alpha}{d\lambda+2\mu}\int_{\O}p\,\tr(\btau)
+ \int_{\O}\bu\cdot\bdiv\btau
+ \int_{\O}\btau:\bgamma &= 0, \nonumber\\
\Big(c_0+\frac{d\alpha^2}{d\lambda+2\mu}\Big)\int_{\O} p q
+ \int_{\O}\kappa(\bu,p)\nabla p\cdot \nabla q
+ \frac{\alpha}{d\lambda+2\mu}\int_{\O}q\,\tr(\bsig)
&= \int_{\O}gq + \langle z_\Gamma, q\rangle_{\Gamma},\\
\int_{\O}\bsig:\boldsymbol{\eta} &= 0,\nonumber
\end{align}
for all $\btau\in\boldsymbol{\mathcal{H}}_{\Gamma_\rN}(\bdiv,\O)$, $\boldsymbol{\eta}\in\bbL^2_{\text{skew}}(\O)$, $q\in \rH^1(\O)$.

On the other hand, the equilibrium equation \eqref{eq:momentum-1} is imposed weakly by testing it against a suitable function $\bv$, yielding
\begin{equation*}\label{eq:weak-momentum}
\int_{\O}\bv\cdot\bdiv\bsig = -\int_{\O}\bF\cdot\bv,
\end{equation*}

Consequently, to ensure the well-posedness of the resulting variational formulation, we augment the formulation with the following residual terms arising from the constitutive and momentum equations, \eqref{eq:constitutive} and \eqref{eq:momentum} respectively.
\begin{equation*}\label{eq:augmented}
\begin{aligned}
\delta_1\int_\Omega (\mathcal{C}^{-1}(\bsig + \alpha p\mathbf{I}) - \varepsilon(\bu)):(\mathcal{C}^{-1}(\btau + \alpha q\mathbf{I}) - \varepsilon(\bv))&=0,\\
\delta_2\int_{\O}\bdiv\btau\cdot\bdiv\bsig &= -\delta_2\int_{\O}\bF\cdot\bdiv\btau,
\end{aligned}
\end{equation*}
for all $\btau\in\boldsymbol{\mathcal{H}}_{\Gamma_\rN}(\bdiv,\O)$, $q\in\rH^1(\O)$ and $\bv\in\bH^1_{\Gamma_\rD}(\Omega):=\{\bv\in\boldsymbol{\H}^1(\O)\,:\, \bv=\boldsymbol{0}\,\,\text{on}\,\,\Gamma_\rD\}$, where $\delta_1$ and $\delta_2$ are positive parameters to be specified later.

Then, we define $\bH:=\bbH_{\Gamma_\rN}(\bdiv,\O) \times \bH^1_{\Gamma_\rD}(\Omega) \times \rH^1(\O)$ and $\mathbf{Q}:=\bbL^2_{\text{skew}}(\O)$. We also define the norm associated with $\bH$ as $\|(\btau,\bv,q)\|_{\bH}:=\|\btau\|_{\bdiv,\O}+\|\bv\|_{1,\O}+\|q\|_{1,\O}$.

In addition, for a given $(\wh{\bu},\wh{p})\in \bH^1_{\Gamma_\rD}(\Omega) \times \rH^1(\O)$, we define the form $a_{\wh{\bu},\wh p}:\bH\times\bH\to\R$ as follows 
\begin{align}\label{eq:form-a}
a_{\wh{\bu},\wh p}((\bsig,\bu,p),(\btau,\bv,q))&:= \int_{\O}\mathcal{C}^{-1}\bsig:\btau
+ \frac{\alpha}{d\lambda+2\mu}\int_{\O}p\,\tr(\btau)
+ \int_{\O}\bu\cdot\bdiv\btau + \Big(c_0+\frac{d\alpha^2}{d\lambda+2\mu}\Big)\int_{\O} p q \nonumber \\
&\quad 
+ \int_{\O}\kappa(\wh{\bu},\wh p)\nabla p\cdot \nabla q
+ \frac{\alpha}{d\lambda+2\mu}\int_{\O}q\,\tr(\bsig) + \int_{\O}\bv\cdot\bdiv\bsig \nonumber \\
& \quad + \delta_1\int_\Omega (\mathcal{C}^{-1}(\bsig + \alpha p\mathbf{I}) - \varepsilon(\bu)):(\mathcal{C}^{-1}(\btau + \alpha q\mathbf{I}) - \varepsilon(\bv)) + \delta_2\int_{\O}\bdiv\btau\cdot\bdiv\bsig,
\end{align}
whereas  the bilinear form $b:\bH\times\mathbf{Q}\to\R$ and the functional $F\in\bH'$ are respectively defined by
\begin{equation}\label{eq:form-b}
b((\btau,\bv,q),\bbeta):= \int_{\O}\btau:\bbeta,
\end{equation}
and
\begin{equation}\label{eq:functional-F}
F(\btau,\bv,q):=  -\int_{\O}\bF\cdot\bv -\delta_2\int_{\O}\bF\cdot\bdiv\btau + \int_{\O}gq + \langle z_{\Gamma}, q\rangle_{\Gamma}.
\end{equation}

Then, from \eqref{eq:weak-1}-\eqref{eq:functional-F}, 
we obtain the variational problem: Find $((\bsig,\bu,p),\bgamma)\in \bH \times \mathbf{Q}$, such that 
\begin{equation}\label{eq:weak-problem}
\begin{aligned}
a_{\bu,p}((\bsig,\bu,p),(\btau,\bv,q)) + b((\btau,\bv,q),\bgamma)
&= F(\btau,\bv,q)
\quad\forall (\btau,\bv,q)\in\bH,\\
b((\bsig,\bu,p),\bbeta)
&=0
\quad\forall \bbeta\in \mathbf{Q}.
\end{aligned}
\end{equation}

\section{Existence and uniqueness of weak solution}
\label{sec:existence_and_uniq}
Now the objective is to establish that our problem has a solution and its uniqueness. To do this task, we employ a fixed point strategy through an operator that
must be well defined, maps a ball into the same ball and is contractive. 

\subsection{Definition of a fixed-point operator}

Before introducing the operator for which we will analyze the aforementioned properties, let us define  some preliminary ingredients.  For a given $r>0$, we introduce the following set
\begin{equation}\label{eq:set-W}
\bW := \Big\{ (\wh{\bu},\wh p)\in\boldsymbol{\H}^1_{\Gamma_\rD}(\Omega)\times\rH^1(\Omega) \,:\quad \|(\wh{\bu},\wh p)\|:=\|\wh{\bu}\|_{1,\Omega} + \|\wh p\|_{1,\Omega} \leq r \Big\},
\end{equation}
which is a closed ball of $\boldsymbol{\H}^1_{\Gamma_\rD}(\Omega)\times\rH^1(\Omega)$, centered at the origin and radius $r$. With this set at hand,  we define the following   operator 
\begin{equation}\label{def:operator-J}
\begin{array}{cc}
\cJ: \bW\subseteq \boldsymbol{\H}^1_{\Gamma_\rD}(\Omega)\times\rH^1(\Omega)\to \boldsymbol{\H}^1_{\Gamma_\rD}(\Omega)\times\rH^1(\Omega),\quad (\wh{\bu},\wh p)\mapsto \cJ(\wh{\bu},\wh p) := (\bu,p),
\end{array}
\end{equation}
where given $(\wh{\bu},\wh p)\in\bW$, there holds $\cJ(\wh{\bu},\wh p)=(\bu,p)\in \boldsymbol{\H}^1_{\Gamma_\rD}(\Omega)\times\rH^1(\Omega)$ which are components of the solution of the linearized version of problem \eqref{eq:weak-problem}: Find $((\bsig,\bu,p),\bgamma)\in \bH\times\mathbf{Q}$ such that
\begin{equation}\label{eq:linear-problem}
\begin{aligned}
a_{\wh{\bu},\wh p}((\bsig,\bu,p),(\btau,\bv,q)) + b((\btau,\bv,q),\bgamma)
&= F(\btau,\bv,q)
\quad\forall (\btau,\bv,q)\in\bH,\\
b((\bsig,\bu,p),\bbeta)
&=0
\quad\forall \bbeta\in \mathbf{Q}.
\end{aligned}
\end{equation}

Our task is to prove that operator $\mathcal{J}$ defined in \eqref{def:operator-J} has a fixed point, is contractive and maps $\bW$ onto $\bW$. It is clear that $((\bsig,\bu,p),\bgamma)$ is a solution to \eqref{eq:weak-problem} if and only if $(\bu,p)$ satisfies $\cJ(\bu,p) = (\bu,p)$, and consequently, the well-posedness of \eqref{eq:weak-problem} is equivalent to the unique solvability of the fixed-point problem: Find $(\bu,p)\in \bW$ such that
\begin{equation}\label{eq:fixed-point-problem-1}
\cJ(\bu,p) = (\bu,p).
\end{equation}

Hence, we now focus on establishing that  problem \eqref{eq:fixed-point-problem-1} has a unique solution. By the definition of the operator $\cJ$ (cf. \eqref{def:operator-J}), showing that $\cJ$ is well defined is equivalent to proving the well-posedness of problem \eqref{eq:linear-problem}.
\subsection{Stability properties and suitable inf-sup conditions}
For the analysis, we need to introduce some assumptions on the given data. In our case, for sake of the analysis in this section, we allow the permeability $\kappa(\bu,p)$ to be anisotropic but still require that it is a uniformly positive definite second-order tensor in $\bbL^\infty(\Omega)$, and  Lipschitz continuous in $p\in \rH^1(\Omega)$. That is, there exist positive constants $\kappa_1,\kappa_2$ such that 
\begin{equation}\label{prop-kappa}
\kappa_1|\bv|^2 \leq \bv^{\tt t}\kappa(\cdot,\cdot)\bv,  
\qquad  
\|\kappa(\cdot,q_1) - \kappa(\cdot,q_2)\|_{\bbL^\infty(\Omega)} \leq \kappa_2 \|q_1 -q_2\|_{1,\Omega},
\end{equation}
for all $ \bv,\bw \in\mathbb{R}^d\setminus\{\bzero\}$, and for all $q_1,q_2\in \rH^1(\Omega)$, we also assume that $\kappa(\cdot,0)=0$.

\begin{remark}
The Lipschitz continuity assumption in \eqref{prop-kappa} can be formulated more generally with respect to the displacement variable $\bu$, or simultaneously with respect to both $\bu$ and $p$. In particular, one may consider conditions of the form
\[\|\kappa(\bu_1,\cdot)-\kappa(\bu_2,\cdot)\|_{\bbL^\infty(\Omega)}
\leq \kappa_2|\bu_1-\bu_2|_{1,\Omega},\]
or
\[\|\kappa(\bu_1,p_1)-\kappa(\bu_2,p_2)\|_{\bbL^\infty(\Omega)}
\leq \kappa_2\bigl(	\|\bu_1-\bu_2\|_{1,\Omega} 	+\|p_1-p_2\|_{1,\Omega}
\bigr).\]
Thus, the analysis below can also accommodate permeability laws whose dependence on the displacement, or on both the displacement and the pressure, satisfies an analogous Lipschitz continuity property.
\end{remark}

Let us begin by establishing the boundedness of the forms involved. For a given $(\wh{\bu},\wh p)\in\bW$, the following estimates hold
\begin{equation}\label{eq:bound-a-b}
\begin{array}{c}
\ds |a_{\wh{\bu},\wh p}((\bsig,\bu,p),(\btau,\bv,q))| \leq \|a\|\|(\bsig,\bu,p)\|_\bH\|(\btau,\bv,q)\|_\bH,\\[1ex]
\ds |b((\btau,\bv,q),\bbeta)| \leq \|(\btau,\bv,q)\|_\bH \|\bbeta\|_{0,\O},
\end{array}
\end{equation}
where the constant $\|a\|$ is defined by
\begin{equation*}
\|a\|:= \frac{1}{2\mu} + \frac{\alpha\sqrt{d}}{\mu} + c_0 + \frac{d\alpha^2}{2\mu} + r + 2 + \delta_2 +\,
\delta_1 \left( 1 +\frac{1}{2\mu} + \frac{\alpha\sqrt{d}}{2\mu}
\right)^2,
\end{equation*}
which is independent of $\lambda$. Let us remark that the parameters $\delta_1$ and $\delta_2$ will be derived in the forthcoming analysis.

On the other hand, invoking  H\"older and trace inequalities we can readily observe that the right-hand side functional is bounded
\begin{equation}\label{eq:bound-F}
|F(\btau,\bv,q)| \leq \|F\| \|(\btau,\bv,q)\|_\bH ,
\end{equation}
with $\|F\|:=(1 + \delta_2 + C_{tr} ) (\|\bF\|_{0,\Omega} + \|g\|_{0,\Omega} + \|z_{\Gamma}\|_{-1/2,\Gamma})$, where $C_{tr}>0$ denotes the constant involved on  the trace inequality (cf. \cite[Theorem 1.4]{gatica14}).

On the other hand, from \cite{MR2449101} we have that $b(\cdot,\cdot)$ (cf. \eqref{eq:form-b}) satisfies the following  inf-sup condition
\begin{equation}\label{eq:infsup-b}
\sup_{\bzero\neq(\btau,\bv,q)\in \bH } \frac{b((\btau,\bv,q),\bbeta)}{\|(\btau,\bv,q)\|_{\bH}} \geq \beta\,\|\bbeta\|_{0,\O}
\quad \forall\,\bbeta\in {\mathbf{Q}},
\end{equation}
where the constant $\beta>0$ depends on the domain $\O$.

We begin by recalling the generalized Poincar\'e and Korn inequalities. More precisely, there exist positive constants $C_P$ and $C_K$, depending only on $\Omega$ and $\Gamma_{\rD}$, such that
\begin{equation}
\label{eq:poincare-Korn}
\|\bv\|_{0,\Omega}
\leq
C_P \|\nabla\bv\|_{0,\Omega}
\qan \|\nabla\bv\|_{0,\Omega}
\leq
C_K
\|\varepsilon(\bv)\|_{0,\Omega}, \qquad
\forall\,\bv\in \boldsymbol{\H}^1_{\Gamma_D}(\Omega).
\end{equation}

In addition, from \cite[Lemmas 3.1 and 3.2]{agr2015} we have that there exists $C_G>0$,  such that
\begin{equation*}\label{eq:ellipticity-Hooke}
\ds C_G \|\btau\|_{\bdiv;\Omega}^2\leq  \|\btau^{\texttt{d}}\|^2_{0,\O} + \|\bdiv\btau\|^2_{0,\O} \quad \forall\,\btau\in\boldsymbol{\mathcal{H}}_{\Gamma_\rN}(\bdiv,\O).
\end{equation*}
We remark that according to  \cite[Lemma 2.2]{gatica06elast} and \cite[Lemma 2.2]{gatica14},  the constant $C_G$ depends on $\mu$, $\Gamma_\rN$, $|\Omega|$, and the Poincar\'e constant.

The following result establishes the coercivity of the form $a_{\wh{\bu},\wh p}(\cdot,\cdot)$ on the space $\bH$, which will be key for the forthcoming analysis.
\begin{lemma}
\label{lmm:coerc_H}
Given $(\wh{\bu},\wh p)\in\bW$, the bilinear form $a_{\wh{\bu},\wh p}(\cdot,\cdot)$ is coercive in $\bH$, i.e., there exists a constant $\wh C>0$ such that
\begin{equation}\label{eq:coer-a}
   a_{\wh{\bu},\wh p}((\btau,\bv,q),(\btau,\bv,q))\geq \wh C \, \|(\btau,\bv,q)\|^2_\bH, \qquad \forall (\btau,\bv,q)\in\bH.
\end{equation}
\end{lemma}
\begin{proof}
Let $(\btau,\bv,q)\in\bH$, the following estimates hold
\begin{subequations}
\begin{equation}\label{eq:est_1}
\int_{\O}\mathcal{C}^{-1}(\btau + \alpha q\mathbf{I}):(\btau + \alpha q\mathbf{I})=\int_{\O}\mathcal{C}^{-1}\btau:\btau
+ \frac{2\alpha}{d\lambda+2\mu}\int_{\O}q\,\tr(\btau)+\frac{d\alpha^2}{d\lambda+2\mu}\int_{\O}q^2,
\end{equation}
\begin{equation}\label{eq:est_2}
2\int_{\O}\bv\cdot\bdiv\btau \geq -\dfrac{1}{\epsilon_{1}}\|\bv\|_{1,\O}^2-\epsilon_{1}\|\bdiv \btau\|_{0,\O}^2,\qquad\text{with }\epsilon_{1}>0,
\end{equation}
\begin{equation}\label{eq:est_3}
\|\mathcal{C}^{-1}(\btau + \alpha p\mathbf{I}) - \varepsilon(\bv)\|_{0,\O}^2\geq \dfrac{1}{2}\|\varepsilon(\bv)\|_{0,\O}^2-\|\mathcal{C}^{-1}(\btau + \alpha p\mathbf{I})\|_{0,\O}^2.
\end{equation}
\end{subequations}

Then, using the definition of $a_{\wh{\bu},\wh p}(\cdot,\cdot)$  given in \eqref{eq:form-a} together with \eqref{eq:est_1} and \eqref{eq:est_2} we have 
\begin{flalign*}
&a_{\wh{\bu},\wh p}((\btau,\bv,q),(\btau,\bv,q))&&\\
&\quad = \int_{\O}\mathcal{C}^{-1}\btau:\btau+ \frac{2\alpha}{d\lambda+2\mu}\int_{\O}q\,\tr(\btau)+ 2\int_{\O}\bv\cdot\bdiv\btau+ \Big(c_0+\frac{d\alpha^2}{d\lambda+2\mu}\Big)\|q\|_{0,\O}^2&&\\
&\qquad + \int_{\O}\kappa(\wh{\bu},\wh p)\nabla q\cdot \nabla q+ \delta_1\|\mathcal{C}^{-1}(\btau + \alpha p\mathbf{I}) - \varepsilon(\bv)\|_{0,\O}^2 + \delta_2\|\bdiv\btau\|_{0,\O}^2&&\\
&\quad =\int_{\O}\mathcal{C}^{-1}(\btau + \alpha q\mathbf{I}):(\btau + \alpha q\mathbf{I})+2\int_{\O}\bv\cdot\bdiv\btau+c_0\|q\|_{0,\O}^2+\int_{\O}\kappa(\wh{\bu},\wh q)\nabla q\cdot \nabla q &&\\
&\qquad + \delta_1\|\mathcal{C}^{-1}(\btau + \alpha p\mathbf{I}) - \varepsilon(\bv)\|_{0,\O}^2 + \delta_2\|\bdiv\btau\|_{0,\O}^2.&&
\end{flalign*}
Now, using \eqref{eq:est_3} and \eqref{prop-kappa} in the previous inequality on the right-hand side, we obtain
\begin{flalign*}
&a_{\wh{\bu},\wh p}((\btau,\bv,q),(\btau,\bv,q))&&\\
&\quad \geq \int_{\O}\mathcal{C}^{-1}(\btau + \alpha q\mathbf{I}):(\btau + \alpha q\mathbf{I})-\dfrac{1}{\epsilon_{1}}\|\bv\|_{1,\O}^2-\epsilon_{1}\|\bdiv \btau\|_{0,\O}^2+c_0\|q\|_{0,\O}^2&&\\
&\qquad +\kappa_1\|\nabla q\|_{0,\O}^2+ \dfrac{\delta_1}{2}\|\varepsilon(\bv)\|_{0,\O}^2-\delta_1\|\mathcal{C}^{-1}(\btau + \alpha q\mathbf{I})\|_{0,\O}^2 + \delta_2\|\bdiv\btau\|_{0,\O}^2.&&
\end{flalign*}
At this point, the aim is to analyze the additional terms on the augmented form. First, let us consider the  following  estimate
\begin{equation*}
\dfrac{1}{2\mu}\int_{\O}\mathcal{C}^{-1}(\btau + \alpha q\mathbf{I}):(\btau + \alpha q\mathbf{I})\geq \|\mathcal{C}^{-1}(\btau + \alpha q\mathbf{I})\|_{0,\O}^2.
\end{equation*}
Furthermore, using Korn's inequality (cf. \eqref{eq:poincare-Korn}), we obtain that
\begin{flalign*}
&a_{\wh{\bu},\wh p}((\btau,\bv,q),(\btau,\bv,q))&&\\
&\quad\geq
\left(1-\dfrac{\delta_1}{2\mu}\right)\int_{\O}\mathcal{C}^{-1}(\btau + \alpha q\mathbf{I}):(\btau + \alpha q\mathbf{I})+\left(\dfrac{\delta_1C_{K}}{2}-\dfrac{1}{\epsilon_{1}}\right)\|\bv\|_{1,\O}^2+c_0\|q\|_{0,\O}^2 &&\\
&\qquad+\kappa_{1}\|\nabla q\|_{0,\O}^2  +\left( \delta_2-\epsilon_{1}\right)\|\bdiv\btau\|_{0,\O}^2&&\\
&\quad \geq
\left(1-\dfrac{\delta_1}{2\mu}\right)\dfrac{1}{2\mu}\|(\btau + \alpha q\mathbf{I})^{{\texttt{d}}}\|_{0,\O}^2+\left(\dfrac{\delta_1C_{K}}{2}-\dfrac{1}{\epsilon_{1}}\right)\|\bv\|_{1,\O}^2 +c_0\|q\|_{0,\O}^2 &&\\
&\qquad+\kappa_{1}\|\nabla q\|_{0,\O}^2  +\left( \delta_2-\epsilon_{1}\right)\|\bdiv\btau\|_{0,\O}^2&&\\
&\quad =\left(1-\dfrac{\delta_1}{2\mu}\right)\dfrac{1}{2\mu}\|\btau^{\texttt{d}}\|_{0,\O}^2+\left(\dfrac{\delta_1C_{K}}{2}-\dfrac{1}{\epsilon_{1}}\right)\|\bv\|_{1,\O}^2 +c_0\|q\|_{0,\O}^2 &&\\
&\qquad+\kappa_{1}\|\nabla q\|_{0,\O}^2  +\left( \delta_2-\epsilon_{1}\right)\|\bdiv\btau\|_{0,\O}^2.
\end{flalign*}
In order to ensure that all constants on the right-hand side of the previous estimate are positive, the following constraints are required:
\begin{equation*}
2\mu> \delta_{1}>\dfrac{2}{C_{k}\epsilon_{1}};\qquad
\delta_{2}> \epsilon_{1}.
\end{equation*}
Thus, by setting $\epsilon_{1}:=\dfrac{2}{C_{K}\mu}$, $\delta_1:=\dfrac{3\mu}{2}$ and $\delta_2:=\dfrac{3}{C_K\mu}$, we have that
\begin{align*}
a_{\wh{\bu},\wh p}((\btau,\bv,p),(\btau,\bv,p)) &
\geq \dfrac{1}{8\mu}\|\btau^\texttt{d}\|_{0,\O}^2+\dfrac{\mu C_K}{4}\|\bv\|_{1,\O}^2
+c_0\|q\|_{0,\O}^2+\kappa_{1}\|\nabla q\|_{0,\O}^2  +\dfrac{1}{C_K\mu}\|\bdiv\btau\|_{0,\O}^2\\
& 
\geq \min\left\{\dfrac{1}{8\mu},\dfrac{1}{C_K\mu}\right\}C_G\|\btau\|_{\bdiv,\O}^2+\min\left\{c_0,\kappa_{1}\right\}\|q\|_{1,\O}^2+\dfrac{\mu C_K}{4}\|\bv\|_{1,\O}^2 \\
& 
=\widehat{C}\left(\|\btau\|_{\bdiv,\O}^2+\|q\|_{1,\O}^2+\|\bv\|_{1,\O}^2\right), 
\end{align*}
where $\ds \widehat{C}:= \min\left\{
\frac{C_G}{8\mu}, \frac{C_G}{C_K\mu}, c_0, \kappa_1, \frac{\mu C_K}{4}
\right\}$. This concludes the proof.
\end{proof}

\begin{remark}
Note that, in the case $c_0=0$, in order to ensure the ellipticity of the bilinear form {$a_{\wh{\bu},\wh p}(\cdot,\cdot)$}, it is necessary to suitably adjust the boundary conditions for $p$ (cf. \eqref{eq:bc}) so that the Poincar\'e inequality (or its generalized version) can be employed. In this way, the seminorm $|p|_{1,\Omega}$ can be used to control the {$\L^2$}-norm of $p$.
\end{remark}

In order to simplify the presentation  we  introduce the bilinear form $A_{\wh{\bu},\wh p}:(\bH\times\mathbf{Q}) \times(\bH\times\mathbf{Q})\to \mathbb{R}$ defined by
\begin{equation}\label{eq:form-A}
A_{\wh{\bu},\wh p}((\bsig,\bu,p,\bgamma),(\btau,\bv,q,\bbeta):=a_{\wh{\bu},\wh p}( (\bsig,\bu,p),(\btau,\bv,q) ) + b((\btau,\bv,q),\bgamma) +b((\bsig,\bu,p),\bbeta).
\end{equation}
\begin{lemma}\label{lem:well-def-J}
Given $r>0$, let us assume the following smallness assumption on the data
\begin{equation}\label{eq:assumption-J}
\dfrac{\rho}{r} (1 + \delta_2 + C_{tr} ) \big( \|\bF\|_{0,\Omega} +
 \|g\|_{0,\Omega} + \|z_\Gamma \|_{-1/2,\Gamma} \big)\leq 1,
\end{equation}
where
\begin{equation}\label{eq:def-rho}
\rho:=\dfrac{(\wh C + \beta + \|a\|)^2}{\wh C\beta}.
\end{equation}
Then, for a given $(\wh{\bu},\wh p)\in\bW$ (cf. \eqref{eq:set-W}),  there exists a unique $(\bu, p)\in\bW$ such that $\cJ(\wh{\bu},\wh p) = (\bu,p)$.
\end{lemma}
\begin{proof}
From the properties of $a_{\wh{\bu},\wh p}(\cdot,\cdot)$ and $b(\cdot,\cdot)$, \eqref{eq:bound-a-b}, \eqref{eq:coer-a} and \eqref{eq:infsup-b}, and a straightforward application of the Babu\v ska--Brezzi theory, we have that there exists a unique $((\bsig,\bu,p),\bgamma)\in 
 \bH\times\mathbf{Q}$ solution of \eqref{eq:linear-problem}, or, equivalently, the existence of a unique $(\bu,p)\in \boldsymbol{\H}^1_{\Gamma_\rD}(\Omega)\times\rH^1(\Omega)$ such that $\cJ(\wh{\bu},\wh p) = (\bu,p)$. Finally, from \cite[Proposition 2.36]{ernguermond}, we obtain that
\begin{equation}\label{eq:global-inf-sup}
\|((\bsig,\bu,p),\bgamma)\|_{\bH\times\mathbf{Q}} \leq \, \rho\,\sup_{\bzero\neq((\btau,\bv,q),\bbeta)\in 
 \bH\times\mathbf{Q}} \frac{A_{\wh{\bu},\wh p}((\bsig,\bu,p,\bgamma),(\btau,\bv,q,\bbeta)) }{\|((\btau,\bv,q),\bbeta)\|_{\bH\times\mathbf{Q}}},
\end{equation}
where $\|((\btau,\bv,q),\bbeta)\|_{\bH\times\mathbf{Q}}:=\|\btau\|_{\bdiv,\O}+\|\bv\|_{1,\O}+\|q\|_{1,\O}+\|\bbeta\|_{0,\O} $. Which together with \eqref{eq:linear-problem} and the boundedness of $F(\cdot)$ (cf. \eqref{eq:bound-F}), implies that
\begin{equation}\label{eq:J-from-W-to-W}
\|((\bsig,\bu,p),\bgamma)\|_{\bH\times\mathbf{Q}}
 \leq \rho \,\sup_{\bzero\neq((\btau,\bv,q),\bbeta)\in 
 \bH\times\mathbf{Q}}\frac{F(\btau,\bv,q)}{\|((\btau,\bv,q),\bbeta)\|_{\bH\times\mathbf{Q}}}
 \ds \leq \rho\, \|F\|.
\end{equation}
The above, together with the definition of $\|F\|$ (cf. \eqref{eq:bound-F}) and assumption \eqref{eq:assumption-J}, implies that $(\bu,p)$ belongs to $\bW$, thereby completing the proof.
\end{proof}

Now the aim is to prove that $\mathcal{J}$ is contractive, This is proved in the following result.
\begin{theorem}\label{theorem:unique-solution-weak1}
Let $\bF \in \boldsymbol{\L}^2(\Omega)$, $g \in \L^2(\Omega)$ and $z_\Gamma \in \rH^{-1/2}(\Gamma)$ such that
\begin{equation}\label{eq:assumption-J-2}
\ds \dfrac{\rho}{r}\max\{1,\, \rho \kappa_2 r\}\,  (1 + \delta_2 + C_{tr} ) (\|\bF\|_{0,\Omega} + \|g\|_{0,\Omega} + \|z_{\Gamma}\|_{-1/2,\Gamma}) < 1,
\end{equation}
where $\rho$ is defined in \eqref{eq:def-rho}.
Then, the operator $\cJ$ (cf. \eqref{def:operator-J}) has a unique fixed point $(\bu,p)\in\bW$. Equivalently, the problem \eqref{eq:weak-problem} has a unique solution $((\bsig,\bu,p),\bgamma)\in \bH\times\mathbf{Q}$ with $(\bu,p)\in\bW$. In addition, we have the following continuous dependence on data
%
\begin{equation}\label{eq:stability-weak1}
\|((\bsig,\bu,p),\bgamma)\|_{\bH\times\mathbf{Q}} \leq \rho\,  
 (1 + \delta_2 + C_{tr} ) (\|\bF\|_{0,\Omega} + \|g\|_{0,\Omega} + \|z_{\Gamma}\|_{-1/2,\Gamma}).
\end{equation}
\end{theorem}
\begin{proof}
We begin by recalling from the previous analysis that assumption \eqref{eq:assumption-J-2} ensures the well-definedness of $\cJ$. Let $(\wh{\bu}_1,\wh p_1),(\wh{\bu}_2,\wh p_2)\in\bW$ be such that $\cJ(\wh{\bu}_1,\wh p_1)=(\bu_1,p_1)$ and $\cJ(\wh{\bu}_2,\wh p_2)=(\bu_2,p_2)$. By the definition of $\mathcal{J}$ (cf. \eqref{eq:linear-problem}), there exist $(\bsig_1,\bgamma_1)$ and $(\bsig_2,\bgamma_2)$ in $\bbH_{\Gamma_\rN}(\bdiv,\O)\times\boldsymbol{\mathcal{L}}^2(\O)_{\text{skew}}$ such that, for all $((\btau,\bv,q),\bbeta)\in\bH\times\mathbf{Q}$, there hold
\begin{equation*}
\begin{aligned}
a_{\wh{\bu}_1,\wh p_1}((\bsig_1,\bu_1,p_1),(\btau,\bv,q)) + b((\btau,\bv,q),\bgamma_1)
&= F(\btau,\bv,q),\\
b((\bsig_1,\bu_1,p_1),\bbeta)
&=0,
\end{aligned}
\end{equation*}
and
\begin{equation*}
\begin{aligned}
a_{\wh{\bu}_2,\wh p_2}((\bsig_2,\bu_2,p_2),(\btau,\bv,q)) + b((\btau,\bv,q),\bgamma_2)
&= F(\btau,\bv,q),\\
b((\bsig_1,\bu_1,p_1),\bbeta)
&=0.
\end{aligned}
\end{equation*}
Then, subtracting these problems and using the linearity of $b(\cdot,\cdot)$, we obtain 
\begin{multline*}
a_{\wh{\bu}_1,\wh p_1}((\bsig_1,\bu_1,p_1),(\btau,\bv,q)) -a_{\wh{\bu}_2,\wh p_2}((\bsig_2,\bu_2,p_2),(\btau,\bv,q)) \\
+ b((\btau,\bv,q),\bgamma_1-\bgamma_2)+b((\bsig_1-\bsig_2,\bu_1-\bu_2,p_1-p_2),\bbeta)=0.
\end{multline*}
Now, adding and subtracting the term $a_{\wh{\bu}_1,\wh p_1}((\bsig_2,\bu_2,p_2),(\btau,\bv,q))$ we arrive at
\begin{multline}\label{eq:auxeq-1}
a_{\wh{\bu}_1,\wh p_1}((\bsig_1-\bsig_2,\bu_1-\bu_2,p_1-p_2),(\btau,\bv,q)) + b((\btau,\bv,q),\bgamma_1-\bgamma_2)\\
+ b((\bsig_1-\bsig_2,\bu_1-\bu_2,p_1-p_2),\bbeta)
=\int_{\O}[\kappa(\widehat{\bu}_2,\widehat{p}_2)-\kappa(\widehat{\bu}_1,\widehat{p}_1)]\nabla p_2\cdot\nabla q.
\end{multline}
Therefore, recalling that $(\wh{\bu},\wh p)\in\bW$, we can use the latter identity, the global inf-sup condition \eqref{eq:global-inf-sup}, and the assumptions of $\kappa(\cdot,\cdot)$ (cf. \eqref{prop-kappa}), to obtain
%
\begin{align*}
 \|(\bu_1,p_1) - (\bu_2,p_2)\| & \leq  
\|((\bsig_1 -\bsig_2,\bu_1 - \bu_2,p_1-p_2),\bgamma_1 -\bgamma_2)\|_{ \bH \times \mathbf{Q}} \\ 
& \leq \rho \,\sup_{\bzero\neq((\btau,\bv,q),\bbeta)\in 
 \bH\times\mathbf{Q}} \frac{A_{\wh{\bu}_1,\wh p_1}((\bsig_1-\bsig_2,\bu_1-\bu_2,p_1-p_2,\bgamma_1-\bgamma_2),(\btau,\bv,q,\bbeta)) }{\|((\btau,\bv,q),\bbeta)\| _{\bH\times\mathbf{Q}}}\\ 
& = \rho \,\sup_{\bzero\neq((\btau,\bv,q),\bbeta)\in 
 \bH\times\mathbf{Q}} \dfrac{\ds \int_{\O}[\kappa(\widehat{\bu}_2,\widehat{p}_2)-\kappa(\widehat{\bu}_1,\widehat{p}_1)]\nabla p_2\cdot\nabla q }{\|((\btau,\bv,q),\bbeta)\|_{ \bH\times\mathbf{Q}}}\\ 
& \leq \rho\,\|\kappa(\wh{\bu}_2,\wh p2) - \kappa(\wh{\bu}_1,\wh p_1)\|_{\bbL^\infty(\Omega)} \|\nabla p_2\|_{0,\Omega}\, .
\end{align*}
Then, using the Lipschitz continuity of $\kappa(\cdot,\cdot)$ (cf. \eqref{prop-kappa}), together with the fact that $\cJ(\wh{\bu}_2,\wh p_2)=(\bu_2,p_2)\in \bW$, and therefore satisfies the estimate \eqref{eq:J-from-W-to-W}, it follows that
\begin{align*}
 \|\cJ(\wh{\bu}_1,\wh p_1)- \cJ(\wh{\bu}_2,\wh p_2)\|  & =  \|(\bu_1, p_1) - (\bu_2, p_2)\| \\
& \leq \rho\,\|\kappa(\wh{\bu}_2,\wh p_2) - \kappa(\wh{\bu}_1,\wh p_1)\|_{\bbL^\infty(\Omega)} \|\nabla p_2\|_{0,\Omega}\\ 
& \leq \rho\,\kappa_2\,\|\wh p_2 - \wh p_1\|_{1,\Omega}\, \rho\|F\|
 \leq \rho^2\,\kappa_2\,  \|F\| \|(\wh{\bu}_1, \wh p_1) - (\wh{\bu}_2, \wh p_2)\|.
\end{align*}
{The  previous estimate}, in combination with the definition of $\|F\|$ (cf. \eqref{eq:bound-F}), assumption \eqref{eq:assumption-J-2}, and the Banach fixed-point theorem, implies that $\cJ$ has a unique fixed point in $\bW$. Equivalently,  there exists a unique $((\bsig,\bu,p),\bgamma)\in \bH\times\mathbf{Q}$ solution of \eqref{eq:weak-problem}. Finally, estimate \eqref{eq:stability-weak1} is obtained analogously to \eqref{eq:J-from-W-to-W}, which completes the proof.
\end{proof}

\section{The mixed element discretization}
\label{sec:fem}
This section is devoted to the analysis of a finite element method to approximate the solution of  problem \eqref{eq:weak-problem}. Let us begin by introducing some preliminary definitions and notations. Let us denote by $\CT_h$ a regular partition of $\overline{\O}$ which consists in triangles/tetrahedra depending on the dimension of the domain. For an element $K\in\CT_h$, we denote its diameter by $h_K$, whereas the mesh size is defined by $h:=\max\{h_K:\;K\in\CT_h\}$.

Given $\ell\geq 0$ and $K\in\CT_h$, we define the space of polynomials of degree less than or equal to $\ell$ on $K$ by $\textrm{P}_{\ell}(K)$. The corresponding vectorial and tensorial spaces are denoted by $\textbf{\textrm{P}}_\ell(K)$ and $\boldsymbol{\mathcal{P}}_{\ell}(K)$, respectively.

In this way, the finite element subspaces are given by:
\begin{equation*}
\begin{array}{l}
\bbH^{\bsig}_h  := \ds \Big\{\btau_h\in\bbH_{\Gamma_\rN}(\bdiv,\O) : \quad \btau_h|_T\in\boldsymbol{\mathcal{P}}_{k+1}(K)\quad \forall\, K\in\cT_h\Big\}, \\[2ex]
\bH^{\bu}_h:= \left\{ \bv_h \in \mathrm{C}(\overline{\Omega})^d \cap\boldsymbol{\H}_{\Gamma_D}^1(\O) :\quad  \bv_h|_K\in \textbf{\textrm{P}}_{k+1}(K)\;\; \quad \forall\, K\in \mathcal{T}_h\right\},\\[2ex]
\rH^{p}_h:= \left\{ q_h \in \mathrm{C}(\overline{\Omega}) :\quad  q_h|_{K}\in \textrm{P}_{k+1}(K)\;\; \quad \forall\, K\in \mathcal{T}_h\right\},\\[2ex]
\bbH^{\bgamma}_h:= \left\{ \bbeta_h \in \bbL^2_{\text{skew}}(\Omega) :\quad \bbeta_h|_T\in \boldsymbol{\mathcal{P}}_k(T)\;\; \quad \forall\, K\in \mathcal{T}_h\right\},
\end{array}
\end{equation*}
which will be used to approximate the poroelastic  tensor, the displacement, the pressure and the rotation tensor, respectively. Note that $\bbH^{\bsig}_h$ and $\bbH^{\bgamma}_h$ are subspaces of the Arnold--Falk--Winther finite element spaces (cf. \cite{arnold-2007}), thereby ensuring that the discrete version of the inf-sup condition \eqref{eq:infsup-b} is satisfied.

With these spaces at hand, we introduce the finite element discretization of problem \eqref{eq:weak-problem}:  Find $((\bsig_h,\bu_h,p_h),\bgamma_h)\in \bH_h \times \mathbf{Q}_h$, such that 
\begin{equation}\label{eq:weak-problem-fem}
\begin{aligned}
a_{\bu_h,p_h}((\bsig_h,\bu_h,p_h),(\btau_h,\bv_h,q_h)) + b((\btau_h,\bv_h,q_h),\bgamma_h)
&= F(\btau_h,\bv_h,q_h),\\
b((\bsig_h,\bu_h,p_h),\bbeta_h)
&=0,
\end{aligned}
\end{equation}
for all $(\btau_h,\bv_h,q_h)\in\bH_h :=\bbH^{\bsig}_h\times\bH^{\bu}_h\times \rH^{p}_h$ and $\bbeta_h\in \mathbf{Q}_h:=\bbH^{\bgamma}_h$. This discrete scheme remains as nonlinear and hence, the analysis of stability, and existence and uniqueness of solutions at discrete level must be performed. 

Since the method is conforming, we immediately observe that the coercivity result given in Lemma \ref{lmm:coerc_H} since this result hold in the whole space $\bH$ and in particular in $\bH_h$. Hence, the constant $\widehat{C}>0$ provided in  Lemma \ref{lmm:coerc_H} also holds for the following estimate of the discrete linearized problem
\begin{equation}\label{eq:coer-a_disc}
   a_{\wh{\bu}_h,\wh p_h}((\btau_h,\bv_h,q_h),(\btau_h,\bv_h,q_h))\geq \wh C \, \|(\btau_h,\bv_h,q_h)\|^2_{\bH}, \qquad \forall (\btau_h,\bv_h,q_h)\in\bH_h,
\end{equation}
for a given $(\wh{\bu}_h,\,\wh p_h) \in \bH^{\bu}_h \times \rH^{p}_h$. It is worth to remark  that {$\widehat{C}>0$} is independent of $h$. On the other hand, the bilinear form $b$ (cf. \eqref{eq:form-b}) satisfies the following discrete inf-sup condition (see \cite[Theorem~11.9]{arnold-2007})
\begin{equation}\label{eq:infsup-b_disc}
\sup_{\bzero\neq(\btau_h,\bv_h,q_h)\in \bH_h } \frac{b((\btau_h,\bv_h,q_h),\bbeta_h)}{\|(\btau_h,\bv_h,q_h)\|_{\bH}} \geq \wt\beta\,\|\bbeta_h\|_{\mathbf{Q}}
\quad \forall\,\bbeta_h\in {\mathbf{Q}_h}.
\end{equation}

With the aim of establish existence and uniqueness pf the discrete problem via a fixed point argument,  we being by introducing the following ball
\begin{equation}\label{eq:set-W_disc}
 \bW_h := \Big\{ (\wh{\bu}_h,\wh p_h)\in\bH^{\bu}_h\times \rH^{p}_h\,:\quad \|(\wh{\bu}_h,\wh p_h)\| \leq r \Big\},
\end{equation}
where $r\geq 0$ is the corresponding radius. Next, we define the discrete version of $\cJ$ by
\begin{equation}\label{def:operator-J-h}
\begin{array}{cc}
\cJ_h: \bW_h\subseteq \bH^{\bu}_h\times \rH^{p}_h \to \bH^{\bu}_h\times \rH^{p}_h,\quad (\wh{\bu}_h,\wh p_h)\mapsto \cJ_h(\wh{\bu}_h,\wh p_h) := (\bu_h,p_h),
\end{array}
\end{equation}
where given $(\wh{\bu},\wh p)\in\bW_h$, the operator is such that  $\cJ_h(\wh{\bu}_h,\wh p_h)=(\bu_h,p_h)\in\bH^{\bu}_h\times \rH^{p}_h$ where this pair corresponds to be  the solution of the linearized version of   \eqref{eq:weak-problem-fem} which reads:  Find $((\bsig_h,\bu_h,p_h),\bgamma_h)\in \bH_h\times\mathbf{Q}_h$ such that
\begin{equation}\label{eq:linear-problem-disc}
\begin{aligned}
a_{\wh{\bu}_h,\wh p_h}((\bsig_h,\bu_h,p_h),(\btau_h,\bv_h,q_h)) + b((\btau_h,\bv_h,q_h),\bgamma_h)
&= F(\btau_h,\bv_h,q_h),\\
b((\bsig_h,\bu_h,p_h),\bbeta_h)
&=0,
\end{aligned}
\end{equation}
for all $(\btau,\bv,q)\in\bH_h$ and $\bbeta_h\in \mathbf{Q}_h$.

It is clear that $((\bsig_h,\bu_h,p_h),\bgamma_h)$ is a solution to \eqref{eq:weak-problem-fem} if and only if $(\bu_h,p_h)$ satisfies $\cJ_h(\bu_h,p_h) = (\bu_h,p_h)$, so that the well-posedness of \eqref{eq:weak-problem-fem} amounts to establishing the unique solvability of the following fixed-point problem: Find $(\bu_h,p_h)\in \bW_h$ such that
\begin{equation}\label{eq:fixed-point-problem-h}
\cJ_h(\bu_h,p_h) = (\bu_h,p_h).
\end{equation}

\noindent Accordingly, we shall prove that \eqref{eq:fixed-point-problem-h} admits a unique solution. The following lemma establishes the well-posedness of the linearized problem \eqref{eq:linear-problem-disc}, thereby showing that the operator  $\cJ_h$ is well defined.
\begin{lemma}\label{lem:well-def-J-h}
	Given $r>0$, assume that 
	\begin{equation*}\label{eq:assumption-J-h}
\dfrac{\wt\rho}{r} (1 + \delta_2 + C_{tr} ) \big( \|\bF\|_{0,\Omega} +
\|g\|_{0,\Omega} + \|z_\Gamma \|_{-1/2,\Gamma} \big)\leq 1,
	\end{equation*}
where $\wt\rho$ is the discrete version of $\rho$ (cf. \eqref{eq:def-rho}), defined by
\begin{equation}\label{eq:def-rho-h}
\wt{\rho}:=\dfrac{(\wh C + \wt\beta + \|a\|)^2}{\wh C\wt\beta}.
\end{equation}
Then, given $(\wh{\bu}_h,\wh p_h)\in\bW_h$ (cf. \eqref{eq:set-W_disc}) there exists a unique $(\bu_h, p_h)\in\bW_h$ such that $\cJ_h(\wh{\bu}_h,\wh p_h) =(\bu_h, p_h)$.
\end{lemma}
\begin{proof}
Given $(\wh{\bu}_h,\wh p_h)\in\bW_h$, we proceed analogously to the proof of Lemma~\ref{lem:well-def-J} and use the boundedness of the bilinear forms {$a(\cdot,\cdot)$ and $b(\cdot,\cdot)$} (cf. \eqref{eq:bound-a-b}), the discrete inf-sup condition for {$b(\cdot,\cdot)$} (cf. \eqref{eq:infsup-b_disc}), the coercivity of {$a(\cdot,\cdot)$} (cf. \eqref{eq:coer-a_disc}), and \cite[Proposition 2.36]{ernguermond} to establish the discrete global inf-sup condition
\begin{equation}\label{eq:J-from-W-to-W-h}
\|((\bsig_h,\bu_h,p_h),\bgamma_h)\|_{\bH\times\mathbf{Q}} 
 \leq \, \wt \rho\,\sup_{\bzero\neq((\btau_h,\bv_h,q_h),\bbeta_h)\in 
	\bH_h\times\mathbf{Q}_h} \frac{A_{\wh{\bu}_h,\wh p_h}((\bsig_h,\bu_h,p_h,\bgamma_h),(\btau_h,\bv_h,q_h,\bbeta_h)) }{\|((\btau_h,\bv_h,q_h),\bbeta_h)\|_{\bH\times\mathbf{Q}}}.
\end{equation}
Therefore, since surjectivity and injectivity are equivalent for finite-dimensional linear problems, \eqref{eq:J-from-W-to-W-h} together with the Banach--Ne\v cas--Babu\v ska theorem yields a unique $((\bsig_h,\bu_h,p_h),\bgamma_h)\in \bH_h\times\mathbf{Q}_h$ satisfying \eqref{eq:linear-problem-disc}, with $(\bu_h,p_h)\in\bW_h$. This concludes the proof.
\end{proof}

The following theorem states the main result of this section, concerning the existence and uniqueness of a solution to the fixed-point problem \eqref{eq:fixed-point-problem-h}, or, equivalently, the well-posedness of \eqref{eq:weak-problem-fem}.

\begin{theorem}\label{theorem:unique-solution-weak1-h}
Let $\bF \in \bL^2(\Omega)$, $g \in \rL^2(\Omega)$, and $z_\Gamma \in \rH^{-1/2}(\Gamma)$ satisfy
\begin{equation}\label{eq:assumption-J-2-h}
\ds \dfrac{\wt\rho}{r}\max\{1,\, \wt\rho \kappa_2 r\}\,
(1 + \delta_2 + C_{tr} )
(\|\bF\|_{0,\Omega} + \|g\|_{0,\Omega} + \|z_{\Gamma}\|_{-1/2,\Gamma}) < 1,
\end{equation}
where $\wt\rho$ is defined in \eqref{eq:def-rho-h}.
Then, the operator $\cJ_h$ (cf. \eqref{def:operator-J-h}) has a unique fixed point
$(\bu_h,p_h)\in\bW_h$. Equivalently, problem \eqref{eq:weak-problem-fem} has a unique solution
$((\bsig_h,\bu_h,p_h),\bgamma_h)\in \bH_h\times\mathbf{Q}_h$ with
$(\bu_h,p_h)\in\bW_h$. Moreover, the following stability estimate holds
\begin{equation}\label{eq:stability-weak-h}
\|((\bsig_h,\bu_h,p_h),\bgamma_h)\|_{\bH\times\mathbf{Q}}
\leq \wt\rho\, (1 + \delta_2 + C_{tr} ) (\|\bF\|_{0,\Omega} + \|g\|_{0,\Omega} + \|z_{\Gamma}\|_{-1/2,\Gamma}).
\end{equation}
\end{theorem}
\begin{proof}
First, as in the continuous case, assumption \eqref{eq:assumption-J-2-h} ensures that the operator $\cJ_h$ is well defined. Moreover, adapting the arguments used in the proof of Theorem~\ref{theorem:unique-solution-weak1}, we obtain
\begin{equation*} 
\|\cJ_h(\wh{\bu}_1,\wh p_1)- \cJ_h(\wh{\bu}_2,\wh p_2)\|  = 
\|(\bu_1,p_1)-(\bu_2,p_2)\| \leq \wt\rho^{\,2}\,\kappa_2\,
\big( \|g\|_{0,\Omega}+\|z_\Gamma\|_{-1/2,\Gamma} + \|\bF\|_{0,\Omega} \big)  \|(\wh{\bu}_1,\wh p_1)-(\wh{\bu}_2,\wh p_2)\|,
\end{equation*}
for all $(\wh{\bu}_1,\wh p_1),(\wh{\bu}_2,\wh p_2)\in\bW_h$. It then follows from \eqref{eq:assumption-J-2-h} that $\cJ_h$ is a 	contraction mapping on $\bW_h$. Hence, by the Banach fixed-point theorem, problem \eqref{eq:fixed-point-problem-h}, or equivalently
\eqref{eq:weak-problem-fem}, admits a unique solution. Finally, estimate \eqref{eq:stability-weak-h} follows analogously to \eqref{eq:J-from-W-to-W}, which completes the proof.
\end{proof}

\subsection{A priori error estimates}\label{sec:apriori}
From now on, we suppose that the assumptions of Theorem~\ref{theorem:unique-solution-weak1} and Theorem~\ref{theorem:unique-solution-weak1-h} are satisfied, and let $((\bsig,\bu,p),\bgamma)\in\bH\times\mathbf{Q}$ and $((\bsig_h,\bu_h,p_h),\bgamma_h)\in\bH_h\times\mathbf{Q}_h$ denote the unique solutions to \eqref{eq:weak-problem} and \eqref{eq:weak-problem-fem}, respectively.

Then, in order to simplify the subsequent analysis, we write
\[
\texttt{e}_{\bsig}:=\bsig-\bsig_h,\qquad
\texttt{e}_{\bu}:=\bu-\bu_h,\qquad
e_{p}:=p-p_h,\qquad
\texttt{e}_{\bgamma}:=\bgamma-\bgamma_h.
\]
As usual, for a given $((\overline{\bsig}_h,\overline{\bu}_h,\overline{p}_h),\overline{\bgamma}_h)\in\bH_h\times\mathbf{Q}_h$, we shall then decompose these errors into
\begin{equation}\label{eq:decompositions}
\texttt{e}_{\bsig}=\bxi_{\bsig}+\bchi_{\bsig},\qquad
\texttt{e}_{\bu}=\bxi_{\bu}+\bchi_{\bu},\qquad
e_{p}=\xi_p+\chi_p,\qquad
\texttt{e}_{\bgamma}=\bxi_{\bgamma}+\bchi_{\bgamma},
\end{equation}
where
\begin{equation*}
\begin{array}{cc}
\bxi_{\bsig}:=\bsig-\overline{\bsig}_h,\qquad
\bchi_{\bsig}:=\overline{\bsig}_h-\bsig_h,\qquad
\bxi_{\bu}:=\bu-\overline{\bu}_h,\qquad
\bchi_{\bu}:=\overline{\bu}_h-\bu_h,\\
\xi_p:=p-\overline{p}_h,\qquad
\chi_p:=\overline{p}_h-p_h,\qquad
\bxi_{\bgamma}:=\bgamma-\overline{\bgamma}_h,\qquad
\bchi_{\bgamma}:=\overline{\bgamma}_h-\bgamma_h.
\end{array}
\end{equation*}

Recalling the definition of the bilinear form $A_{\wh{\bu},\wh p}(\cdot,\cdot)$ in \eqref{eq:form-A}, from \eqref{eq:weak-problem} and \eqref{eq:weak-problem-fem} we have that the following identities hold
\begin{equation*} 
A_{\bu, p}((\bsig,\bu,p,\bgamma),(\btau,\bv,q,\bbeta) = F(\btau,\bv,q),\quad 
A_{\bu_h, p_h}((\bsig_h,\bu_h,p_h,\bgamma_h),(\btau_h,\bv_h,q_h,\bbeta_h) = F(\btau_h,\bv_h,q_h),
\end{equation*}
for all $((\btau,\bv,q),\bbeta) \in \bH\times\mathbf{Q}$ and $((\btau_h,\bv_h,q_h),\bbeta_h) \in \bH_h\times\mathbf{Q}_h$, respectively. From these relations, and similarly to \eqref{eq:auxeq-1}, we can obtain that for all $((\btau_h,\bv_h,q_h),\bbeta_h) \in \bH_h\times\mathbf{Q}_h$, there holds
\begin{equation*}
A_{\bu_h,p_h}( (\texttt{e}_{\bsig},\texttt{e}_{\bu}, \texttt{e}_{p}, \texttt{e}_{\bgamma}),(\btau_h,\bv_h,q_h,\bbeta_h) ) = \int_\Omega (\kappa(\bu_h,p_h) - \kappa(\bu,p) )\nabla p\cdot \nabla q_h,
\end{equation*}
which together with the definition of the errors in \eqref{eq:decompositions}, implies that 
\begin{equation}\label{eq:auxiliar-equation}
\begin{array}{ll}
A_{\bu_h,p_h}( (\bchi_{\bsig},\bchi_{\bu}, \chi_{p}, \bchi_{\bgamma}),(\btau_h,\bv_h,q_h,\bbeta_h) ) =\\
\ds \qquad -A_{\bu_h,p_h}( (\bxi_{\bsig},\bxi_{\bu}, \xi_{p}, \bxi_{\bgamma}),(\btau_h,\bv_h,q_h,\bbeta_h) ) + \int_\Omega (\kappa(\bu_h,p_h) - \kappa(\bu,p) )\nabla p\cdot \nabla q_h,
\end{array}
\end{equation}
for all for all $((\btau_h,\bv_h,q_h),\bbeta_h) \in \bH_h\times\mathbf{Q}_h$. Then, since $(\bu_h,p_h) \in \bW_h$, we apply the discrete inf-sup condition \eqref{eq:J-from-W-to-W-h} at the left-hand side of \eqref{eq:auxiliar-equation} followed by the continuity properties of {$a_{\wh{\bu},\wh p}(\cdot,\cdot)$ and $b(\cdot,\cdot)$} (cf. \eqref{eq:bound-a-b}) on the right-hand side of \eqref{eq:auxiliar-equation}, to obtain
\begin{equation}\label{eq:auxiliar-equation-2}
\begin{array}{ll}
\|((\bchi_{\bsig},\bchi_{\bu},\chi_{p}),\bchi_{\bgamma})\|_{\bH\times\mathbf{Q}} \\[2ex]
\qquad \ds \leq \wt\rho\,\Big( (\|a\| + 1) \big(\|(\bxi_{\bsig},\bxi_{\bu},\xi_{p})\|_{\bH}\big) + \|\bxi_{\bgamma}\|_{0,\O} + \kappa_2 \|p_h - p\|_{1,\Omega} \|\nabla p\|_{0,\Omega} \Big)\\[2ex]
\qquad \ds \leq \wt\rho\,\Big( (\|a\| + 1) \big(\|((\bxi_{\bsig},\bxi_{\bu},\xi_{p}),\bxi_{\bgamma})\|_{\bH\times\mathbf{Q}}\big) + \kappa_2 (\|\xi_p\|_{1,\Omega} + \|\chi_p\|_{1,\Omega}) \|\nabla p\|_{0,\Omega} \Big).
\end{array}
\end{equation}

Now we turn to providing a best approximation estimate corresponding with the Galerkin scheme \eqref{eq:weak-problem-fem}.
\begin{theorem}
\label{thm:wea}
Let us assume that the hypotheses of Theorem \ref{theorem:unique-solution-weak1} and Theorem \ref{theorem:unique-solution-weak1-h} are satisfied. Additionally, assume the following smallness assumption
\begin{equation}\label{eq:assumption-cea}
\kappa_2\, \wt\rho\,\rho\, (1 + \delta_2 + C_{tr} ) (\|\bF\|_{0,\Omega} + \|g\|_{0,\Omega} + \|z_{\Gamma}\|_{-1/2,\Gamma})\leq \frac{1}{2},
\end{equation}
where $\rho$ and $\widetilde{\rho}$ are defined in \eqref{eq:def-rho} and \eqref{eq:def-rho-h}, respectively and $\kappa_2$ is the Lipschitz constant given in \eqref{prop-kappa}	. Then, there exists a constant $C_{\texttt{C\'ea}}>0$ such that the following estimate holds
\begin{equation*}
\|((\texttt{e}_{\bsig},\texttt{e}_{\bu},e_p),\texttt{e}_{\bgamma})\|_{\bH\times\mathbf{Q}} 
 \leq C_{\texttt{C\'ea}}\inf_{((\btau_h,\bv_h,q_h),\bbeta_h)\in\bH_h\times\mathbf{Q}_h} \|((\bsig,\bu,p),\bgamma)-((\btau_h,\bv_h,q_h),\bbeta_h)\|_{\bH\times\mathbf{Q}}.
\end{equation*}
\end{theorem}
\begin{proof}
From \eqref{eq:auxiliar-equation-2}, we have
\begin{equation}\label{eq:auxiliar-equation-4}
\|(\bchi_{\bsig},\bchi_{\bu},\chi_{p})\|_\bH \big(1-\kappa_2\, \wt\rho \, \|\nabla p\|_{0,\Omega} \big) + \|\bchi_{\bgamma}\|_{0,\O}
\leq \wt\rho\,\Big( (\|a\| + 1) \big(\|((\bxi_{\bsig},\bxi_{\bu},\xi_{p}),\bxi_{\bgamma})\|_{\bH\times\mathbf{Q}}\big) + \kappa_2 \|\xi_p\|_{1,\Omega} \|\nabla p\|_{0,\Omega} \Big).
\end{equation}
Since $p$ satisfies \eqref{eq:J-from-W-to-W}, assumption \eqref{eq:assumption-cea} yields
\begin{equation*}
1-\kappa_2\, \wt\rho \, \|\nabla p\|_{0,\Omega}\geq 1- \kappa_2\, \wt\rho\,\rho\, \|F\| 
=1- \kappa_2\, \wt\rho\,\rho\, (1 + \delta_2 + C_{tr} ) (\|\bF\|_{0,\Omega} + \|g\|_{0,\Omega} + \|z_{\Gamma}\|_{-1/2,\Gamma}) \geq \frac{1}{2}.
\end{equation*}
Combining this inequality with the estimate \eqref{eq:auxiliar-equation-4}, and using again the fact that $p$ satisfies \eqref{eq:J-from-W-to-W}, we obtain
\begin{equation}\label{eq:auxiliar-equation-3}
\|((\bchi_{\bsig},\bchi_{\bu},\chi_{p}),\bchi_{\bgamma})\|_{\bH\times\mathbf{Q}}  \leq C\, \|((\bxi_{\bsig},\bxi_{\bu},\xi_{p}),\bxi_{\bgamma})\|_{\bH\times\mathbf{Q}},
\end{equation}
with $C>0$ independent of $h$. Consequently, from \eqref{eq:decompositions}, \eqref{eq:auxiliar-equation-3}, and the triangle inequality, we obtain
\begin{equation*}
\|((\texttt{e}_{\bsig},\texttt{e}_{\bu},e_p),\texttt{e}_{\bgamma})\|_{\bH\times\mathbf{Q}} \leq \|((\bchi_{\bsig},\bchi_{\bu},\chi_{p}),\bchi_{\bgamma})\|_{\bH\times\mathbf{Q}} + \|((\bxi_{\bsig},\bxi_{\bu},\xi_{p}),\bxi_{\bgamma})\|_{\bH\times\mathbf{Q}}
\leq (C+1) \|((\bxi_{\bsig},\bxi_{\bu},\xi_{p}),\bxi_{\bgamma})\|_{\bH\times\mathbf{Q}},
\end{equation*}
which, since $((\overline{\bsig}_h,\overline{\bu}_h,\overline{p}_h),\overline{\bgamma}_h)\in\bH_h\times\mathbf{Q}_h$ are arbitrary, completes the proof.
\end{proof}

With this best approximation result at hand, we are in position to obtain the desire a priori error estimates for the method. To do this task, we first introduce some approximation results 
for the discrete spaces involved on the scheme which hold for the regularity that we assume for each case. The properties are presented in the sequel:
\begin{equation}
\label{eq:approx-properties}
\begin{aligned}
\inf_{\boldsymbol{\tau}_h\in\boldsymbol{\mathcal{H}}_h^{\bsig}}
\|\boldsymbol{\tau}-\boldsymbol{\tau}_h\|_{\mathbf{div},\Omega}
&\leq Ch^m\left(
\|\boldsymbol{\tau}\|_{m,\Omega}+
\|\mathbf{div}\,\boldsymbol{\tau}\|_{m,\Omega}
\right),
\qquad 1\leq m\leq k+1,
\\
\inf_{\bv_h\in\bH^{\bu}_h}
\|\bv-\bv_h\|_{1,\Omega}
&\leq Ch^{m}\|\bv\|_{m+1,\Omega},
\qquad 0\leq m\leq k+1,
\\
\inf_{q_h\in\rH^{p}_h}
\|q-q_h\|_{1,\Omega}
&\leq Ch^m\|q\|_{m+1,\Omega},
\qquad 0\leq m\leq k+1,
\\
\inf_{\bbeta_h\in\mathbf{Q}_h}
\|\bbeta-\bbeta_h\|_{0,\Omega}
&\leq Ch^m\|\bbeta\|_{m,\Omega},
\qquad 0\leq m\leq k+1.
\end{aligned}
\end{equation}
For the approximation properties for $\bsig$ and $\bgamma$, we refer to \cite{arnold-2007}, whereas the approximation properties for $\bu$ and $p$ can be found in \cite[Corollary 1.128]{ernguermond}. We mention that for each of these approximation properties, the generic constant $C>0$ is independent of $h$. Finally, the main result of this section is given in the following theorem.
\begin{theorem}
\label{thm:a_priori_error}
Let the hypotheses of Theorem \ref{thm:wea} hold. Let $((\bsig,\bu,p),\bgamma)\in\bH\times\mathbf{Q}$ and $((\bsig_h,\bu_h,p_h),\bgamma_h)\in\bH_h\times\mathbf{Q}_h$ be the solutions of problems  \eqref{eq:weak-problem} and \eqref{eq:weak-problem-fem}, respectively. If $\bsig\in\boldsymbol{\mathcal{H}}^m(\O)$, $\bdiv\bsig\in\bH^m(\O)$, $p\in\H^m(\O)$ and $\bgamma\in\boldsymbol{\mathcal{H}}^m(\O)$ for $1\leq m\leq k+1$, there exists a constant $C>0$ independent of $h$ such that 
\begin{equation*}
\|((e_{\bsig},e_{\bu},e_p),e_{\bgamma})\|_{\bH\times\mathbf{Q}}
\leq Ch^m(\|\bsig\|_{m,\O}+\|\bdiv\bsig\|_{m,\O}+\|\bu\|_{m+1,\O}+\|p\|_{m+1,\O}+\|\bgamma\|_{m,\O}).
\end{equation*}
\end{theorem}
\begin{proof}
It follows immediately from the C\'ea estimate of Theorem \ref{thm:wea} and the approximation properties   \eqref{eq:approx-properties}.
\end{proof}
\section{A posteriori error analysis}\label{sec:a-posteriori}
For completeness of the analysis of the proposed numerical method, now we focus our effort on the design and analysis of an a posteriori error estimator of the residual type. 
The analysis is based on the work of \cite{KLBRV2026} focusing, for simplicity, on the two dimensional case. Nevertheless, the three dimensional analysis is possible to be performed under the spirit  of \cite{Caucao_ql_2023}.  Let us mention that a key result for  the reliability analysis is \cite[Lemma 3.9]{agr2015}, which states the stable Helmholtz decomposition that we need for the $\boldsymbol{\mathcal{H}}(\bdiv,\O)$ elements.

We begin this section with some useful notations and definitions, if $\psi$ represents a sufficiently smooth scalar field, $\bv:=(v_1,v_2)^{\texttt{t}}$ is a regular vector field and $\btau:=(\tau_{ij})\in\mathbb{R}^2$ is a sufficiently regular tensor field, we define the following differential operators
\begin{equation*}
	\text{curl}(\psi):=\left(\frac{\partial\psi}{\partial x_2},-\frac{\partial\psi}{\partial x_1}\right)^{\texttt{t}}, \quad\rot(\bv):=\frac{\partial v_2}{\partial x_1}-\frac{\partial v_1}{\partial x_2},
\quad 
	\curl\bv:=\begin{pmatrix}\text{curl} (v_1)^{\texttt{t}}\\\text{curl} (v_2)^{\texttt{t}}\ \end{pmatrix},\quad\underline{\curl}\btau:=\begin{pmatrix}\rot(\btau_1)\\ \rot(\btau_2)\end{pmatrix}.
\end{equation*}

In addition, we denote the set of edges on $\CT_h$ by $\mathcal{E}_h$. In order to distinguish those edges on the interior from those that lie on the boundary $\Gamma_D$ or $\Gamma_N$, or an element $K\in\CT_h$ we employ the notation 
$$\mathcal{E}_h(\star):=\{\ell\in\mathcal{E}_h\,:\; \ell\subseteq\star, \,\text{where}\,\,\star\in\{\Omega, \Gamma_N,\Gamma_D,K\}\}.$$ 

Given an arbitrary edge  $\ell$ of the mesh, we denote the outward normal component on this edge by $\boldsymbol{n}_\ell:=(n_1,n_2)^{\texttt{t}}$, whereas the tangential component is defined by $\boldsymbol{s}_\ell:=(-n_2,n_1)^{\texttt{t}}$. The internal jumps across interelements are defined in the usual way for vectors,
 tensors and scalar valued functions. 
 
The following inverse inequality  will be useful for the analysis (see \cite[Theorem 3.2.6]{MR0520174}).
\begin{lemma}\label{lmm_inv}
Let $\widetilde{m},m\in\mathbb{N}\cup\{0\}$ such that $\widetilde{m}\leq m$. Then, for each $K\in\mathcal{T}_h$ there holds
\begin{equation*}
|q|_{m,K}\leq C_{*} h_K^{\widetilde{m}-m}|q|_{\widetilde{m},K}\quad\forall q\in\mathrm{P}_k(K),
\end{equation*}
where the constant depends on $k,\widetilde{m},m$ and the shape regularity of the triangulations.
\end{lemma}

We now introduce the following residual-based a posteriori error estimator for the  mixed-primal Galerkin scheme:
\begin{equation}
\label{eq:estimator_global}
\Theta:=\left\{\sum_{K\in\CT_h}\Theta_{1,K}^2\right\}^{1/2}+\left\{\sum_{K\in\CT_h}\Theta_{2,K}^2\right\}^{1/2}+\left\{\sum_{K\in\CT_h}\Theta_{3,K}^2\right\}^{1/2},
\end{equation}
{where $\Theta_{1,K}^2$, $\Theta_{2,K}^2$ and $\Theta_{3,K}^2$ are defined, respectively, as follows}
\begin{flalign*}
&\Theta_{1,K}^2 := (1+\delta_1)^2(\|\bF- \mathcal{P}_h\bF\|_{0,K}^2
	+\|\mathcal{P}_h\bF+\bdiv\bsig_h\|_{0,K}^2)  + h_K^2\left\|\nabla\bu_h-\mathcal{C}^{-1}(\bsig_h)
-\dfrac{\alpha}{2(\lambda+\mu)} p_h\mathbf{I}
-\boldsymbol{\gamma}_h\right\|_{0,K}^2 &&\\
&\qquad + h_K^2\left\|\curl\left(\mathcal{C}^{-1}(\bsig_h)
+\dfrac{\alpha}{2(\lambda+\mu)} p_h\mathbf{I}
+\boldsymbol{\gamma}_h\right)\right\|_{0,K}^2
+\dfrac{1}{2}\|\bsig_h-\bsig_h^{\texttt{t}}\|_{0,K}^2 &&\\
&\qquad + \delta_1^2h_K^2\Bigg\{ (1+\alpha 2)^2
\|\mathcal{C}^{-1}\left(\mathcal{C}^{-1}(\bsig_h+\alpha p_h\mathbf{I})
-\varepsilon(\bu_h)\right)\|_{0,K}^2   + \|\curl\left(\mathcal{C}^{-1}\left(
\mathcal{C}^{-1}(\bsig_h+\alpha p_h\mathbf{I})
-\varepsilon(\bu_h)\right)\right)\|_{0,K}^2 &&\\
&\qquad + \|\bdiv\left(\mathcal{C}^{-1}(\bsig_h+\alpha p_h\mathbf{I})
-\varepsilon(\bu_h)\right)\|_{0,K} \Bigg\} &&\\
&\qquad + h_K^2\left\|g+\div(\kappa(\bu_h,p_h)\nabla p_h)
-\left(c_0+\dfrac{\alpha^2}{(\lambda+\mu)}\right)p_h
-\dfrac{\alpha}{2(\lambda+\mu)}\tr(\bsig_h)\right\|_{0,K}^2, &&
\end{flalign*}
\begin{flalign*}
&\Theta_{2,K}^2:=\sum_{\ell\in\partial K\cap \mathcal{E}_h(\O)}\delta_1^2h_{\ell}\left\|\jumpp{\left(\mathcal{C}^{-1}(\bsig_h)+\dfrac{\alpha}{2(\lambda+\mu)} p_h\mathbf{I}+\boldsymbol{\gamma}_h\right)\boldsymbol{s}_{\ell}}\right\|_{0,\ell}^2&&\\
&\qquad+\sum_{\ell\in\partial K\cap \mathcal{E}_h(\O)}\delta_1^2h_{\ell}\left\|\jump{\left(\mathcal{C}^{-1}\left(\mathcal{C}^{-1}(\bsig_h+\alpha p_h\mathbf{I})-\varepsilon(\bu_h)\right)\right)\boldsymbol{s}_{\ell}}\right\|_{0,\ell}^2&&\\
&\qquad+\sum_{\ell\in\partial K\cap \mathcal{E}_h(\O)}\delta_1^2h_{\ell}\left\|\jump{\left(\mathcal{C}^{-1}(\bsig_h+\alpha p_h\mathbf{I})-\varepsilon(\bu_h)\right)\bn_{\ell}}\right\|_{0,\ell}^2 +\sum_{\ell\in\partial K\cap \mathcal{E}_h(\O)}h_e\|\jump{\left(\kappa(\bu_h,p_h)\nabla p_h\right)\cdot\bn_{\ell}}\|_{0,\ell}^2,&&
\end{flalign*}
\begin{flalign*}
&\Theta_{3,K}^2:=\sum_{\ell\in\partial K\cap \mathcal{E}_h(\Gamma_{D})}\delta_1^2h_{\ell}\left\|\left(\mathcal{C}^{-1}(\bsig_h)+\dfrac{\alpha}{2(\lambda+\mu)} p_h\mathbf{I}+\boldsymbol{\gamma}_h\right)\boldsymbol{s}_{\ell}\right\|_{0,\ell}^2&&\\
&\qquad+\sum_{\ell\in\partial K\cap \mathcal{E}_h(\Gamma_{D})}\delta_1^2h_{\ell}\|\left(\mathcal{C}^{-1}\left(\mathcal{C}^{-1}(\bsig_h+\alpha p_h\mathbf{I})-\varepsilon(\bu_h)\right)\right)\boldsymbol{s}_{\ell}\|_{0,\ell}^2&&\\
&\qquad+\sum_{\ell\in\partial K\cap \mathcal{E}_h(\Gamma_{N})}\delta_1^2h_{\ell}\|\left(\mathcal{C}^{-1}(\bsig_h+\alpha p_h\mathbf{I})-\varepsilon(\bu_h)\right)\bn_{\ell}\|_{0,\ell}^2&&\\
&\qquad+\sum_{\ell\in\partial K\cap \mathcal{E}_h(\Gamma)}h_e\|\left(\kappa(\bu_h,p_h)\nabla p_h\right)\cdot\bn_{\ell}-z_{\Gamma}\|_{0,\ell}^2+\sum_{\ell\in\partial K\cap \mathcal{E}_h(\Gamma_N)}h_{\ell}\|\bu_h\|_{0,\ell}^2.&&
\end{flalign*}

 \subsection{Reliability}
We begin by introducing the Raviart--Thomas interpolator $\Pi_h$ and the Cl\'ement operator $I_h$ whose properties can be found in \cite[Section 3]{agr2015}. In what follows, their respective tensor and vector generalizations are denoted by $\boldsymbol{\Pi}_h$ (defined row by row through $\Pi_h$) and $\boldsymbol{\mathcal{I}}_h$ (defined componentwise through $I_h$).

Following the notation introduced in Subsection~\ref{sec:apriori}, the approximation errors associated with our solution are given by:
$\texttt{e}_{\bsig}:=\bsig-\bsig_h\in \bbH_{\Gamma_\rN}(\bdiv,\O)$, $\texttt{e}_{\bu}:=\bu-\bu_h\in\boldsymbol{\H}^1_{\Gamma_\rD}(\Omega)$, $e_{p}:=p-p_h\in\H^1(\O)$ and $\texttt{e}_{\bgamma}:=\bgamma-\bgamma_{h}\in\bbL^2_{\text{skew}}(\O).$ 

We begin the analysis with the following inf-sup condition for $A_{\bu, p}(\cdot,\cdot)$.
 \begin{lemma}\label{lemma_conf1}
The following estimate holds true
\begin{flalign*} 
&\|(\texttt{e}_{\bsig},\texttt{e}_{\bu},e_{p}),\texttt{e}_{\bgamma})\|_{\mathbf{H}\times\mathbf{Q}}&&\\
&\quad \leq \rho \sup_{\bzero\neq((\btau,\bv,q),\bbeta)\in 
 \bH\times\mathbf{Q}} \frac{A_{\bu, p}( (\texttt{e}_{\bsig},\texttt{e}_{\bu},e_{p},\texttt{e}_{\bgamma}) , (\btau-\btau_h,\bv-\bv_h,q-q_h,\bbeta-\bbeta_h)) }{\|((\btau,\bv,q),\bbeta)\|_{\bH\times\mathbf{Q}}}  +\rho\widetilde{\rho}\kappa_2\|p-p_h\|_{0,\O}\|F\|,
\end{flalign*} 
where the constant $\rho>0$ is defined in \eqref{eq:def-rho}.
 \end{lemma}
 \begin{proof}
 Invoking the inf-sup condition given in \eqref{eq:global-inf-sup} and applying it to the errors defined above, we obtain
\begin{flalign*}
&\dfrac{1}{\rho}\|(({\texttt{e}_{\bsig}},\texttt{e}_{\bu},e_{p}),\texttt{e}_{\bgamma})\|_{\mathbf{H}\times\mathbf{Q}}&&\\
&\quad\leq \underbrace{\sup_{\bzero\neq((\btau,\bv,q),\bbeta)\in 
 \bH\times\mathbf{Q}} \frac{A_{\bu, p}(({\texttt{e}_{\bsig}},\texttt{e}_{\bu},e_{p},\texttt{e}_{\bgamma}),(\btau-\btau_h,\bv-\bv_h,q-q_h,\bbeta-\bbeta_h)) }{\|((\btau,\bv,q),\bbeta)\|_{\bH\times\mathbf{Q}}}}_{T_{1}}&&\\
&\qquad +\sup_{\bzero\neq((\btau,\bv,q),\bbeta)\in 
 \bH\times\mathbf{Q}} \frac{A_{\bu, p}(({\texttt{e}_{\bsig}},\texttt{e}_{\bu},e_{p},\texttt{e}_{\bgamma}),{(\btau,\bv_h,q_h,\bbeta_h)}) }{\|((\btau,\bv,q),\bbeta)\|_{\bH\times\mathbf{Q}}}&&\\
&\quad =T_1 +\sup_{\bzero\neq((\btau,\bv,q),\bbeta)\in 
 \bH\times\mathbf{Q}} \frac{\ds\int_\O \left(\kappa(\bu_h,p_h)-\kappa(\bu,p)\right)\nabla p_h\cdot\nabla q_h}{\|((\btau,\bv,q),\bbeta)\|_{\bH\times\mathbf{Q}}}&&\\
&\quad \leq T_1+\sup_{\bzero\neq((\btau,\bv,q),\bbeta)\in 
 \bH\times\mathbf{Q}} \frac{\kappa_2\widetilde{\rho}\|F\|\|p-p_h\|_{1,\O}\|\nabla q_h\|_{0,\O}}{\|((\btau,\bv,q),\bbeta)\|_{\bH\times\mathbf{Q}}},
\end{flalign*} 
where for the last estimate we have used the Lipschitz continuous \eqref{prop-kappa} and  \eqref{eq:stability-weak-h}. Now taking $q_h:=I_{h}q$, and using the stability of the Clément interpolator in the $\H^1$-seminorm, we obtain the following estimate
\begin{equation*} 
\|((\texttt{e}_{\bsig},\texttt{e}_{\bu},e_{p}),\texttt{e}_{\bgamma})\|_{\mathbf{H}\times\mathbf{Q}}\leq \rho(T_1+\widetilde{\rho}\kappa_2\|F\|\|p-p_h\|_{1,\O}).
\end{equation*}
This concludes the proof.
 \end{proof}
 
Now the task is to estimate the supremum term  $T_1$ defined on the proof of the  previous lemma. With this aim in mind, we group the terms suitably and, using the property of the supremum, we obtain
\begin{flalign}\label{eq:inf-sup_ap}
&\sup_{\bzero\neq((\btau,\bv,q),\bbeta)\in 
 \bH\times\mathbf{Q}} \frac{A_{\bu, p}(({\texttt{e}_{\bsig}},\texttt{e}_{\bu},e_{p},\texttt{e}_{\bgamma}),(\btau-\btau_h,\bv-\bv_h,q-q_h,\bbeta-\bbeta_h)) }{\|((\btau,\bv,q),\bbeta)\|_{\bH\times\mathbf{Q}}}&&\\
&\quad \leq \sup_{\bzero\neq\btau\in \bbH_{\Gamma_\rN}(\bdiv,\O)}\dfrac{|\mathcal{R}_{1}(\btau-\btau_h)|}{\|\btau\|_{\bdiv,\O}}+\sup_{\bzero\neq\bv\in \boldsymbol{\H}^1_{\Gamma_\rD}(\Omega)}\dfrac{|\mathcal{R}_2(\bv-\bv_h)|}{\|\bv\|_{1,\O}} +\sup_{0\neq q\in\H^1(\O)}\dfrac{|\mathcal{R}_3(q-q_h)|}{\|q\|_{1,\O}}+\sup_{\bzero\neq\bbeta\in\mathbf{Q}}\dfrac{|\mathcal{R}_4(\bbeta-\bbeta_h)|}{\|\bbeta\|_{0,\O}},\nonumber&&
\end{flalign}
where the terms $\mathcal{R}_i(\cdot)$, with $i=\{1,2,3,,4\}$, are defined as follows:
\begin{multline*}
\mathcal{R}_{1}(\btau-\btau_h):=-\int_{\O}\bu_h\cdot\bdiv(\btau-\btau_h)-\int_{\O}\left(\mathcal{C}^{-1}\bsig_h+\dfrac{\alpha}{2(\lambda+\mu)}p_h\mathbf{I}+\bgamma_h\right):(\btau-\btau_h)\\
-\delta_1\int_{\O}\left(\mathcal{C}^{-1}(\bsig_h+\alpha p_h\mathbf{I})-\varepsilon(\bu_h)\right):\mathcal{C}^{-1}(\btau-\btau_h)-\delta_2\int_\O(\bF+\bdiv\bsig_h)\cdot\bdiv(\btau-\btau_h).
\end{multline*} 
The second term, $\mathcal{R}_{2}(\cdot)$, is given by
\begin{equation*}
\mathcal{R}_2(\bv-\bv_h):=-\int_\O(\bF+\bdiv\bsig_h)\cdot\bdiv(\bv-\bv_h) 
+\delta_1\int_{\O}\left(\mathcal{C}^{-1}(\bsig_h+\alpha p_h\mathbf{I})-\varepsilon(\bu_h)\right):\varepsilon(\bv-\bv_h).
\end{equation*} 
The third term, $\mathcal{R}_{3}(\cdot)$, is defined as follows
\begin{multline*}
\mathcal{R}_3(q-q_h):=\int_{\O}\left(g-\left(c_0+\dfrac{\alpha^2 }{(\lambda+\mu)}\right)p_h-\dfrac{\alpha}{2(\lambda+\mu)}\tr{\bsig_h}\right)(q-q_h)\\
-\delta_1\int_{\O}\left(\mathcal{C}^{-1}(\bsig_h+\alpha p_h\mathbf{I})-\varepsilon(\bu_h)\right):\mathcal{C}^{-1}(\alpha(q-q_h)\mathbf{I})-\int_{\O}\kappa(\bu,p)\nabla p_h\cdot\nabla(q-q_h).
\end{multline*} 
Finally we define $\mathcal{R}_4(\cdot)$
\begin{equation*}
\mathcal{R}_4(\bbeta-\bbeta_h):=-\int_{\O}\bsig_h:(\bbeta-\bbeta_h)=\int_{\O}\left(\dfrac{\bsig_h^{\texttt{t}}-\bsig_h}{2}\right):(\bbeta-\bbeta_h).
\end{equation*} 

The next task is to estimate the terms on the right-hand side of \eqref{eq:inf-sup_ap} with respect to the estimator \eqref{eq:estimator_global}. To do this, we first need the following Helmholtz decomposition proved in \cite[Lemma 3.9]{alvarez16}.
 \begin{lemma}\label{lemma_Mario2026}
 Assume that there exist a convex domain $B$ such that $\O\subset B$ and $\Gamma_\rN\subset\partial B$. Let $\btau\in\bbH_{\Gamma_\rN}(\bdiv,\O)$. Then, there exist $\boldsymbol{\xi}\in\boldsymbol{\H}^1_{\Gamma_\rD}(\Omega)$ and $\boldsymbol{\chi}\in \boldsymbol{\H}^1_{\Gamma_\rN}(\Omega)$ such that 
 $$\btau=\boldsymbol{\xi}+\curl{\boldsymbol{\chi}}\quad\text{and}\quad \|\boldsymbol{\xi}\|_{1,\O}+\|\boldsymbol{\chi}\|_{1,\O}\leq C_{\texttt{Helm}}\|\btau\|_{\bdiv,\O},$$
where the constant $ C_{\texttt{Helm}}>0$ is independent of $\btau$.
 \end{lemma}
 
 {Now the task is to obtain an upper bound for each $\mathcal{R}_i(\cdot,\cdot)$, with $i\in\{1,2,3,4\}$, depending on the estimator $\Theta$. The following results provide  the desire bounds.}
 \begin{lemma}\label{lemma_con2}
 Under the same hypotheses of Lemma \ref{lemma_Mario2026}, there exists a constant $C>0$
 such that
 $$ \sup_{\bzero\neq\btau\in \bbH_{\Gamma_\rN}(\bdiv,\O)}\dfrac{|\mathcal{R}_{1}(\btau-\btau_h)|}{\|\btau\|_{\bdiv,\O}}\leq C\Theta.$$
 \end{lemma}
 \begin{proof}
 Since $\boldsymbol{\tau}_h \in \bbH^{\bsig}_h $ is arbitrary, we choose $\boldsymbol{\tau}_h := \boldsymbol{\Pi}_h\boldsymbol{\xi} + \operatorname{curl}(\boldsymbol{\mathcal{I}}_{h}\boldsymbol{\chi})$, defining $\widehat{\boldsymbol{\xi}} := \boldsymbol{\xi}-\boldsymbol{\Pi}_h\boldsymbol{\xi}$ and $\widehat{\boldsymbol{\chi}} := \boldsymbol{\chi}-\boldsymbol{\mathcal{I}}_{h}\boldsymbol{\chi}$. Consequently, the term $\mathcal{R}_{1}(\boldsymbol{\tau}-\boldsymbol{\tau}_h)$ can be rewritten as follows:
\begin{flalign*}
&\mathcal{R}_{1}(\btau-\btau_h) = -\int_{\O}\bu_h\cdot\bdiv\widehat{\boldsymbol{\xi}}-\int_{\O} \left(\mathcal{C}^{-1}\bsig_h + \dfrac{\alpha}{2(\lambda+\mu)}p_h\mathbf{I} + \bgamma_h\right):\widehat{\boldsymbol{\xi}}&&\\
&\quad -\int_{\O}\left(\mathcal{C}^{-1}\bsig_h + \dfrac{\alpha}{2(\lambda+\mu)}p_h\mathbf{I}+\bgamma_h\right):\curl\widehat{\boldsymbol{\chi}}  -\delta_1\int_{\O} \mathcal{C}^{-1}\left(\mathcal{C}^{-1}(\bsig_h + \alpha p_h\mathbf{I})-\varepsilon(\bu_h)\right) :\widehat{\boldsymbol{\xi}}&&\\
&\quad -\delta_1\int_{\O}\mathcal{C}^{-1}\left( \mathcal{C}^{-1}(\bsig_h+\alpha p_h\mathbf{I})-\varepsilon(\bu_h)\right) :\curl\widehat{\boldsymbol{\chi}}
-\delta_2\int_\O(\bF+\bdiv\bsig_h) \cdot\bdiv\widehat{\boldsymbol{\xi}}.&&
\end{flalign*}
Applying local integration by parts and taking into account that $\bu_h\in\boldsymbol{\H}^1_{\Gamma_\rD}(\Omega)$, $(\bxi-\boldsymbol{\Pi}_{h}\bxi_h)\in\boldsymbol{\H}^1_{\Gamma_\rD}(\Omega)$ and $\boldsymbol{\chi}-\boldsymbol{\mathcal{I}}_{h}\boldsymbol{\chi}_h\in\boldsymbol{\H}^1_{\Gamma_\rN}(\Omega)$ 
\begin{flalign*}
&\mathcal{R}_{1}(\btau-\btau_h)=-\delta_1\int_{\O}\mathcal{C}^{-1}\left(\mathcal{C}^{-1}(\bsig_h+\alpha p_h\mathbf{I})-\varepsilon(\bu_h)\right):\widehat{\boldsymbol{\xi}}&&\\
&\quad -\delta_2\int_\O(\bF+\bdiv\bsig_h)\cdot\bdiv\widehat{\boldsymbol{\xi}}+\sum_{K\in\CT_h}\int_{K}\left(\nabla\bu_h-\mathcal{C}^{-1}\bsig_h-\dfrac{\alpha}{2(\lambda+\mu)}p_h\mathbf{I}-\bgamma_h\right):\widehat{\boldsymbol{\xi}}&&\\
&\quad-\sum_{\ell\in\partial K\cap \mathcal{E}_h(\Gamma_N)}\int_{\ell}\bu_h\cdot\widehat{\boldsymbol{\xi}}\n-\sum_{K\in\CT_h}\int_K\underline{\curl}\left(\mathcal{C}^{-1}\bsig_h+\dfrac{\alpha}{2(\lambda+\mu)}p_h\mathbf{I}+\bgamma_h\right)\cdot\widehat{\boldsymbol{\chi}}&&\\
&\quad-\sum_{\ell\in\mathcal{E}_h(\O)}\jumpp{\left(\mathcal{C}^{-1}(\bsig_h)+\dfrac{\alpha}{2(\lambda+\mu)} p\mathbf{I}+\boldsymbol{\gamma}\right)\boldsymbol{s}_{\ell}}\cdot\widehat{\boldsymbol{\chi}} -\sum_{\ell\in\mathcal{E}_h(\Gamma_{\rD})}\left(\mathcal{C}^{-1}(\bsig_h)+\dfrac{\alpha}{2(\lambda+\mu)} p\mathbf{I}+\boldsymbol{\gamma}\right)\boldsymbol{s}_{\ell}\cdot\widehat{\boldsymbol{\chi}}&&\\
&\quad-\delta_1\sum_{K\in\CT_h}\int_K\underline{\curl}\left(\mathcal{C}^{-1}\left(\mathcal{C}^{-1}(\bsig_h+\alpha p_h\mathbf{I})-\varepsilon(\bu_h)\right)\right)\cdot\widehat{\boldsymbol{\chi}} -\delta_1\sum_{\ell\in\mathcal{E}_h(\O)}\jump{\left(\mathcal{C}^{-1}\left(\mathcal{C}^{-1}(\bsig_h+\alpha p_h\mathbf{I})-\varepsilon(\bu_h)\right)\right)\boldsymbol{s}_{\ell}}\cdot\widehat{\boldsymbol{\chi}}&&\\
&\quad-\delta_1\sum_{\ell\in\mathcal{E}_h(\Gamma_{\rD})}\left(\mathcal{C}^{-1}\left(\mathcal{C}^{-1}(\bsig_h+\alpha p_h\mathbf{I})-\varepsilon(\bu_h)\right)\right)\boldsymbol{s}_{\ell}\cdot\widehat{\boldsymbol{\chi}}.&&
\end{flalign*}
Finally, by applying the Cauchy--Schwarz inequality un conjunction with the approximation properties of the Raviart--Thomas and Cl\'ement interpolators, together with the trace property and Lemma \ref{lemma_Mario2026}, we obtain 
 $$ \sup_{\bzero\neq\btau\in \bbH_{\Gamma_\rN}(\bdiv,\O)}\dfrac{|\mathcal{R}_{1}(\btau-\btau_h)|}{\|\btau\|_{\bdiv,\O}}\leq C\Theta.$$
This concludes the proof.
 \end{proof}
 \begin{lemma} There exists a constant $C>0$
 such that
$$ \sup_{\bzero\neq\bv\in \boldsymbol{\H}^1_{\Gamma_\rD}(\Omega)}\dfrac{|\mathcal{R}_2(\bv-\bv_h)|}{\|\bv\|_{1,\O}}\leq C\Theta$$
 \end{lemma}
 \begin{proof}
 Proceeding analogously to the previous proof, we first set $\bv_h=\boldsymbol{\mathcal{I}}_{h}\bv\in\boldsymbol{\H}^1_{\Gamma_\rD}(\Omega)$ {and then  define} $\widehat{\bv}:=\bv-\boldsymbol{\mathcal{I}}_{h}\bv$. {Hence, an integration by parts reveals that}
\begin{flalign*}
&  \mathcal{R}_2(\bv-\bv_h):=-\int_\O(\bF+\bdiv\bsig_h)\cdot\bdiv\widehat{\bv} -\delta_1\sum_{K\in\CT_h}\int_{K}\bdiv\left(\mathcal{C}^{-1}(\bsig_h+\alpha p_h\mathbf{I})-\varepsilon(\bu_h)\right)\cdot\widehat{\bv}&&\\
&\qquad+\delta_1\sum_{\ell\in\mathcal{E}_h(\O)}\int_{\ell}\jump{\left(\mathcal{C}^{-1}(\bsig_h+\alpha p_h\mathbf{I})-\varepsilon(\bu_h)\right)\bn_{\ell}}\cdot \widehat{\bv} +\delta_1\sum_{\ell\in\mathcal{E}_h(\Gamma_{\rN})}\left(\mathcal{C}^{-1}(\bsig_h+\alpha p_h\mathbf{I})-\varepsilon(\bu_h)\right)\bn_{\ell}\cdot \widehat{\bv}.&&
\end{flalign*}
Thus, combining  Cauchy--Schwarz inequality, the approximation properties  of the  Cl\'ement interpolant together with the trace property and Lemma \ref{lemma_Mario2026}, we conclude the proof.
 \end{proof}
 \begin{lemma}\label{lemma_con3}There exists a constant $C>0$
 such that
 $$\sup_{0\neq q\in\H^1(\O)}\dfrac{|\mathcal{R}_3(q-q_h)|}{\|q\|_{1,\O}}\leq C\Theta+\kappa_2\widetilde{\rho}\|F\|\|p-p_h\|_{0,\O}.$$
 \end{lemma}
 \begin{proof}
{Let us set $q_h=\mathcal{I}_{h}q\in\H^1_{\Gamma}(\Omega)$. Then,  defining  $\widehat{q}:=q-q_h$  and integrating  by parts,} we obtain:
\begin{flalign*}
&\mathcal{R}_3(q-q_h):=\int_{\O}\left(g-\left(c_0+\dfrac{\alpha^2 }{\lambda+\mu}\right)p_h-\dfrac{\alpha}{2(\lambda+\mu)}\tr({\bsig_h})\right)\widehat{q}&&\\
&\qquad-\delta_1\int_{\O}\mathcal{C}^{-1}\left(\mathcal{C}^{-1}(\bsig_h+\alpha p_h\mathbf{I})-\varepsilon(\bu_h)\right):(\alpha\widehat{q}\;\  \mathbf{I}) +\int_{\O}(\kappa(\bu_h,p_h)-\kappa(\bu,p))\nabla p_h\cdot\nabla\widehat{q}-\int_{\O}\kappa(\bu_h,p_h)\nabla p_h\cdot\nabla\widehat{q}&&\\
&\quad =\sum_{K\in\CT_h}\int_{K}\left(g+\bdiv(\kappa(\bu_h,p_h)\nabla p_h)-\left(c_0+\dfrac{\alpha^2 }{\lambda+\mu}\right)p_h-\dfrac{\alpha}{2(\lambda+\mu)}\tr({\bsig_h})\right)\widehat{q}&&\\
&\qquad -\sum_{\ell\in\cap \mathcal{E}_h(\O)}\jump{\left(\kappa(\bu_h,p_h)\nabla p_h\right)\cdot\bn_{\ell}}\widehat{q}
-\sum_{\ell\in\cap \mathcal{E}_h(\Gamma)}(\left(\kappa(\bu_h,p_h)\nabla p_h\right)\cdot\bn_{\ell}-z_{\Gamma})\widehat{q}&&\\
&\qquad -\delta_1\int_{\O}\mathcal{C}^{-1}\left(\mathcal{C}^{-1}(\bsig_h+\alpha p_h\mathbf{I})-\varepsilon(\bu_h)\right):(\alpha\widehat{q}\;\ \mathbf{I}) +\int_{\O}(\kappa(\bu_h,p_h)-\kappa(\bu,p))\nabla p_h\cdot\nabla\widehat{q}.&&
\end{flalign*}
Combining the Cauchy--Schwarz inequality with the approximation property and stability of the Cl\'ement interpolant, together with Lipschitz condition \eqref{prop-kappa}, the continuous dependence on the data established in \eqref{eq:stability-weak-h} and the trace inequality (cf. \cite[Theorem 1.4]{{gatica14}}), we obtain:
$$\sup_{0\neq q\in\H^1(\O)}\dfrac{|\mathcal{R}_3(q-q_h)|}{\|q\|_{1,\O}}\leq C\Theta+\kappa_2\widetilde{\rho}\|F\|\|p-p_h\|_{1,\O}.$$
This concludes the proof.
 \end{proof}
 
Now, to estimate  the term associated with  $\mathcal{R}_4(\cdot)$, we define $\bbeta_h$  as the $\bbL^2$ projection of $\bbeta$, which is stable in $\bbL^2(\O)$. Then, we have:
 \begin{equation}
 \label{eq_cotaR4}
 \sup_{\bzero\neq\bbeta\in\mathbf{Q}}\dfrac{|\mathcal{R}_4(\bbeta-\bbeta_h)|}{\|\bbeta\|_{0,\O}}\leq C\Theta.
 \end{equation}
 We now state the main reliability bound for the proposed estimator.
 \begin{corollary} 
Assume that the  hypotheses stated in Lemma \ref{lemma_con2} hold. Let $((\bsig,\bu,p),\bgamma)\in \bH \times \mathbf{Q}$ and $((\bsig_h,\bu_h,p_h),\bgamma_h)\in \bH_h\times\mathbf{Q}_h$  solutions of \eqref{eq:weak-problem} and \eqref{eq:linear-problem-disc}, respectively. Assume  further that 
\begin{equation}\label{eq_sup_final}
\kappa_2\rho\widetilde{\rho}\left(\|\bF\|_{0,\Omega} + \|g\|_{0,\Omega} + \|z_{\Gamma}\|_{-1/2,\Gamma}\right)\leq \dfrac{1}{4}.
\end{equation}
Then, the following estimate hods 
$$\|(\texttt{e}_{\bsig},\texttt{e}_{\bu},e_{p}),\texttt{e}_{\bgamma})\|_{\mathbf{H}\times\mathbf{Q}}\\
\leq C_{\text{rel}}\Theta.$$
 \end{corollary}
 \begin{proof}
 The proof is obtained by combining Lemma \ref{lemma_conf1}, estimate \eqref{eq:inf-sup_ap}, Lemmas \ref{lemma_con2}--\ref{lemma_con3}, together with estimate \eqref{eq_cotaR4} and assumption \eqref{eq_sup_final}.
 \end{proof}

 \subsection{Efficiency}
 We begin by introducing the bubble functions for two dimensional elements. Given $K\in\mathcal{T}_h$ and $e\in\mathcal{E}(K)$, we let $\psi_K$ and $\psi_\ell$ be the usual triangle-bubble and edge-bubble functions, respectively (see \cite{MR3059294} for further details about these functions), which satisfy the following properties
\begin{enumerate}
\item $\psi_K\in\mathrm{P}_{3}(K)$,  $\text{supp}(\psi_K)\subset K$, $\psi_K=0$ on $\partial K$ and $0\leq\psi_K\leq 1$ in $K$;
\item $\psi_\ell|_K\in\mathrm{P}_{2}(K)$, $\text{supp}(\psi_\ell)\subset \omega_\ell:=\cup\{K'\in\mathcal{K}_h\,:\, e\in\mathcal{E}(K')\}$, $\psi_\ell=0$ on $\partial K\setminus \ell$ and $0\leq\psi_\ell\leq 1$ in $\omega_\ell$.
\end{enumerate}

The following properties, proved in \cite[Lemma 1.3]{MR1284252} for an arbitrary polynomial order of approximation, hold.
\begin{lemma}[Bubble function properties]
\label{lmm:bubble_estimates}
Given $\widetilde{m}\in\mathbb{N}\cup\{0\}$, and for each $K\in\CT_h$ and $e\in\mathcal{E}(K)$, there hold
\begin{equation*}
\|\psi_K q\|_{0,K}^2\leq \|q\|_{0,K}^2\leq C_1 \|\psi_K^{1/2} q\|_{0,K}^2\quad\forall q\in\mathrm{P}_{\widetilde{m}}(K),
\end{equation*}
\begin{equation*}
\|\psi_\ell L(p)\|_{0,\ell}^2\leq \| p\|_{0,\ell}^2\leq C_2 \|\psi_\ell^{1/2} p\|_{0,\ell}^2\quad\forall p\in\mathrm{P}_{\widetilde{m}}(\ell),
\end{equation*}
and 
\begin{equation*}
h_e\|p \|_{0,\ell}^2\leq C_3 \|\psi_\ell^{1/2} L(p)\|_{0,K}^2\leq C_4
 h_\ell\|p\|_{0,\ell}^2\quad\forall p\in\mathrm{P}_{\widetilde{m}}(\ell),
 \end{equation*}
 where $L$ is the extension operator defined by  $L: \mathcal{C}(\ell)\rightarrow \mathcal{C}(K)$ with $\mathcal{C}(\ell)$ and $\mathcal{C}(K)$ being the spaces of continuous functions defined on $\ell$ and $K$, respectively,  and satisfying $L(p)\in\mathrm{P}_k(K)$ and $L(p)|_\ell=p$ for all $p\in\mathrm{P}_k(\ell)$, where  each  constant depends on $k$ and the shape regularity of the triangulation.
 \end{lemma}
 
As usual, to prove efficiency it is necessary to control all the terms of the estimator in terms of the error. To this end, most of the terms are bounded using standard arguments based on the localization of bubble functions, so we omit these results and refer the reader to \cite{MR3453481,GATICA2022114144,KLBRV2026}, where such bounds are detailed. In this work, we will focus solely on bounding the terms containing the nonlinear term $\kappa(\cdot,\cdot)$. 

To prove the following result, we assume that $\kappa(\bu_h,p_h)$  is piecewise polynomial.

\begin{lemma} Let  $\kappa(\bu_h,p_h)$ be  piecewise polynomial. Let $\mathcal{P}_h$ be the $\L^2(K)$- orthogonal projector onto $\textrm{P}_{k}$. Then for all $K\in\CT_h$ and $\ell\in\mathcal{E}_h(\O)$ there holds
\begin{subequations}
\begin{align}\label{eq_estimate1_eff}
 h_K\|R_g\|_{0,K}&\leq  C(\|\bsig-\bsig_h\|_{\bdiv,K}+\|\bu-\bu_h\|_{1,K}+\|p-p_h\|_{1,K}+h_K\|g-\mathcal{P}_h g\|_{0,K}),\\\label{eq_estimate2_eff}
 h_e^{1/2}\|\jump{(\kappa(\bu_h,p_h)\nabla p_h)\cdot\bn_{\ell}}\|_{0,\ell}&\leq  C(\|\bsig-\bsig_h\|_{\bdiv,\omega_{\ell}}+\|\bu-\bu_h\|_{1,\omega_{\ell}} +\|p-p_h\|_{1,\omega_{\ell}}+h_K\|g-\mathcal{P}_h g\|_{0,\omega_{\ell}}).
\end{align}\end{subequations}
where 
$$R_g:=g+\div(\kappa(\bu_h,p_h)\nabla p_h)-\left(c_0+\dfrac{\alpha^2}{(\lambda+\mu)}\right)p_h-\dfrac{\alpha}{2(\lambda+\mu)}\tr(\bsig_h),$$
and $\omega_{\ell}:=\bigcup\{K'\in \CT_h:\ell\in\mathcal{E}_h(K')\}.$
\end{lemma}
\begin{proof}
First, we define $$R_{g_{h}}:= \mathcal{P}_h g+\div(\kappa(\bu_h,p_h)\nabla p_h)-\left(c_0+\dfrac{\alpha^2}{(\lambda+\mu)}\right)p_h-\dfrac{\alpha}{2(\lambda+\mu)}\tr(\bsig_h).$$ 
Consequently we have that
\begin{equation}\label{eq_1_eff}
h_K\| R_g\|_{0,K}\leq h_K(\|g-\mathcal{P}_h g|\|_{0,K}+\| R_{g_{h}}\|_{0,K}).
\end{equation}
We now proceed to bound $\|R_{g_{h}}\|_{0,K}$. To this end, we define $q_K:=R_{g_{h}}\psi_K$ where $\psi_K$ is an interior bubble satisfying the properties given in Lemma \ref{lmm:bubble_estimates}. Then we have
\begin{equation}
\label{eq:gonzalo}
\dfrac{1}{C_1}\|R_{g_{h}}\|_{0,K}^2\leq \int_KR_{g_{h}}q_K=\int_{K}(\mathcal{P}_h g-g)q_K+\int_{K}R_gq_K\leq \|\mathcal{P}_h g-g\|_{0,K}\|R_{g_{h}}\|_{0,K}+\int_{K}R_gq_K.
\end{equation}
Using \eqref{eq:mass-1}, together with integration by parts and  adding and subtracting $\kappa(\bu_h,p_h)\nabla p$, we obtain that
\begin{flalign}\label{eq_finaleff}
&\int_{K}R_gq_K=\int_K\Big(c_0+\frac{\alpha^2}{\lambda+\mu}\Big)(p-p_h)q_K+\int_K\frac{\alpha}{2(\lambda+\mu)}\tr(\bsig-\bsig_h)q_K&& \nonumber\\
&\qquad +\int_K\kappa(\bu_h,p_h)(\nabla p-\nabla p_h)\cdot\nabla q_{K}+\int_K(\kappa(\bu,p)-\kappa(\bu_h,p_h))\nabla p\cdot\nabla q_{K}&&\\ 
& \leq C(\|p-p_h\|_{1,K}+\|\bsig-\bsig_h\|_{0,K})\|q_{K}\|_{0,K}+C_{*}(\kappa_{2}+\kappa_{2}\rho\|F\|)\|p-p_h\|_{1,K}\| q_K\|_{0,K} &&\nonumber\\
&\leq C(\|p-p_h\|_{1,K}+\|\bsig-\bsig_h\|_{0,K}+h_K^{-1}(\kappa_{2}+\kappa_{2}\rho\|F\|)\|p-p_h\|_{1,K})\|R_{g_{h}}\|_{0,K}, \nonumber&&
\end{flalign}
where for the last estimate we have used Lemmas \ref{lmm:bubble_estimates} and \ref{lmm_inv}. Then, estimate \eqref{eq_estimate1_eff} is concluded using \eqref{eq_1_eff}, \eqref{eq:gonzalo}, and \eqref{eq_finaleff}.

We now proceed to prove the second estimate of the lemma. To this end, we define $\varUpsilon_h:=\jump{\kappa(\bu_h,p_h)\nabla p_h\cdot\n}$. Then, by Lemma \ref{lmm:bubble_estimates}, we have that
\begin{equation*}
\dfrac{1}{C_{3}}\|\varUpsilon_h\|_{0,\ell}^2\leq\int_{\ell}\psi_{\ell}\varUpsilon_h\varUpsilon_h= \int_{\partial \omega_{\ell}}\psi_{\ell}L(\varUpsilon_h)(\jump{(\kappa(\bu_h,p_h)\nabla p_h-\kappa(\bu,p)\nabla p)\cdot\n},
\end{equation*}
where $L$ is the extension operator introduced in Lemma \ref{lmm:bubble_estimates}. Note that for the last estimate above, we have used that $\div(\kappa(\bu,p)\nabla p)\in \L^2(\O)$, and therefore its normal jumps vanish. Integrating by parts the last term on the right-hand side of the previous estimate and adding and subtracting $\kappa(\bu,p)\nabla p_h$, we obtain
\begin{flalign*}
&\dfrac{1}{C_{3}}\|\varUpsilon_h\|_{0,\ell}^2\leq\int_{\omega_{\ell}}\div(\kappa(\bu_h,p_h)\nabla p_h-\kappa(\bu,p)\nabla p)\psi_{\ell}L(\varUpsilon_h) +\int_{\omega_{\ell}}(\kappa(\bu_h,p_h)\nabla p_h-\kappa(\bu,p)\nabla p)\cdot\nabla(\psi_{\ell}L(\varUpsilon_h))&&\\
&\quad \leq \|\div(\kappa(\bu_h,p_h)\nabla p_h-\kappa(\bu,p)\nabla p)\|_{0,\omega_{\ell}}\|\psi_{\ell}L(\varUpsilon_h)\|_{0,\omega_{\ell}}  +(\kappa_2\widetilde{\rho}\|F\|+\kappa_{2})\|p-p_h\|_{1,\omega_{\ell}}\|\nabla(\psi_{\ell}L(\varUpsilon_h))\|_{0,\omega_{\ell}}&&\\
&\quad \leq \dfrac{C_4}{C_3}\|\div(\kappa(\bu_h,p_h)\nabla p_h-\kappa(\bu,p)\nabla p)\|_{0,\omega_{\ell}}\|\varUpsilon_h\|_{0,\ell} +C_{*}(\kappa_2\widetilde{\rho}\|F\|+\kappa_{2})\|p-p_h\|_{1,\omega_{\ell}}\|\varUpsilon_h\|_{0,\ell}&&\\
&\quad \leq C(h_e^{1/2}\|\div(\kappa(\bu_h,p_h)\nabla p_h-\kappa(\bu,p)\nabla p)\|_{0,\omega_{\ell}}+h_e^{-1/2}\|p-p_h\|_{1,\omega_{\ell}})\|\varUpsilon_h\|_{0,\ell},&&
\end{flalign*}
where for that last estimate we used Lemmas \ref{lmm:bubble_estimates} and \ref{lmm_inv}. Then used that $h_e\leq h_{\omega_{\ell}}$ we have
\begin{equation}\label{eq_lado_int}
h_e^{1/2}\|\varUpsilon_h\|_{0,\ell}\leq C (h_{\omega_{\ell}}\|\div(\kappa(\bu_h,p_h)\nabla p_h-\kappa(\bu,p)\nabla p)\|_{0,\omega_{\ell}}+\|p-p_h\|_{1,\omega_{\ell}}).
\end{equation}
The next step is to bound the first term on the right-hand side of the above estimate. For this purpose, using the definition of $R_g$ and \eqref{eq:mass-1}, we obtain that
\begin{equation*}
h_{\omega_{\ell}}\|\div(\kappa(\bu_h,p_h)\nabla p_h-\kappa(\bu,p)\nabla p)\|_{0,\omega_{\ell}}
\leq h_{\omega_{\ell}}(\|R_g\|_{0,\omega_{\ell}}+\Big(c_0+\frac{\alpha^2}{\lambda+\mu}\Big)\|p-p_h\|_{1,\omega_{\ell}}+\frac{2\alpha}{(\lambda+\mu)}\|\bsig-\bsig_h\|_{0,\omega_{\ell}}).
\end{equation*}
Thus, \eqref{eq_estimate2_eff} follows from \eqref{eq_lado_int}, the preceding estimate, and \eqref{eq_estimate1_eff}, which completes the proof.
\end{proof}
\begin{lemma}Assume that $\kappa(\bu_h,p_h)$ and $z_\Gamma$ are piecewise polynomial.  Then for all $\ell\in\mathcal{E}_h(\Gamma)$, there holds
\begin{align*}
 h_e^{1/2}\|\left(\kappa(\bu_h,p_h)\nabla p_h\right)\cdot\bn_{\ell}-z_{\Gamma}\|_{0,\ell}\leq  C\left(\|\bsig-\bsig_h\|_{\bdiv,K_{\ell}}+\|\bu-\bu_h\|_{1,K_{\ell}}\right.\\ 
 \left.+\|p-p_h\|_{1,K_{\ell}}+h_k\|g-\mathcal{P}_h g\|_{0,K_{\ell}}\right),
\end{align*}
where $K_{\ell}$ is a element of $\CT_h$ that contains $\ell$ on its boundary.
\end{lemma}
\begin{proof}
The proof follows by observing that $z_\Gamma=\kappa(\bu,p)\nabla p$ on $\Gamma$ and proceeding analogously to the proof of \eqref{eq_estimate2_eff}, by defining $\varUpsilon_h:=\kappa(\bu_h,p_h)\nabla p_h\cdot\n-z_\Gamma$.
\end{proof} 

Consequently, the efficiency follows from the two preceding lemmas and standard arguments that allow us to bound the remaining terms in the estimator. The result is summarized in the following theorem.
\begin{theorem}Assume that $\kappa(\bu_h,p_h)$ and $z_\Gamma$ are piecewise polynomial.  Then there exists $C_{\text{eff}}>0$, such that
$$\Theta\leq C_{\text{eff}}\left(\|(\texttt{e}_{\bsig},\texttt{e}_{\bu},e_{p}),\texttt{e}_{\bgamma})\|_{\mathbf{H}\times\mathbf{Q}}+\sum_{K\in\CT_h}h_K\|g-\mathcal{P}_h g\|_{0,K}\right).$$
\end{theorem}
\begin{remark}
Let us notice that of $\kappa(\cdot,\cdot)$ is not a polynomial term, the analysis is similar to what we have presented, but taking into account that additional terms appear naturally, more precisely, terms associated to the projection of $\kappa(\cdot,\cdot)$ onto polynomials of degree $k\geq 0$, similarly to what occurs with the sources $\boldsymbol{f}$ and $g$.
\end{remark}

\section{Numerical tests}
\label{sec:numerics}
On this section we present a series of numerical tests with the aim of  confirming our theoretical results. All the reported results have been obtained with the open source finite element library \texttt{FEniCS} \cite{AlnaesBlechta2015a}. The nonlinear  systems were solved with Newton--Raphson's method with  a residual tolerance of $10^{-7}$, and the linear systems were solved using the sparse LU factorisation of MUMPS.

\begin{figure}[t!]
\begin{center}
\includegraphics[width=0.24\textwidth]{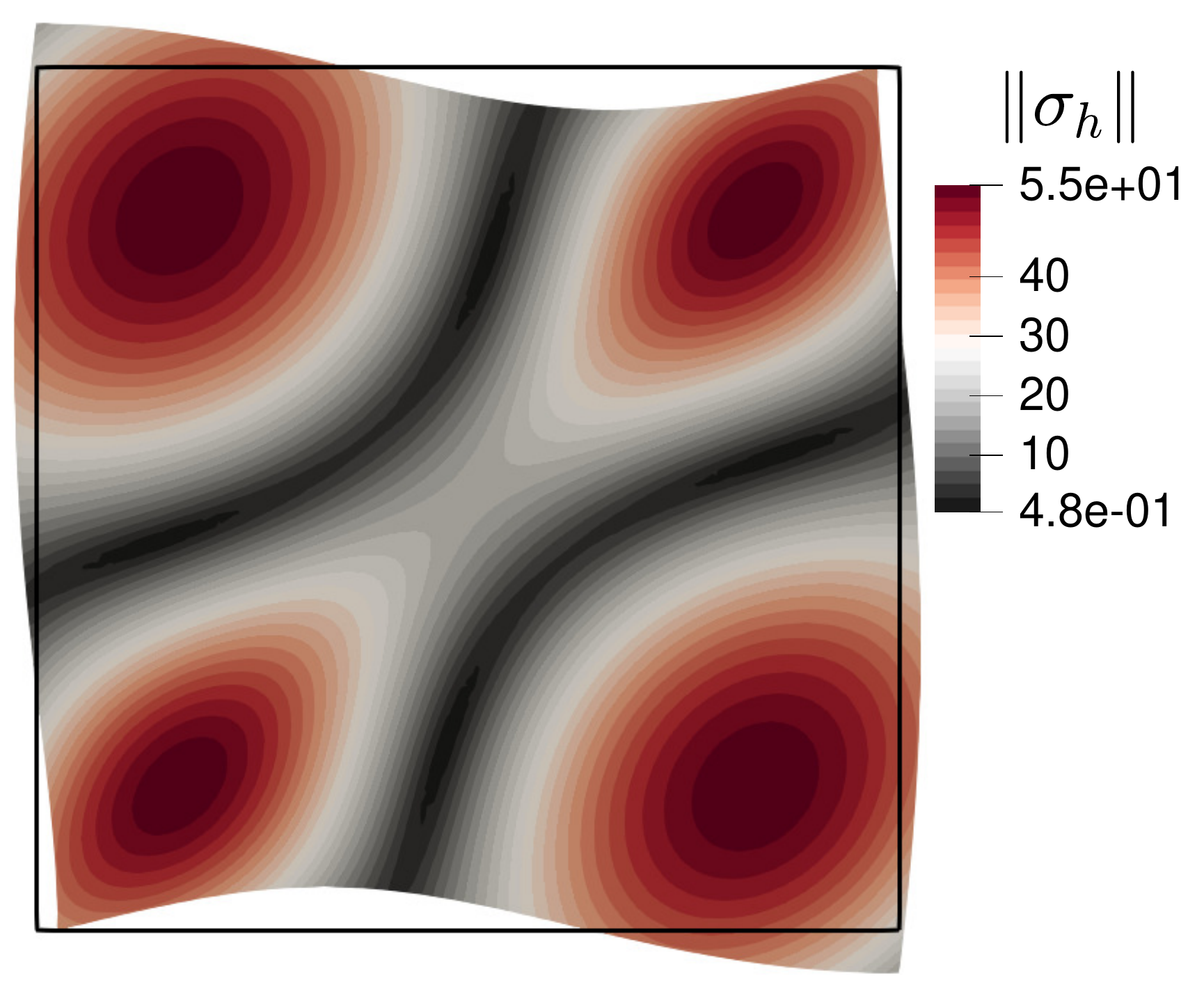}
\includegraphics[width=0.24\textwidth]{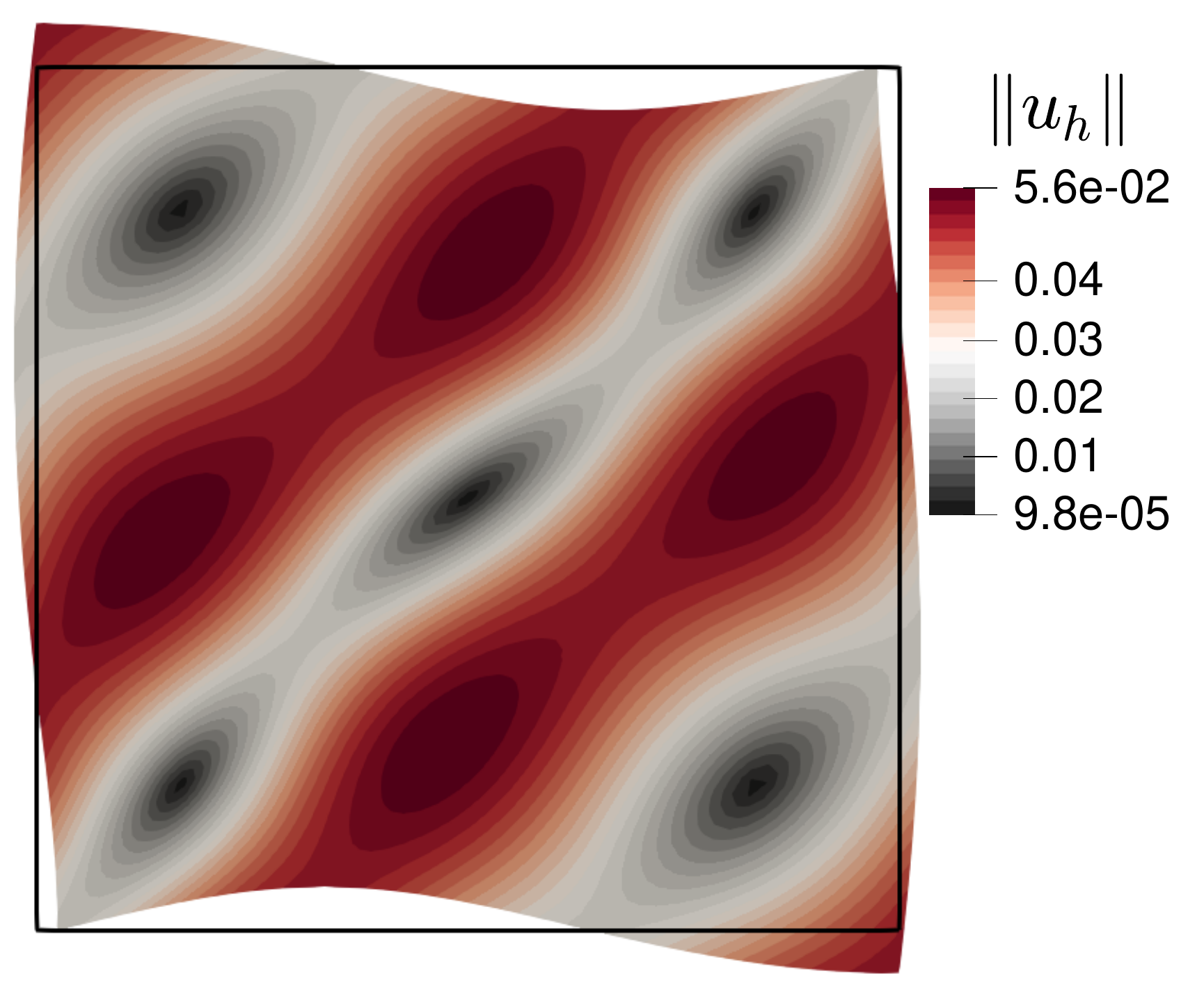}
\includegraphics[width=0.24\textwidth]{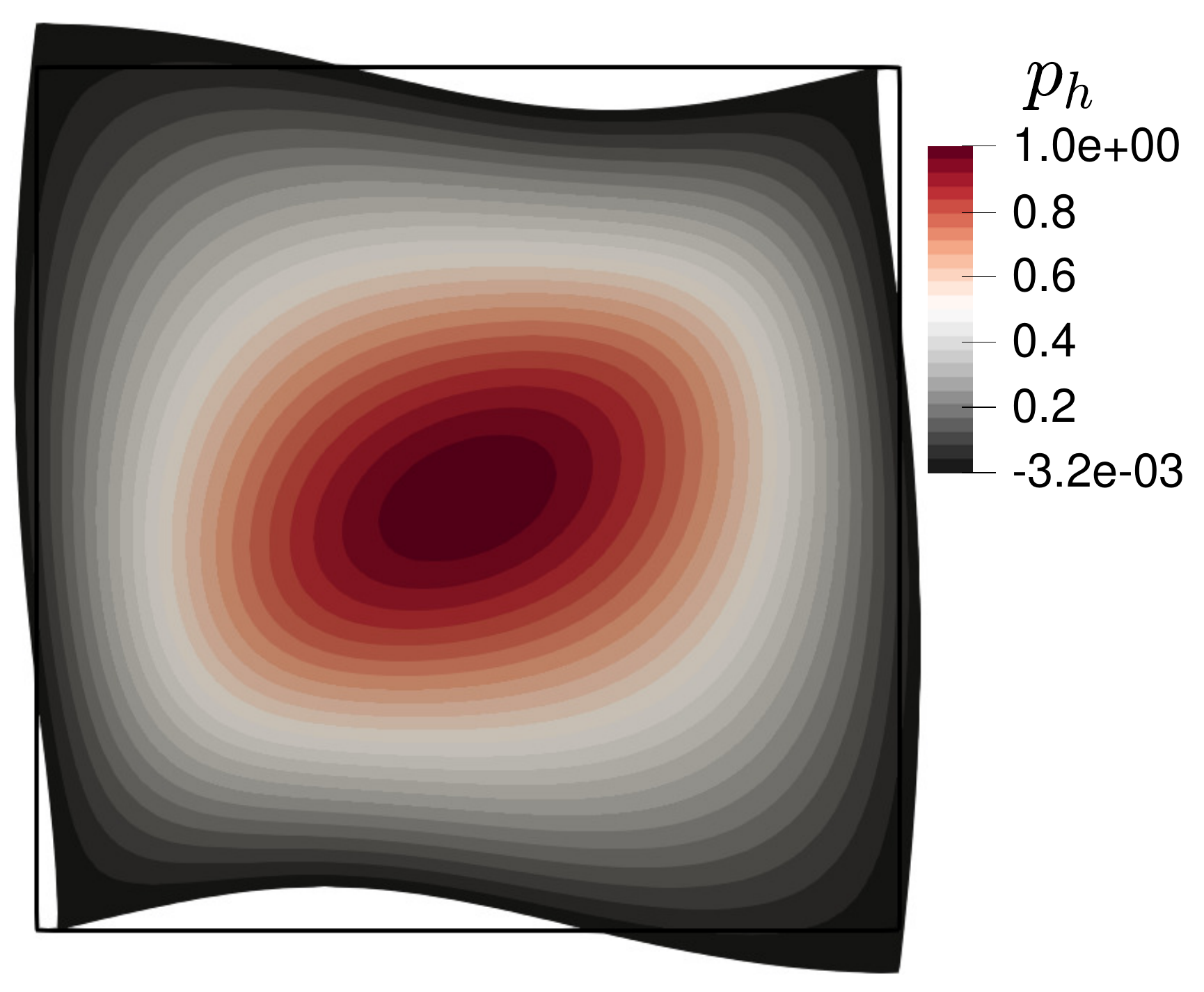}
\includegraphics[width=0.24\textwidth]{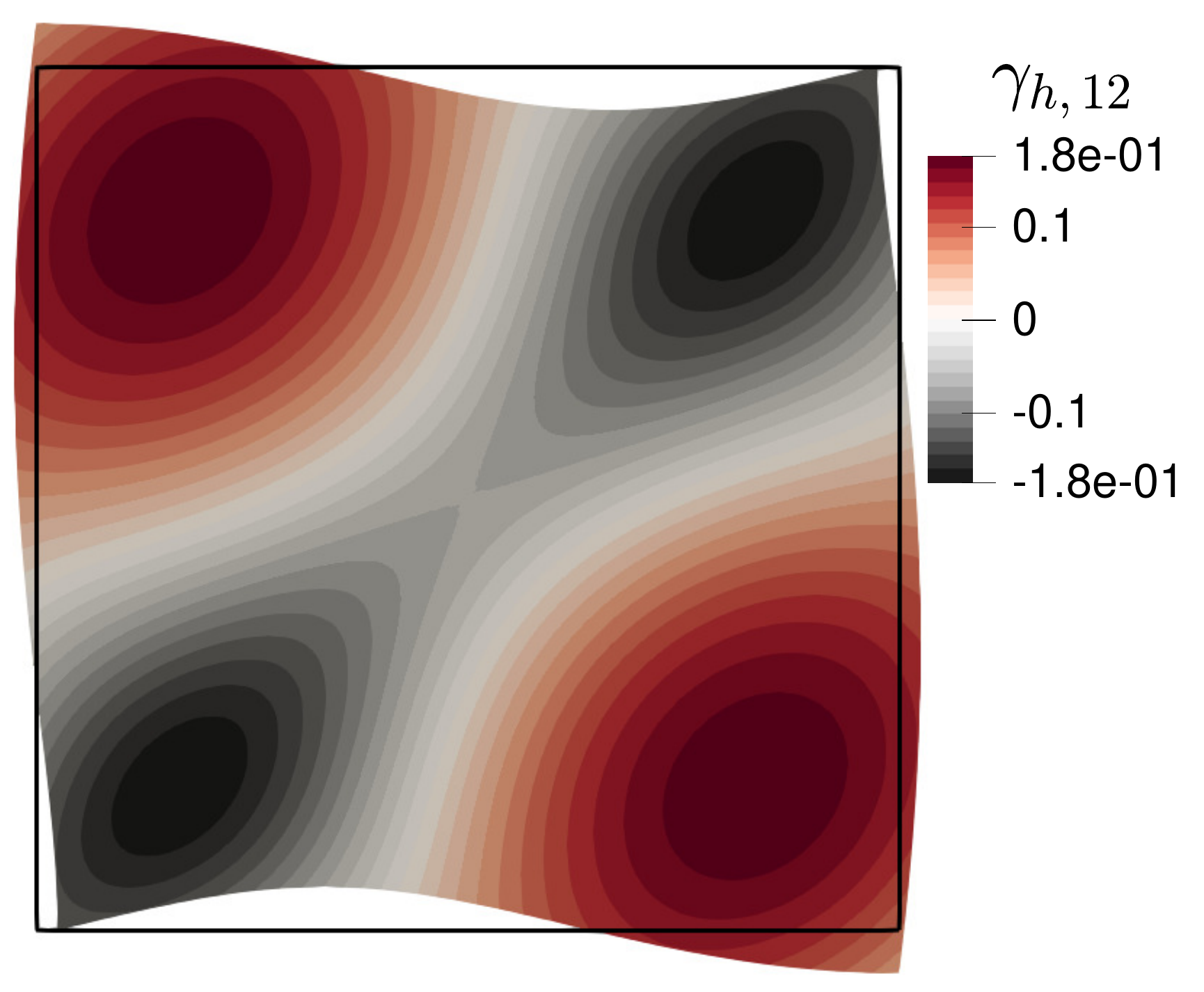}
\end{center}
\caption{Example 1. Approximate solutions (stress magnitude, displacement magnitude, pressure distribution, and principal rotation component, all rendered on the deformed configuration) for the convergence test in 2D, obtained with the second-order scheme.}\label{fig:ex01}
\end{figure}

\begin{table}[t!]
    \centering
    \begin{tabular}{|c|rc|cccccccccc|}
    \hline
$k$ & DoFs & $h$ & $e(\bsig)$ & EOC & $e(\bu)$  & EOC & $e(p)$  & EOC & $e(\bgamma)$  & EOC & it & eff$(\Theta)$\\
\hline 
 \multirow{6}{*}{0}  &  99 & 0.7071 & 8.8e+01 & 0.00 & 6.5e-01 & $\star$ & 2.4e-01 & $\star$ & 1.1e+00 & $\star$ & 4 & 1.28 \\
&   331 & 0.3536 & 4.9e+01 & 0.85 & 3.0e-01 & 1.10 & 7.3e-02 & 1.70 & 7.4e-01 & 0.62 & 4 & 1.28 \\
&  1203 & 0.1768 & 2.5e+01 & 0.97 & 9.6e-02 & 1.66 & 2.7e-02 & 1.41 & 4.1e-01 & 0.85 & 4 & 1.28 \\
&  4579 & 0.0884 & 1.3e+01 & 0.99 & 3.5e-02 & 1.44 & 1.3e-02 & 1.06 & 2.1e-01 & 0.95 & 4 & 1.28 \\
& 17859 & 0.0442 & 6.3e+00 & 1.00 & 1.5e-02 & 1.19 & 6.5e-03 & 1.01 & 1.1e-01 & 0.98 & 3 & 1.28 \\
& 70531 & 0.0221 & 3.1e+00 & 1.00 & 7.4e-03 & 1.06 & 3.3e-03 & 1.00 & 5.5e-02 & 0.99 & 3 & 1.28 \\
\hline
 \multirow{6}{*}{1} &   243 & 0.7071 & 3.1e+01 & $\star$ & 2.2e-01 & $\star$ & 5.9e-02 & $\star$ & 3.8e-01 & $\star$ & 3 & 1.38 \\
&   867 & 0.3536 & 9.0e+00 & 1.79 & 4.0e-02 & 2.44 & 1.2e-02 & 2.26 & 1.2e-01 & 1.70 & 4 & 1.34 \\
&  3267 & 0.1768 & 2.3e+00 & 1.95 & 7.8e-03 & 2.34 & 3.0e-03 & 2.03 & 3.2e-02 & 1.87 &  4 & 1.32 \\
& 12675 & 0.0884 & 5.9e-01 & 1.99 & 1.7e-03 & 2.19 & 7.6e-04 & 1.98 & 9.5e-03 & 1.76 & 3 & 1.31 \\
& 49923 & 0.0442 & 1.5e-01 & 2.00 & 4.1e-04 & 2.07 & 1.9e-04 & 1.99 & 3.2e-03 & 1.87 & 3 & 1.31 \\
&198147 & 0.0221 & 3.7e-02 & 2.00 & 1.0e-04 & 2.02 & 4.8e-05 & 2.00 & 7.8e-04 & 1.92 & 3 & 1.31 \\
\hline
\end{tabular}
\caption{Example 1. Error history (degrees of freedom, mesh size, individual errors and experimental rates of convergence) in 2D for the formulation using the two lowest-order FE families with BDM$_{k+1}$ elements.}
    \label{tab:ex01}
\end{table}

\subsection{Example 1: Error history and effectivity of the estimator for 2D smooth solutions}

To assess the accuracy and optimal convergence properties of the proposed formulation in two dimensions, we conduct a benchmark test using a manufactured smooth solution on the square domain $\Omega := (0,1)^2$. The domain boundary $\partial\Omega$ is partitioned into a Dirichlet segment $\Gamma_{\mathrm{D}}$ (comprising the bottom and left edges) and a Neumann segment $\Gamma_{\mathrm{N}}$ (formed by the top and right edges). 

We prescribe the exact displacement vector $\bu$ and fluid pressure $p$ as
\begin{equation*}
\bu(x,y) = \begin{pmatrix}
\frac{1}{40}\cos\left[\frac{3\pi}{2}(x+y)\right] \\[1ex] 
\frac{1}{20}\sin\left[\frac{3\pi}{2}(x-y)\right]
\end{pmatrix}, \qquad 
p(x,y) = \sin(\pi x) \sin(\pi y).
\end{equation*}
The exact expressions for the auxiliary mixed variables—stress, rotation—are directly derived from these primary fields. The body load $\boldsymbol{f}$, fluid source $g$, along with boundary terms $\mathbf{u}_{\mathrm{D}}, z_{\mathrm{N}}$, and $\boldsymbol{\sigma}_{\mathrm{N}}$, are generated to match these exact solutions. Although the theoretical analysis assumes homogeneous boundary conditions, standard lifting arguments allow straightforward extension to the inhomogeneous data  $\boldsymbol{\sigma}\boldsymbol{n} = \boldsymbol{\sigma}_{\mathrm{N}}$ used here.

The permeability tensor adopts the Kozeny--Carman  constitutive relation 
\[ \boldsymbol{\kappa}(\bu,p) = \frac{k_0}{\mu_f}\mathbf{I} + \frac{k_1}{\mu_f}\frac{(c_0p+\alpha\div\bu)^3}{(1-c_0p-\alpha \div \bu)^2}\mathbf{I},\] 
and we fix the dimensionless physical parameters to the following values:
\begin{equation*}
k_0 =  c_0 = \alpha = 0.1, \quad  k_1 = 0.9, \quad \lambda = 100, \quad \mu = 10, \quad \mu_f = 1.
\end{equation*}
Under this parameter regime, the non-linear variation of the permeability remains mild, allowing the Newton--Raphson solver to achieve convergence within very few iterations.

To perform the empirical convergence analysis, the spatial domain is discretized across six levels of uniform mesh refinement. Approximations and corresponding error norms for each unknown variable are computed in their natural functional spaces. Here and in other tests using uniform mesh refinement, the experimental order of convergence (EOC) between two successive meshes of typical sizes $h$ and $\hat{h}$, yielding errors $e$ and $\hat{e}$, is determined by
\begin{equation*}
\text{EOC} = \frac{\log(e / \hat{e})}{\log(h / \hat{h})}.
\end{equation*}

The resulting error history for the  finite element family described in Section \ref{sec:fem} (setting $k=0$ and $k=1$) is compiled in Table~\ref{tab:ex01}. All primary and mixed variables exhibit the optimal $\mathcal{O}(h^{k+1})$ decay predicted by the theoretical error bounds derived in Section \ref{sec:apriori}.  The table also shows the effectivity index associated with the a posteriori error estimator $\Theta$. The value remains almost constant for all mesh refinements. 
Finally, visual representations of the computed discrete fields are shown in Figure~\ref{fig:ex01} to demonstrate the behavior of the numerical solution.

\subsection{Example 2. Robustness with respect to material parameters in 3D} We continue assessing the experimental convergence rate of the proposed schemes using manufactured smooth solutions, now in 3D. Exact displacement and pressure are given by 
\begin{gather*}
\bu(x,y,z) = \begin{pmatrix}
\frac{3}{100}\cos\left[\frac{3\pi}{2}(x+y+z)\right] \\[1ex] 
\frac{3}{100}\sin\left[\frac{3\pi}{2}(x-y-z)\right]\\[1ex]
\frac{1}{20}\cos\left[\frac{3\pi}{2}(x-y-z)\right]
\end{pmatrix}, \quad 
p(x,y,z) = \sin(\pi x) \cos(\pi y)\sin(\pi z)+\cos(\pi x)\sin(\pi y)\cos(\pi z), \end{gather*}
and we consider four sets of model parameters to check the robustness of the method in mild parametric regimes as well as in the challenging cases of near incompressibility, vanishing storativity, and small permeability. These sets for Lam\'e, Biot--Willis, storativity, and permeability principal modulation coefficients are taken from \cite{botti2021hybrid} (we consider their second fluid compartment), while we use an exponential form for the permeability function with the fixed relation $k_1 = \frac12 k_0$. As we use a primal formulation for the mass balance equation, we do not expect robustness with respect to low-permeable media (Set IV) since the pressure solution operator loses coercivity in this case.

\begin{figure}[t!]
\begin{center}
\includegraphics[width=0.24\textwidth]{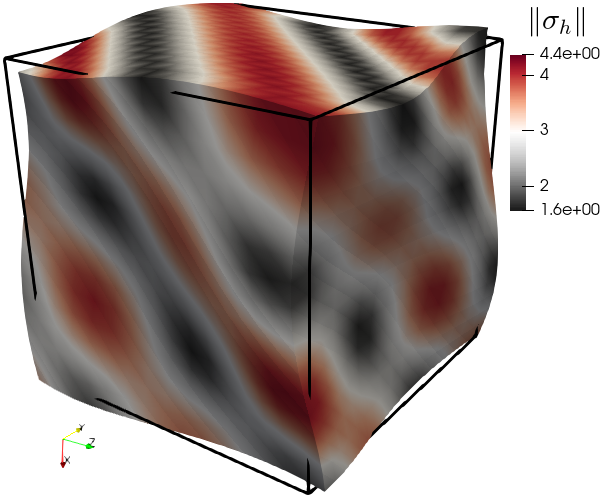}
\includegraphics[width=0.24\textwidth]{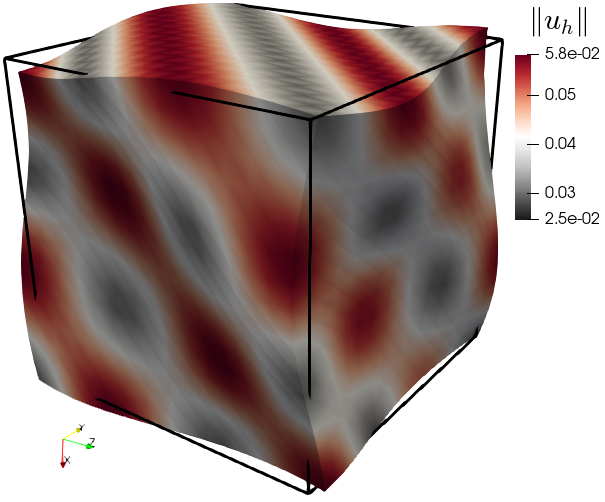}
\includegraphics[width=0.24\textwidth]{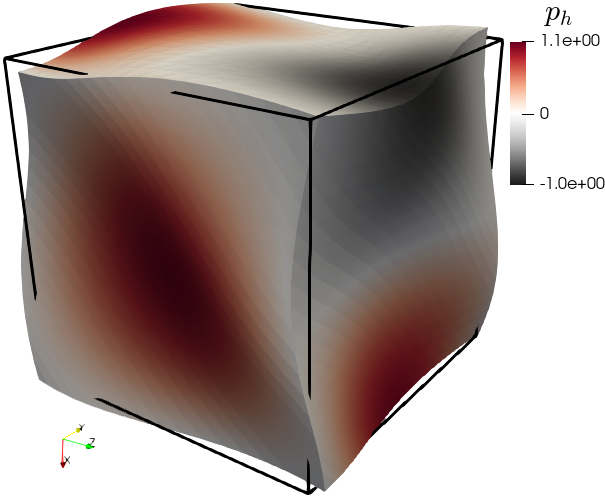}
\includegraphics[width=0.24\textwidth]{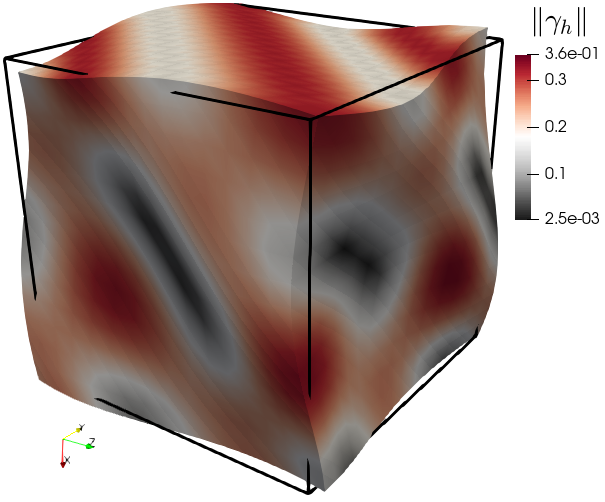}
\end{center}
\caption{Example 2. Approximate solutions (stress magnitude, displacement magnitude, pressure distribution, and rotation magnitude, all rendered on the deformed configuration) for the convergence test in 3D, obtained with the lowest-order scheme and with parameter Set I.}\label{fig:ex02}
\end{figure}

\bigskip
\begin{center}
\begin{tabular}{ccllll}
 \hline
 Parameter & [Unit] & Set I & Set II & Set III & Set IV \\
 \hline
 $\mu$ & MPa & 4.2 & 4.2 & 4.2 & 4.2 \\ 
  $\lambda$ & MPa & 2.4 & 2.4$\cdot10^5$ & 2.4 & 2.4 \\ 
  $\alpha$ &  -- & 0.12 & 0.12 & 0.12 & 0.12 \\
  $c_0$ & MPa$^{-1}$ & 0.014 & 0.014 & 0 & 0.014 \\
  $k_0$ & m$^2$MPa$^{-1}$s$^{-1}$ & 2.75$\cdot10^{-5}$ & 2.75$\cdot10^{-5}$ & 2.75$\cdot10^{-5}$ & $10^{-11}$ \\
  \hline 
\end{tabular}
\end{center}
\bigskip

The error history for the four parameter sets is presented in Table~\ref{tab:ex02}, which confirms that all fields converge optimally in all parameter regimes, except for the pressure in Set IV, which exhibits a slightly suboptimal behaviour. A sample of numerical solutions are portrayed on the deformed domain in Figure~\ref{fig:ex02}. 

\begin{table}[t!]
    \centering
    \begin{tabular}{|rc|ccccccccc|}
    \hline
DoFs & $h$ & $e(\bsig)$ & EOC & $e(\bu)$  & EOC & $e(p)$  & EOC & $e(\bgamma)$  & EOC & eff$(\Theta)$\\
\hline 
\multicolumn{11}{|c|}{Set I}\\
\hline 
 1332 & 0.8660 & 1.5e+01 &  $\star$  & 4.2e-01 &  $\star$  & 1.8e-01 &  $\star$  & 4.1e+00 &  $\star$  & 1.38 \\
  9428 & 0.4330 & 8.5e+00 & 0.79 & 2.7e-01 & 0.61 & 1.1e-01 & 0.69 & 2.4e+00 & 0.78 & 1.38 \\
 70884 & 0.2165 & 4.4e+00 & 0.94 & 1.4e-01 & 0.93 & 6.3e-02 & 0.86 & 1.1e+00 & 1.10 & 1.41 \\
549572 & 0.1083 & 2.2e+00 & 0.99 & 7.1e-02 & 1.02 & 3.2e-02 & 0.99 & 4.4e-01 & 1.35 & 1.48 \\
4327812 & 0.0542 & 1.1e+00 & 0.99 & 3.5e-02 & 1.02 & 1.6e-02 & 1.00 & 2.2e-01 & 1.00 & 1.48 \\
\hline 
\multicolumn{11}{|c|}{Set II}\\
\hline 
 1332 & 0.8660 & 2.9e+05 &  $\star$  & 3.8e+03 &  $\star$  & 2.3e+03 &  $\star$  & 2.9e+00 &  $\star$  & 1.69 \\
  9428 & 0.4330 & 1.7e+05 & 0.81 & 2.0e+03 & 0.95 & 7.2e+02 & 1.66 & 1.8e+00 & 0.72 & 1.70 \\
 70884 & 0.2165 & 8.7e+04 & 0.95 & 5.9e+02 & 1.72 & 1.9e+02 & 1.89 & 7.0e-01 & 1.36 & 1.70 \\
549572 & 0.1083 & 4.4e+04 & 0.99 & 1.6e+02 & 1.92 & 5.1e+01 & 1.94 & 3.4e-01 & 1.03 & 1.71 \\
4327812 & 0.0542 & 2.2e+04 & 0.99 & 6.1e+01 & 1.77 & 1.8e+01 & 1.84 & 1.7e-01 & 1.00 & 1.71 \\
\hline 
\multicolumn{11}{|c|}{Set III}\\
\hline 
1332 & 0.8660 & 1.5e+01 &  $\star$ & 4.2e-01 &  $\star$  & 1.8e-01 &  $\star$  & 1.3e+01 &  $\star$  & 0.93 \\
  9428 & 0.4330 & 8.5e+00 & 0.79 & 2.7e-01 & 0.61 & 1.1e-01 & 0.69 & 7.2e+00 & 0.90 & 0.97 \\
 70884 & 0.2165 & 4.4e+00 & 0.95 & 1.4e-01 & 0.93 & 6.2e-02 & 0.86 & 3.0e+00 & 1.25 & 1.06 \\
549572 & 0.1083 & 2.2e+00 & 0.99 & 7.1e-02 & 1.02 & 3.2e-02 & 0.98 & 9.4e-01 & 1.68 & 1.25 \\
4327812 & 0.0542 & 1.1e+00 & 0.99 & 3.5e-02 & 1.02 & 1.6e-02 & 1.00 & 4.7e-01 & 1.00 & 1.25 \\
\hline 
\multicolumn{11}{|c|}{Set IV}\\
\hline 
 1332 & 0.8660 & 1.5e+01 &  $\star$  & 4.2e-01 &  $\star$  & 1.8e-01 &  $\star$  & 4.9e+00 &  $\star$  & 1.32 \\
  9428 & 0.4330 & 8.5e+00 & 0.79 & 2.7e-01 & 0.61 & 1.1e-01 & 0.69 & 3.2e+00 & 0.60 & 1.29 \\
 70884 & 0.2165 & 4.4e+00 & 0.94 & 1.4e-01 & 0.93 & 6.3e-02 & 0.86 & 1.9e+00 & 0.79 & 1.25 \\
549572 & 0.1083 & 2.2e+00 & 0.99 & 7.1e-02 & 1.02 & 3.2e-02 & 0.99 & 1.1e+00 & 0.79 & 1.20 \\
4327812 & 0.0542 & 1.1e+00 & 0.99 & 3.5e-02 & 1.02 & 1.6e-02 & 1.00 & 5.7e-01 & 0.83 & 1.20 \\
\hline
\end{tabular}
\caption{Example 2. Error history (degrees of freedom, mesh size, individual errors and experimental rates of convergence) in 3D for the formulation using the  lowest-order FE family and different parameter configurations.}
    \label{tab:ex02}
\end{table}

\subsection{Example 3. Adaptive mesh refinement} 
Next we investigate the properties (robustness, reliability, and efficiency) of the a posteriori error estimator when guiding adaptive mesh refinement. We follow the usual approach  of solving, then computing the estimator, marking, and refining. Marking is done as follows \cite{dorfler_sinum96}: a given $K\in \mathcal{T}_h$ is added to the marking set $\mathcal{M}_h\subset\mathcal{T}_h$  whenever the local error indicator $\Theta_K$ satisfies 
\[ \sum_{K \in \mathcal{M}_h} \Theta^2_K \geq \zeta \sum_{K\in\mathcal{T}_h} \Xi_K^2,\]
where $\zeta$ is a user-defined bulk density. All  elements in $\mathcal{M}_h$ are marked for refinement and also some neighbours are marked for the sake of closure. 

\begin{figure}[t!]
\begin{center}
\includegraphics[width=0.4\textwidth]{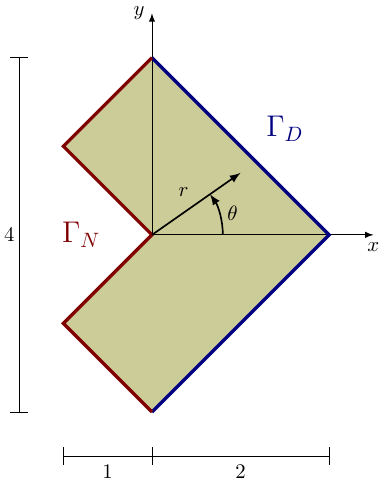}
\end{center}
\caption{Example 3. Sketch of the domain and boundary configuration used for the adaptive mesh refinement tests.}\label{fig:sketch}
\end{figure}

We employ the reentrant-corner-singular manufactured solutions for plane elasticity in mixed form developed in \cite{carstensen2000locking}.  Let us consider the non-convex rotated L-shaped domain sketched in Figure \ref{fig:sketch}, and  
\begin{gather*}
\bu(r,\theta)  =\frac{r^{\chi}}{2\mu} \begin{pmatrix} 
-(\chi + 1) \cos([\chi+1]\theta) + (M_2 - \chi - 1)M_1\cos([\chi-1]\theta)\\
(\chi + 1) \sin([\chi+1]\theta) + (M_2 + \chi - 1)M_1\sin([\chi-1]\theta)
\end{pmatrix}, \quad 
p(r,\theta)  =  r^{\chi}\sin\biggl(\chi(\frac{\pi}{2} + \theta)\biggr),
\end{gather*}
with polar coordinates $r = \sqrt{x_1^2 + x_2^2}$, $\theta= \arctan(x_2,x_1)$, and parameters $\chi \approx 0.54448373$, $M_1 = -\cos([\chi+1]\omega)/\cos([\chi-1]\omega)$, and $M_2 = 2(\lambda + 2\mu)/(\mu + \lambda)$. The boundary conditions (taking as $\Gamma_\rD$ the rightmost segments  and $\Gamma_\rN$ the remainder of the boundary) and forcing data are constructed from these solutions and the model parameters are Young modulus $E = 10^3$, Poisson ratio $\nu = 0.49$,  $k_0 =  \mu_f =1$, $c_0= 0.1$, $k_1 = 0.5$, $\alpha = 0.9$, where for this test we consider again an exponential   permeability constitutive equation. Owing to the regularity of these exact solutions, as in \cite{carstensen2000locking} (see also \cite{KLBRV2026}), one expects that the experimental rate of convergence for the planar elasticity unknowns approaches $\chi$ (here also the pressure should do so since we use a primal formulation). 
The agglomeration parameter is chosen as $\zeta = 10^{-3}$, and we perform steps of uniform mesh refinement as well as adaptive mesh refinement. A comparison of the performance (error history for each individual variable, Newton iteration count) and effectivity index of the a posteriori error estimator are collected in Figure~\ref{fig:ex03-error}. We can observe the expected suboptimal convergence attained under uniform mesh refinement for all the unknowns (and the stress error dominating). Using adaptive mesh refinement we see that the initial two refinements are similar to the uniform case but immediately after that, one restores an optimal convergence (in this case, around quadratic). Also, from the left figure we see that the errors produced by the adaptive mesh refinement are   several orders of magnitude smaller than the uniform refinement errors, and using a fraction of the degrees of freedom. The plot on the right panel indicates that the effectivity index is always bounded near 2.08 for both adaptive and uniform mesh refinement, and we also see that the parametric regime used here requires at most three Newton iterations for convergence. Samples of adaptively refined meshes and approximate solutions are shown in Figure~\ref{fig:ex03-sols}.

\begin{figure}[t!]
\begin{center}
\includegraphics[width=0.98\textwidth]{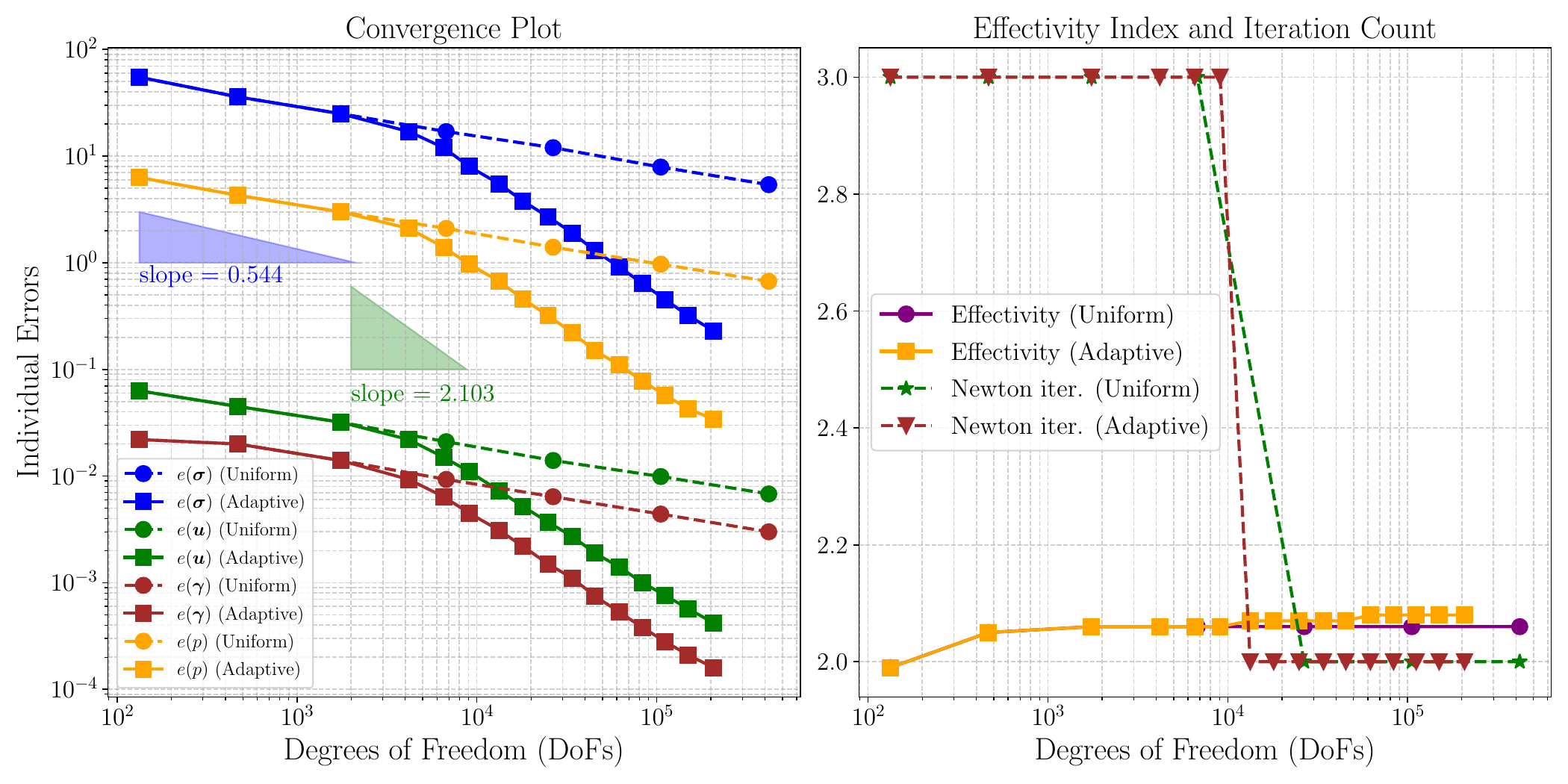}
\end{center}
\caption{Example 3. Performance of the finite element discretization using uniform versus adaptive mesh refinement for the rotated L-shaped domain with non-smooth solutions.}
\label{fig:ex03-error}
\end{figure}

\begin{figure}[t!]
\begin{center}
\includegraphics[width=0.325\textwidth]{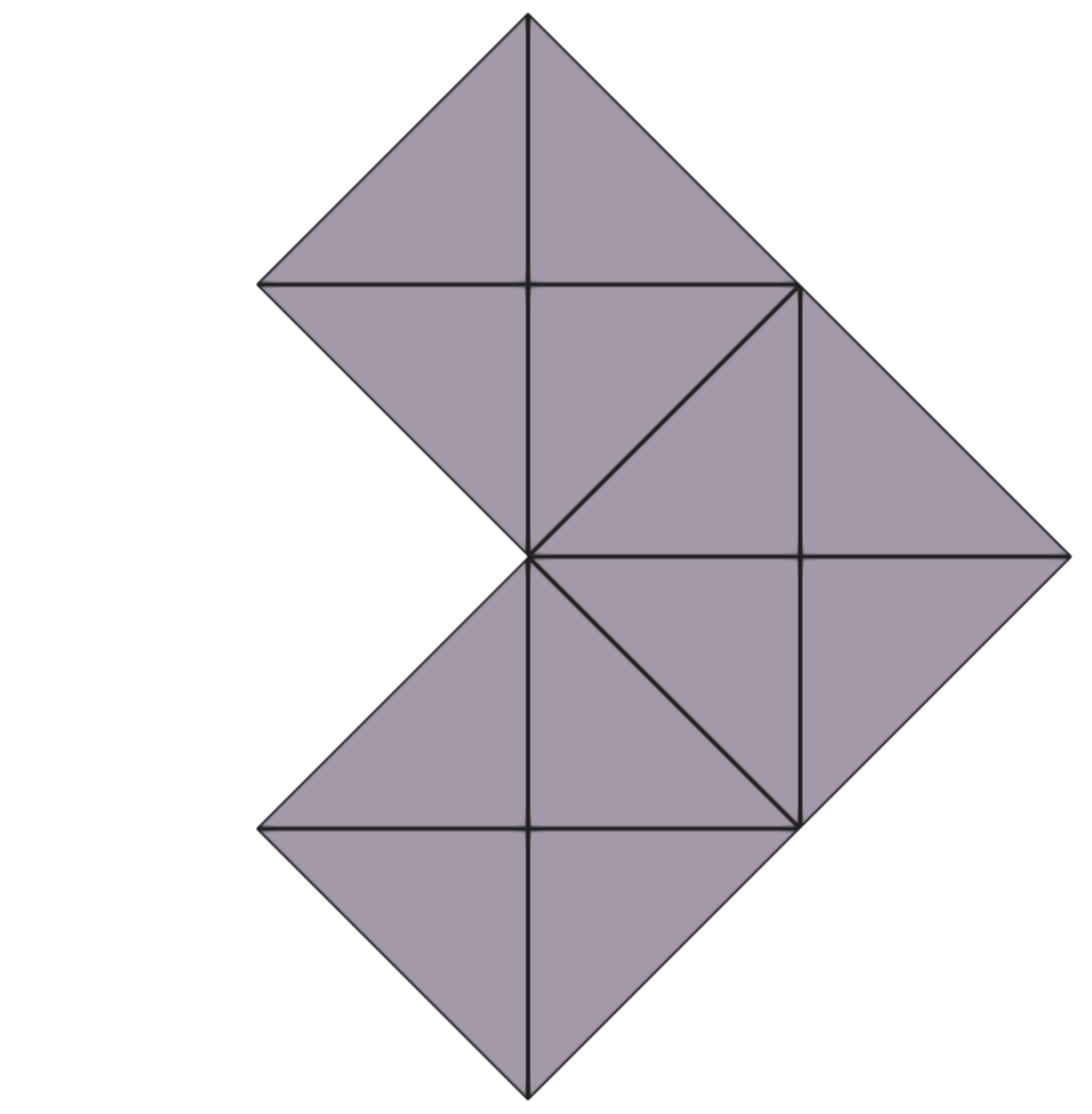}
\includegraphics[width=0.325\textwidth]{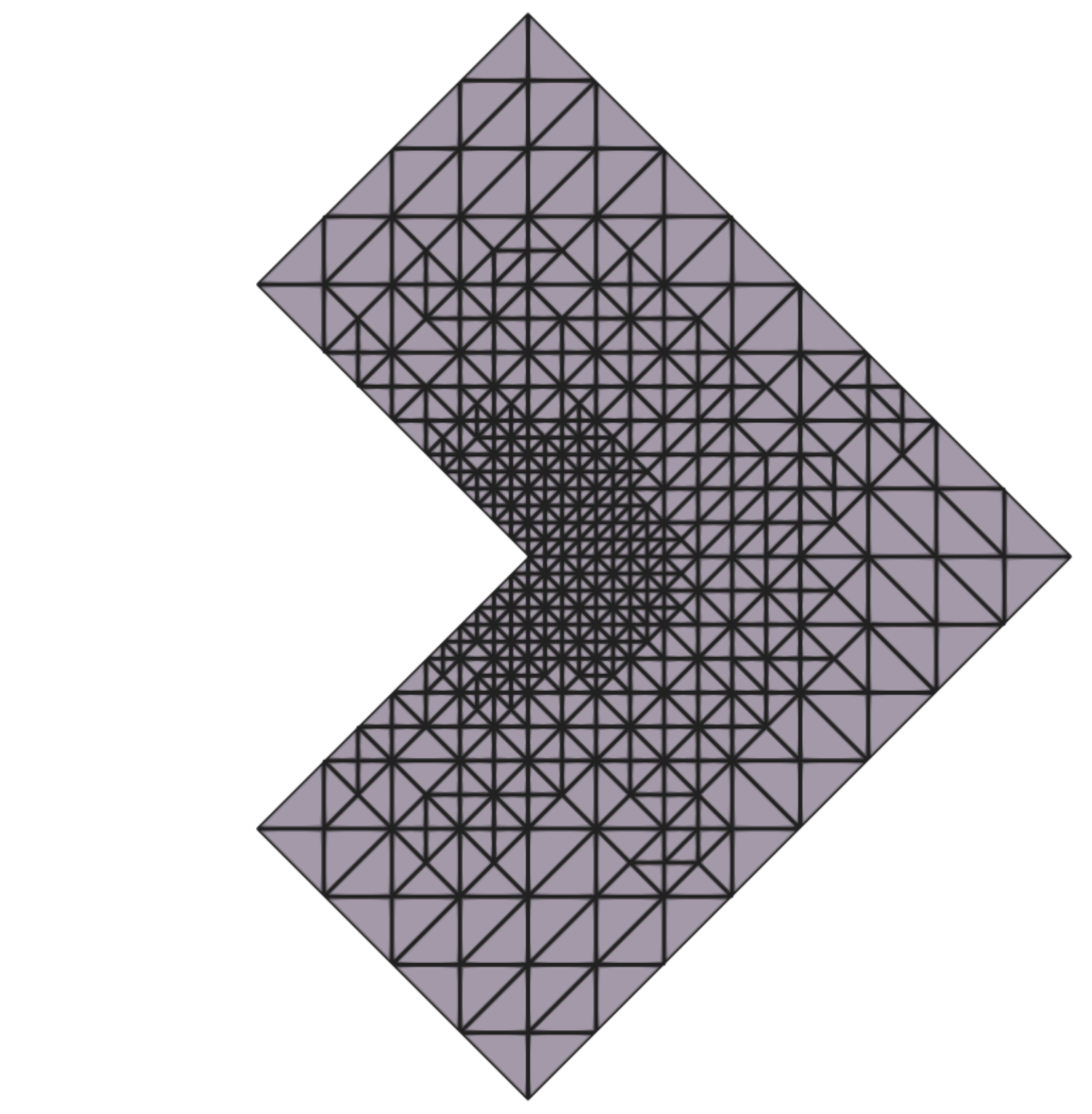}
\includegraphics[width=0.325\textwidth]{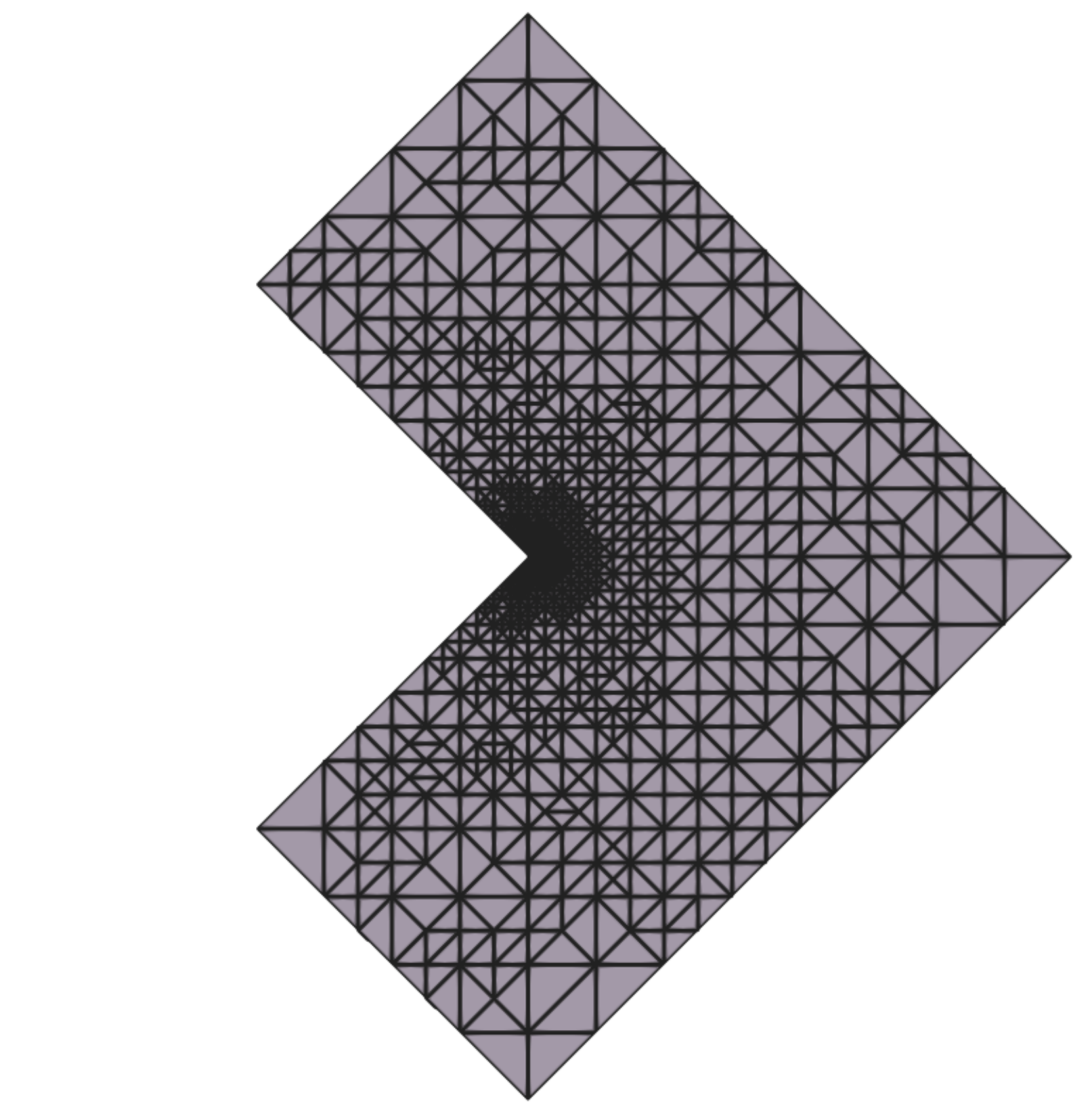}\\
\includegraphics[width=0.325\textwidth]{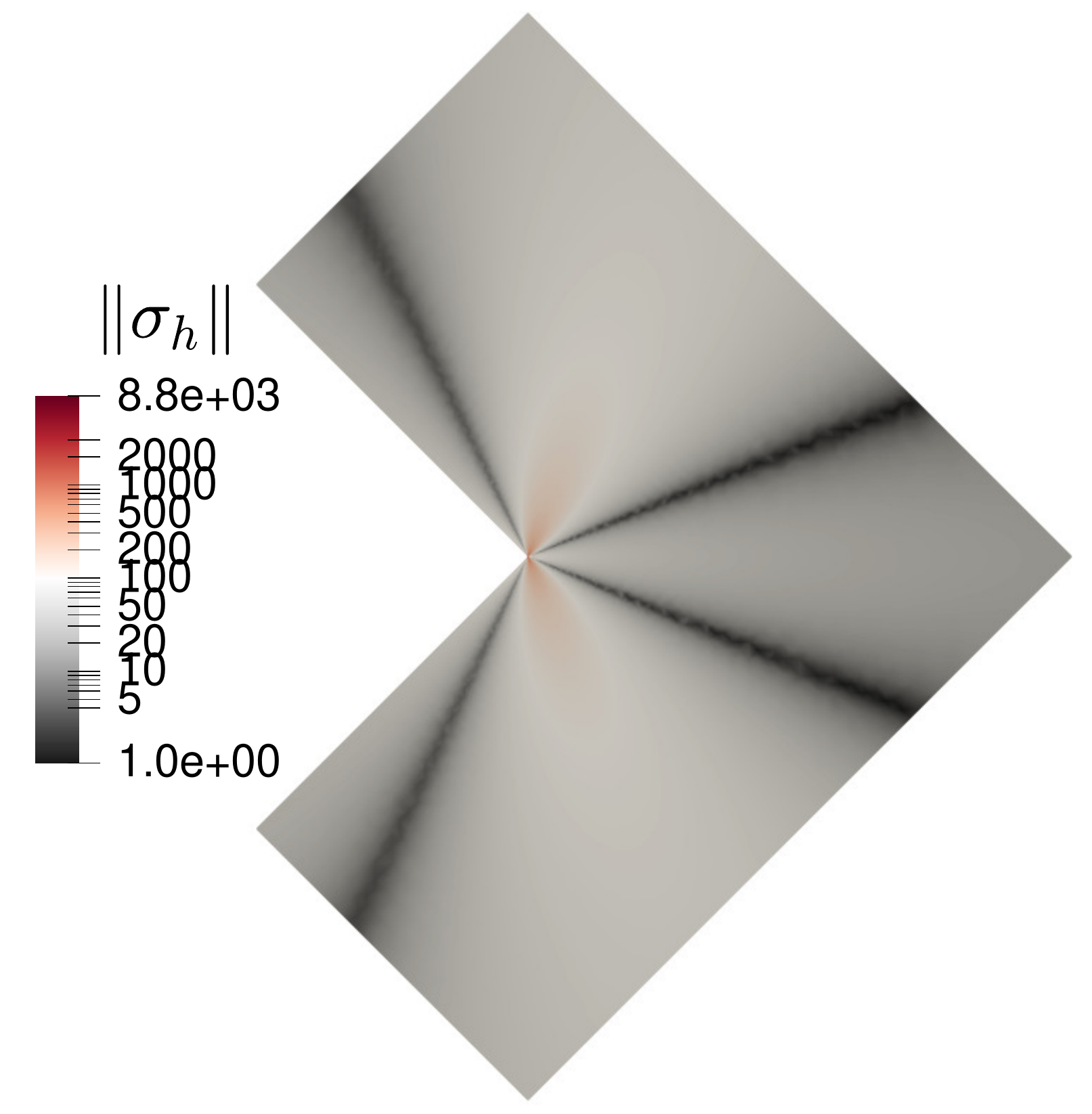}
\includegraphics[width=0.325\textwidth]{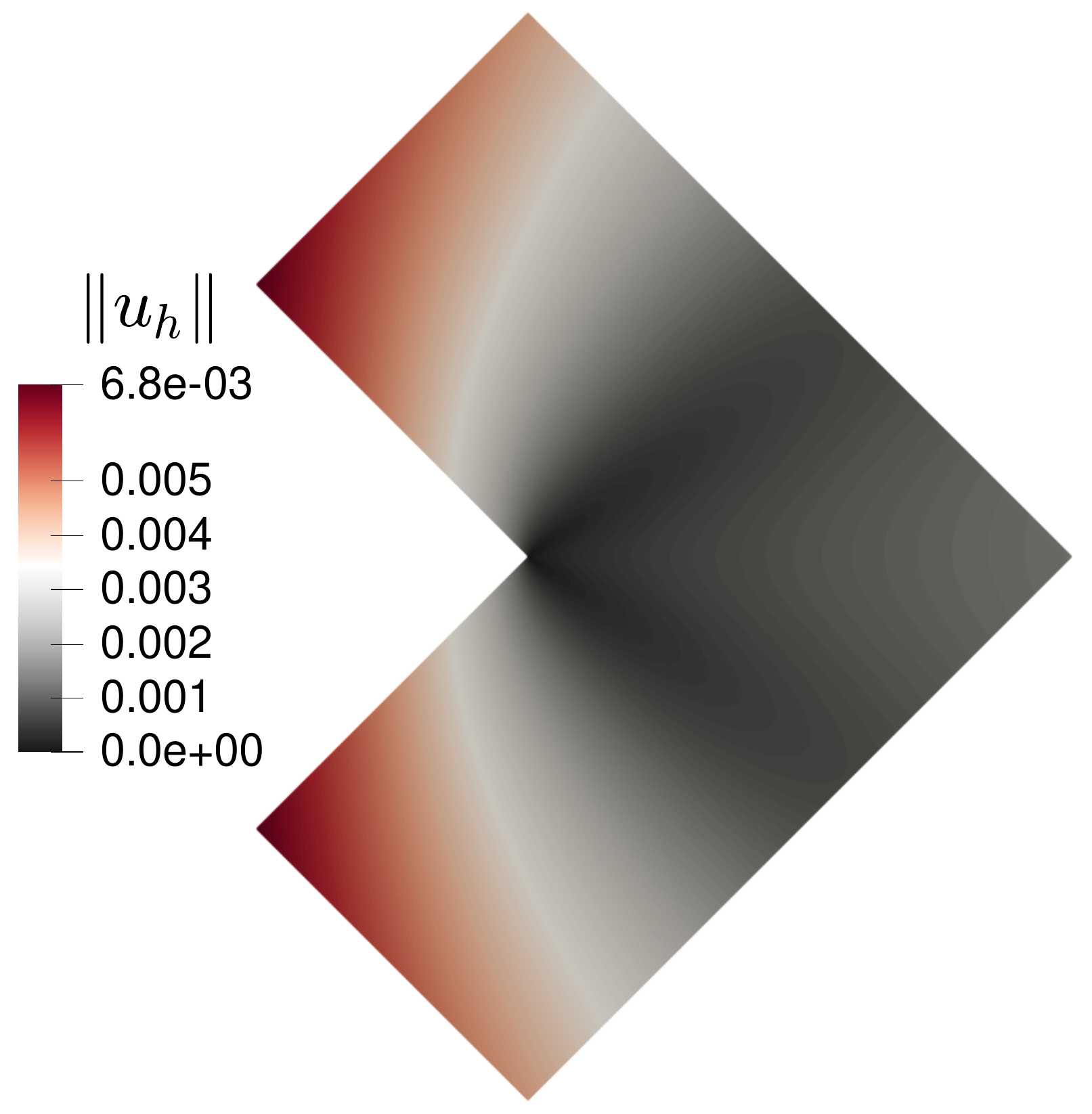}
\includegraphics[width=0.325\textwidth]{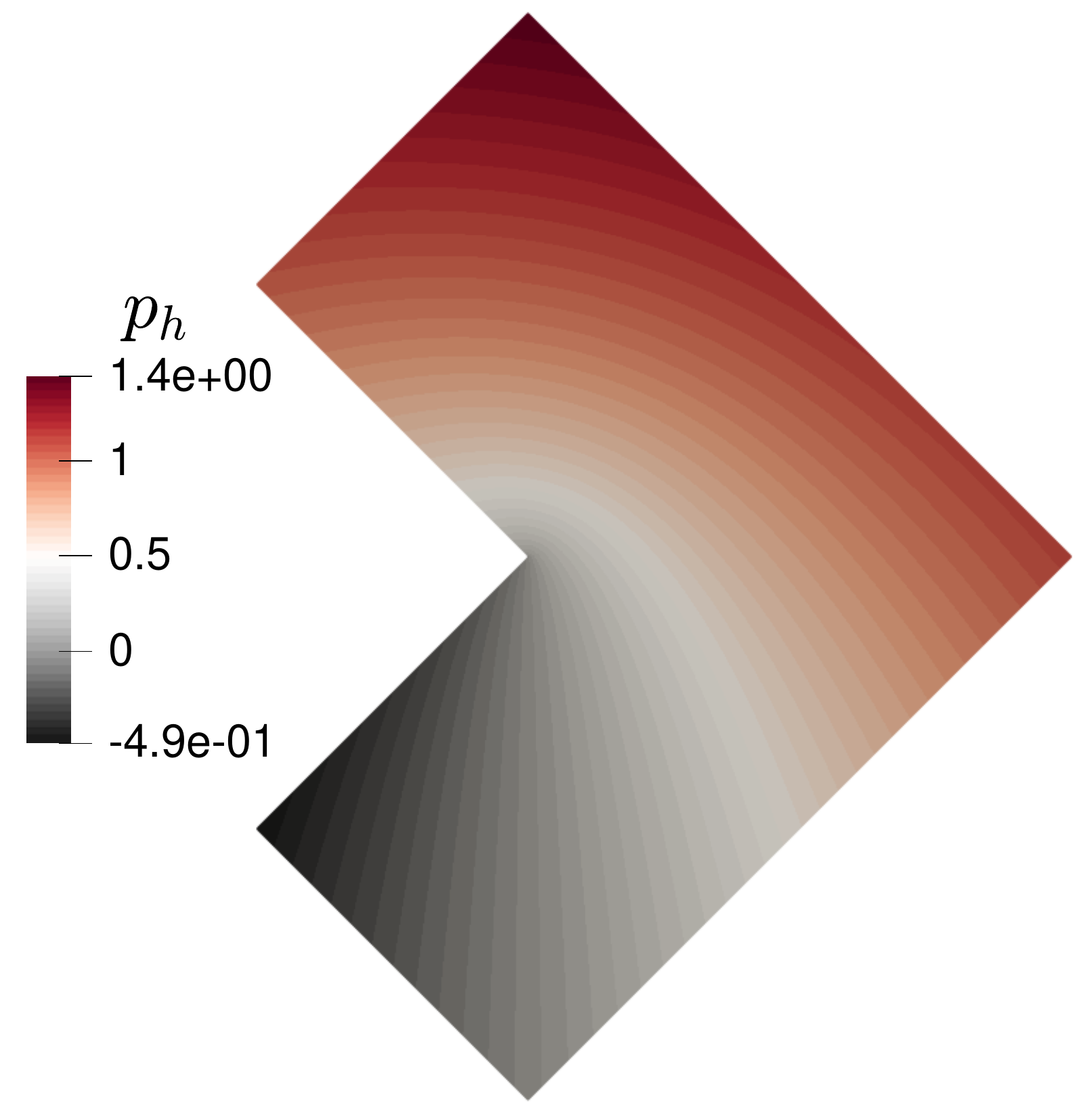}

\end{center}
\caption{Example 3. Adaptively refined meshes and sample of approximate stress magnitude, displacement magnitude, and pressure distribution plotted on the reference domain.}
\label{fig:ex03-sols}
\end{figure}

\subsection{Example 4. Comparison against non-augmented strategy}
We close this section with a simple test using Cook's membrane benchmark to illustrate a possible scenario where augmentation of the displacement space (providing $\bH^1(\Omega)$ conformity) can be beneficial. The objective is only to evaluate the influence of load boundary condition imposition, so we focus on the case of constant permeability, take the Biot--Willis coefficient $\alpha = 0$, and compare the proposed augmented scheme with a non-augmented version (where the displacement is sought in a discontinuous polynomial space and the clamped boundary condition is enforced only weakly).

The domain is the tapered panel of unit thickness with vertices $(0,0)$, $(48,44)$, $(48,60)$ and $(0,44)$, sketched in Figure~\ref{fig:ex04-sketch}. We take as $\Gamma_\rD$ the clamped left edge $\{x_1 = 0\}$ and as $\Gamma_\rN$ the remaining three sides, on which we prescribe
\[
\bsig\bn = \bsig_\rN := \Bigl(0,\tfrac{100}{16}\Bigr)^{\texttt{t}}
\quad\text{on }\{x_1 = 48\},
\qquad
\bsig\bn = \bzero \quad\text{on the two slanted edges},
\]
so that the panel is subject to a total transverse load of $100$ distributed over an edge of length $16$. As in Example~1, the non-homogeneous traction datum is incorporated by a standard lifting argument. The remaining data vanish, $\bF = \bzero$, $g = 0$ and $z_\Gamma = 0$.

All the meshes in the family are images of uniform triangulations of the unit square $\wh\O := (0,1)^2$ under the bilinear map
\begin{equation}\label{eq:cook-map}
(\wh x_1,\wh x_2)\ \longmapsto\ \bigl(48\,\wh x_1,\ 44(\wh x_1+\wh x_2) - 28\,\wh x_1\wh x_2\bigr),
\end{equation}
where $\wh\O$ is first partitioned into $2n\times n$ congruent rectangles, each of which is split into two triangles by its left-leaning diagonal. We take $n\in\{10,50,100,150,200,250\}$, which yields $4n^2$ elements and mesh size $h = 44/n$; the map \eqref{eq:cook-map} is affine along each edge of $\wh\O$, so the tapered geometry is reproduced exactly.

The elastic parameters are chosen in the nearly incompressible regime, namely $E = 250$ and $\nu = 0.4999$, that is
\[
\mu = 83.3389,\qquad \lambda = 4.16611\cdot10^{5},\qquad \lambda/\mu \approx 4999,
\]
while the remaining model coefficients are the constant permeability $\kappa \equiv 10^{-2}$ and the storativity $c_0 = 1$. The stabilisation parameters are taken exactly as prescribed by the proof of Lemma~\ref{lmm:coerc_H}, that is $\delta_1 = \frac32\mu$ and $\delta_2 = 3/(C_K\mu)$, with $C_K$ the Korn constant in \eqref{eq:poincare-Korn} and $C_K = \sqrt{2}$. Let us stress that, since $\alpha = 0$ and $g = 0$, $z_\Gamma = 0$, the mass balance equation \eqref{eq:mass-1} decouples from the momentum and constitutive equations and yields $p \equiv 0$. The resulting problem is therefore linear (Newton's method converges in a single iteration) and the test isolates the behaviour of the poroelastic stress--displacement--rotation block of the formulation in the nearly incompressible limit, which is precisely the regime in which the augmentation is expected to pay off.

\begin{figure}[t!]
\begin{center}
\includegraphics[width=0.7\textwidth]{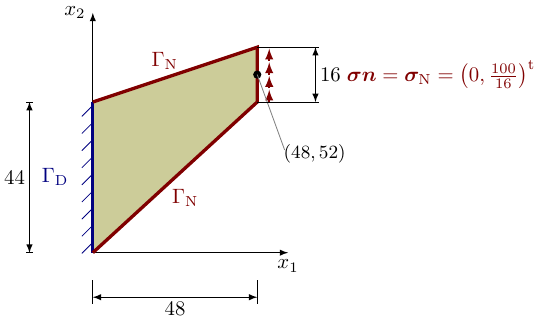}
\end{center}
\caption{Example 4. Sketch of the domain, boundary configuration and applied load for Cook's membrane benchmark. The deflection is monitored at the midpoint $(48,52)$ of the loaded edge.}\label{fig:ex04-sketch}
\end{figure}

We compare three discretizations of the panel problem:
\begin{itemize}
\item[(i)] the \emph{augmented} scheme \eqref{eq:weak-problem-fem} with $k=0$, that is $\bsig_h\in\bbH^{\bsig}_h$ of Brezzi--Douglas--Marini type $\mathbf{BDM}_1$, $\bu_h\in\bH^{\bu}_h$ of continuous $\textbf{\textrm{P}}_1$ type with $\bu_h = \bzero$ imposed strongly on $\Gamma_\rD$, $p_h\in\rH^{p}_h$ of continuous $\textrm{P}_1$ type, and $\bgamma_h\in\bbH^{\bgamma}_h$ piecewise constant and skew-symmetric;
\item[(ii)] the \emph{non-augmented} counterpart, obtained by setting $\delta_1=\delta_2=0$ and replacing $\bH^{\bu}_h$ by the discontinuous space $\{\bv_h\in\bL^2(\O):\ \bv_h|_K\in\textbf{\textrm{P}}_0(K)\ \forall K\in\CT_h\}$, in which case the clamping condition is only imposed weakly through the term $\langle\btau_h\bn,\bzero\rangle_{\Gamma_\rD}$;
\item[(iii)] the three-field displacement--total pressure--fluid pressure discretization $\textbf{\textrm{P}}_2-\textrm{P}_1-\textrm{P}_1$ of \cite{MR3552204,lee2017parameter}.
\end{itemize}
The last one serves as an independent reference: it is locking-free uniformly in $\lambda$ and its displacement is approximated with quadratic elements, so it provides an accurate value of the quantity of interest per degree of freedom.

Following common practice for this benchmark, we monitor the vertical deflection $u_{2,h}(48,52)$ at the midpoint of the loaded edge. The three-field scheme converges to this value monotonically and with experimental rate $\approx 1.06$ in $h$, and its Richardson extrapolation gives the reference value
\begin{equation}\label{eq:cook-ref}
u_2^{\mathrm{ref}}(48,52) = 7.4048\ \pm\ 2\cdot10^{-4},
\end{equation}
which we use to measure the errors $e(u_2) := |u_{2,h}(48,52) - u_2^{\mathrm{ref}}|$ reported below. Note that the deflection converges with a rate close to one rather than at the rates predicted in Section~\ref{sec:apriori}; this is consistent with the stress singularities generated at the two clamped corners $(0,0)$ and $(0,44)$, similarly as the non-smooth solutions of Example~3.

Samples of approximate solutions are displayed in Figure~\ref{fig:ex04-sols}, and a more quantitative study is  displayed in Figure~\ref{fig:ex04}. None of the three schemes exhibits volumetric locking: all of them approach \eqref{eq:cook-ref} as the mesh is refined (which would not be the case for a primal discretization with piecewise linear displacements at $\nu = 0.4999$). More relevant for the purpose of this test is the comparison between the augmented and the non-augmented schemes. The augmented formulation is between four and eleven times more accurate on every mesh of the family, and it achieves this with fewer degrees of freedom, since the continuous $\textbf{\textrm{P}}_1$ displacement space is smaller than the discontinuous $\textbf{\textrm{P}}_0$ one on this mesh family. The augmented deflection approaches \eqref{eq:cook-ref} monotonically from below, and on the finest mesh  it matches the three-field Taylor--Hood deflection to within $2\cdot10^{-4}$, but at the price of roughly $1.7$ times as many unknowns. For the non-augmented scheme, the error stagnates around $10^{-2}$ and does not decrease   between the last two refinements. These results support the claim that $\bH^1(\Omega)$-conformity of the discrete displacement can be beneficial.

\begin{figure}[t]
    \centering
    \includegraphics[width=0.325\textwidth]{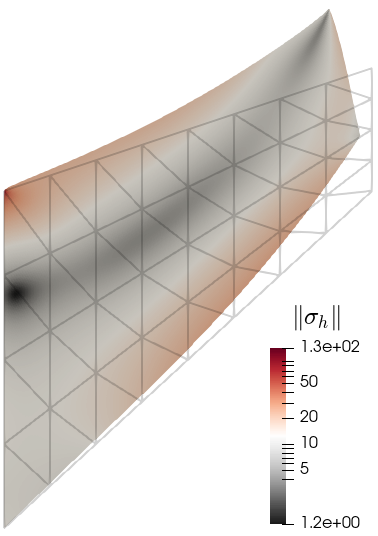}
    \includegraphics[width=0.325\textwidth]{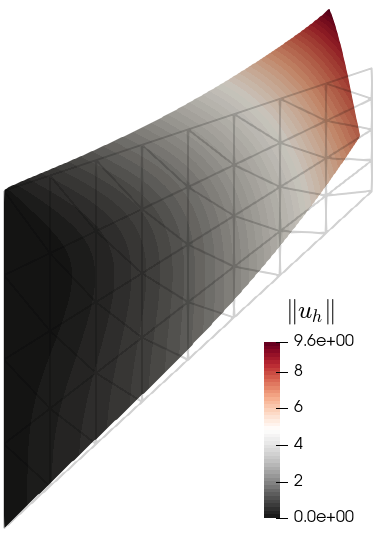}
    \includegraphics[width=0.325\textwidth]{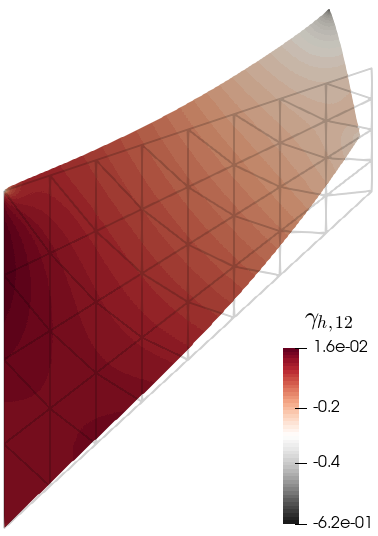}
\caption{Example 4. Approximate stress magnitude, displacement magnitude, and relevant rotation component computed with the proposed augmented scheme, with lowest-order polynomial degree, and on a coarse mesh.}\label{fig:ex04-sols}
\end{figure}

\begin{figure}[t!]
\begin{center}
\includegraphics[width=0.98\textwidth]{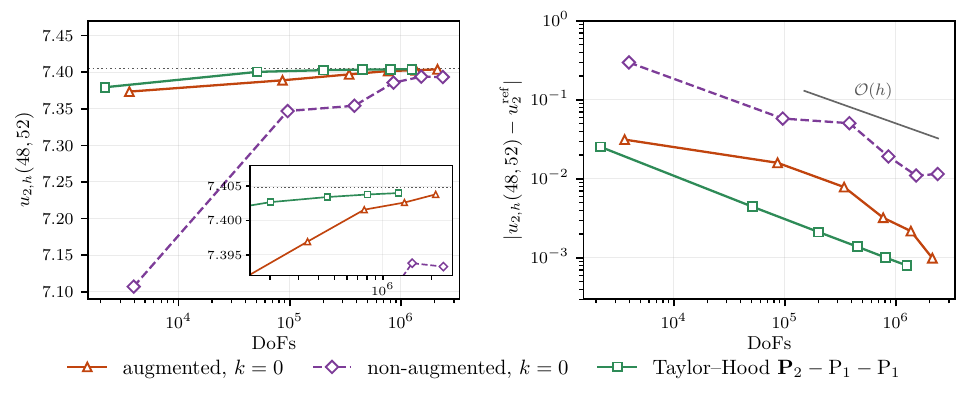}
\end{center}
\caption{Example 4. Computed deflection $u_{2,h}(48,52)$ as a function of the number of degrees of freedom, with a zoom on the asymptotic range (left, the dotted line indicates the extrapolated reference $u_2^{\mathrm{ref}} = 7.4048$), and the corresponding error (right).}
\label{fig:ex04}
\end{figure}

\section{Conclusions}\label{sec:concl}

We have {proposed and analyzed}  an augmented mixed finite element method for a nonlinear Biot poroelasticity system in which the intrinsic permeability $\kappa(\bu,p)$ depends explicitly on the solid displacement and the pore fluid pressure. The augmentation strategy allows to seek $\bu$ directly in the classical energy space $\bH^1(\Omega)$, which is the natural setting in which to evaluate the nonlinear permeability, while also permitting a strong imposition of the essential displacement boundary condition.  In contrast with the recent works \cite{lamichhane24,KLBRV2026}, where the coupled problem led to a twofold saddle-point structure or required decoupling into separate elasticity and diffusion sub-problems, the present strategy yields a single saddle-point formulation  after linearising the permeability, so that well-posedness follows directly from the classical Babu\v{s}ka--Brezzi theory. Existence and uniqueness of both the continuous and discrete solutions {were} established through a Banach fixed-point argument under a smallness assumption on the data. We also derived a C\'ea estimate yielding optimal convergence rates, and designed a residual-based a posteriori error estimator, shown to be reliable and efficient. The numerical tests reported confirm these theoretical properties.
Particularly, we have confirmed some advantages of using augmentation and $\mathbf{H}^1(\Omega)$-conformity of the displacement in Cook's membrane benchmark in the nearly incompressible regime. Even if both augmented and non-augmented schemes are locking-free, the augmented formulation is substantially more efficient as it reduces the deflection error by a factor of 4 to 11 compared to the non-augmented version, and using fewer degrees of freedom.

Possible extensions of this study include the time-dependent Biot system, the treatment of anisotropic permeability laws  and the analysis of the augmented formulation for poromechanical models coupled with additional physics such as reactive or thermally driven flow-transport processes.

\bigskip 
\noindent\textbf{Funding. } DICREA through
	Proyecto Regular RE2514703 Universidad del B\'io-B\'io, Chile (FL). 
	National Research and Development Agency (ANID) of the Ministry of
	Science, Technology, Knowledge and Innovation of Chile through FONDECYT project 1231619 (GR) and through the \textit{Concurso de Subvenci\'on a la Instalaci\'on en la Academia--convocatoria 2025}, project number 85250057 (SVF).
	Australian Research Council through the Future Fellowship grant FT220100496 (RRB). 

\bibliographystyle{siam}
\bibliography{mbibliography}
\end{document}